\documentclass{amsart}
\usepackage{xr}
\usepackage{amsmath, amssymb, amsthm, amsbsy, a4wide}
\usepackage{bbm}
\usepackage{verbatim}
\usepackage{MnSymbol}
\usepackage{mathrsfs}
\usepackage{enumerate}
\usepackage[all]{xy}
\usepackage[pdftex]{color}
\usepackage{tikz}
\newtheorem{thm}{Theorem}[subsection]
\newtheorem{pro}[thm]{Proposition}
\newtheorem{lm}[thm]{Lemma}
\newtheorem{cor}[thm]{Corollary}
\theoremstyle{definition}
\newtheorem{df}[thm]{Definition}
\newtheorem{remark}[thm]{Remark}

\newtheorem{notn}[thm]{Notation}

\let\S\undefined

\DeclareMathOperator\Aut{Aut}

\def\S{\mathrm{S}}

\def\m{\mathfrak{m}}

\def\p{\mathfrak{p}}

\def\Z{\mathrm{Z}}
\def\GL{\mathrm{GL}}

\def\Sym{\text{Sym}}

\DeclareMathOperator{\Gal}{Gal}
\DeclareMathOperator{\spec}{Spec}

\DeclareMathOperator{\Hom}{Hom}
\DeclareMathOperator{\ext}{ext}
\DeclareMathOperator{\val}{val}

\def\T{\mathrm{T}}
\def\R{\mathrm{R}}
\def\X{\mathrm{X}}
\def\Y{\mathrm{Y}}
\def\W{\mathrm{W}}

\def\CC{\mathbb{C}}

\def\ZZ{\mathbb{Z}}
\def\GG{\mathbb{G}}

\def\mcT{\mathcal T}

\DeclareMathOperator{\Res}{Res}

\DeclareMathOperator{\Tr}{Tr}

\def\sep{\mathrm{sep}}

\def\mO{\mathfrak O}

\newcommand{\Del}{\mathrm{Del}}

\newcommand{\ft}{\mathrm{ft}}
\newcommand{\lleq}{\mathrel{\underline{\ll}}}
\newcommand{\ggeq}{\mathrel{\underline{\gg}}}

\DeclareMathOperator{\std}{std}
\DeclareMathOperator{\naive}{naive}
\DeclareMathOperator{\Ind}{Ind}

\title[On congruent isomorphisms for tori]{On congruent isomorphisms for tori}

\author{Anne-Marie Aubert}
\address{Sorbonne Universit\'e and Universit\'e Paris Cit\'e, CNRS,
IMJ-PRG, F-75005 Paris, France}
\email{anne-marie.aubert@imj-prg.fr}

\author{Sandeep Varma}
\address{School of Mathematics, Tata Institute of Fundamental Research, Homi Bhabha Road, Colaba, Mumbai, India}
\email{sandeepvarmav@gmail.com}

\date{\today}

\begin{document}
\maketitle

\begin{abstract} 
Let $F$ and $F'$ be two $l$-close nonarchimedean local fields, where $l$ is a positive integer, and let $\T$ and $\T'$ be two tori 
over $F$ and $F'$, respectively, such that their cocharacter lattices can be identified as modules over the ``at most $l$-ramified'' absolute 
Galois group $\Gamma_F/I_F^l \cong\Gamma_{F'}/I_{F'}^l$. In the spirit of the work of Kazhdan and Ganapathy, for  every positive integer $m$ relative to which $l$ is large, we construct a congruent
isomorphism $\T(F)/\T(F)_m \cong \T'(F')/\T'(F')_m$, where $\T(F)_m$ and  $\T(F')_m$ are the minimal congruent filtration subgroups 
of $\T(F)$ and  $\T(F')$, respectively, defined by J.-K.~Yu. We prove that this isomorphism is functorial and compatible with 
both the isomorphism constructed by Chai and Yu and the Kottwitz homomorphism for tori.  We show that, when $l$ is even larger relative 
to $m$, it moreover respects the local Langlands correspondence for tori.
\end{abstract}

\section{Introduction} \label{sec: introduction}

\subsection{A crude version of the main result}

Two nonarchimedean local fields $F$ and $F'$ are said to be $l$-close, where $l$
is a positive integer, if $\mO_F/\p_F^l \cong \mO_{F'}/\p_{F'}^l$,
where $\mO_?$ stands for the ring of integers of $?$, and
and $\p_?$ for the maximal ideal of $\mO_?$.

If $F$ and $F'$ are $l$-close, then P. Deligne (\cite{Del84}) constructs
an isomorphism $\Gamma_F/I_F^l \rightarrow \Gamma_{F'}/I_{F'}^l$ now
known as a Deligne isomorphism,
where $\Gamma_?$ denotes the Galois group of a chosen separable closure over $?$,
and $I_?^l$ stands for the $l$-th upper ramification filtration
subgroup of the inertia subgroup $I_? \subset \Gamma_?$.
If further $F$ and $F'$ have finite residue
fields, Kazhdan isomorphisms (see \cite{Kaz86}), pioneered
by D. Kazhdan and studied by various others, notably by
R. Ganapathy (see, e.g., \cite{Gan15} and \cite{Gan22}),
allow us to relate harmonic analysis on reductive
groups over $F$ to that on reductive groups over $F'$.
Thus, for instance, one could hope to study local Langlands correspondence
for a group over $F'$ by using local Langlands correspondence for
a group over $F$, if the latter is known.

One has a good understanding of Kazhdan isomorphisms for split groups,
by \cite{Kaz86} and \cite{Gan15}. For reductive groups that may not be split,
Kazhdan isomorphisms have been constructed by Ganapathy
in \cite{Gan22}. However, a lot of the properties of these isomorphisms
remain to be studied, and such a study is being pursued by
Ganapathy and her collaborators.

In the present paper, we will stick to tori, and investigate
questions related to Kazhdan isomorphisms
$\T(F)/\T(F)_m \cong \T'(F')/\T'(F')_m$ for tori, when $F$ and $F'$ are $l$-close and 
$\T'/F'$ is a \textit{transfer} of $\T/F$, that is, we have an identification
$X^*(\T) = X^*(\T')$ of character lattices, or equivalently
an identification $X_*(\T) = X_*(\T')$ of cocharacter lattices,
as modules over $\Gamma_F/I_F^l$, identified via \cite{Del84}
with $\Gamma_{F'}/I_{F'}^l$,
under the implicitly imposed assumption that $I_F^l$ acts trivially on $X^*(\T)$
and $I_{F'}^l$ on $X^*(\T')$. Here the filtrations $\{\T(F)_m\}_{m \geq 0}$ and 
$\{\T(F')_m\}_{m \geq 0}$ are the minimal congruent filtrations of $\T(F)$ and  $\T(F')$, 
respectively, as defined in \cite{Yu15}.

A crude version of our main result, which we will state in greater detail
in Theorem \ref{thm: Ganapathy Kazhdan isomorphism} below, is as follows:
\begin{thm} \label{thm: Ganapathy Kazhdan isomorphism crude version}
Suppose a local field $F$ is $l$-close to a local field $F'$,
and a torus $\T'/F'$ is a transfer of a torus $\T/F$. Then:
\begin{enumerate}[(i)]
\item If $l$ is large relative to $m$, then there
exists a (necessarily unique) ``congruent''
isomorphism $\T(F)/\T(F)_m \rightarrow \T'(F')/\T'(F')_m$.
\item These isomorphisms are suitably functorial.
\item They are compatible with the isomorphisms constructed by Chai and Yu
(see \cite{CY01}) and with Kottwitz homomorphisms for
tori (see \cite[Section 7]{Kot97} and \cite[Section 11.1]{KP23}).
\item If $l$ is even larger relative to $m$, these isomorphisms
respect the local Langlands correspondence for tori.
\end{enumerate}
\end{thm}

\subsection{Statement of the main result --- more precise version} \label{subsec:1.2}

Now we state our main result, Theorem \ref{thm: Ganapathy Kazhdan isomorphism}
below, in terms of objects and notation defined in later sections, especially
Section \ref{sec: notation}; however, let
us give an introduction to the main objects involved:
\begin{enumerate}[(i)]
\item If we say $(F, \T) \leftrightarrow_l (F', \T')$
(see Notation \ref{notn: close local fields tori}\eqref{close local fields tori}),
we roughly mean that $F$ and $F'$ are discretely valued Henselian fields
with perfect residue field that are $l$-close to each other
($F \leftrightarrow_l F'$), and that the torus $\T'$
over $F'$ is 
a transfer of the torus $\T$ over $F$.

\item For $(F, \T)$ as above, $h(F, \T)$ is a positive integer
from \cite[Section 8.1]{CY01}, sort of upper-bounding the nontriviality
of the smoothening process required to arrive at the
N\'eron model of $\T$.

\item We will be interested in
``congruent isomorphisms'' $\T(F)/\T(F)_m \rightarrow \T'(F')/\T'(F')_m$
(Definition \ref{df: standard congruent Chai-Yu}\eqref{congruent isomorphism}).
These are isomorphisms of abelian groups.

\item Interpolating the $\T(F)_m$, with $m$ varying over nonnegative
integers, are the ``minimal congruent filtration subgroups''
$\T(F)_r$ of J.-K. Yu, with $r$ varying over nonnegative real numbers.

\item Each torus $\T$ determines a relation ``$m \lleq_{\T} l$''
meaning that $l$ is sufficiently large relative to $m$ and the Herbrand
function of a minimal splitting field for $\T$
(see Notation \ref{notn: close local fields tori}\eqref{lleq T}).
\end{enumerate}

Our more precise version
of Theorem \ref{thm: Ganapathy Kazhdan isomorphism crude version}
is as follows; note that it has individual assertions
that are more precise versions of the corresponding assertions 
of Theorem \ref{thm: Ganapathy Kazhdan isomorphism crude version}.

\begin{thm} \label{thm: Ganapathy Kazhdan isomorphism}
\begin{enumerate}[(i)]
\item \label{homomorphisms}
Suppose $(F, \T) \leftrightarrow_l (F', \T')$, set $h = h(F, \T)$,
and suppose $m$ is a positive integer, with
$0 < m + 3 h(F, \T) \lleq_{\T} l$.
Then there exists a (unique) congruent isomorphism
\[ \T(F)/\T(F)_m \rightarrow \T'(F')/\T'(F')_m. \]
Moreover,
if $m + 3 h(F, \T) + 1 \lleq_{\T} l$, then this isomorphism
respects the minimal congruent filtration, i.e., takes the image
of $\T(F)_r$ to that of $\T'(F')_r$, for $0 \leq r \leq m$.

\item \label{homomorphisms functorial} The isomorphisms
of \eqref{homomorphisms} satisfy the following functoriality.
Whenever $(F, \T_i) \leftrightarrow_l (F', \T_i')$ for $i = 1, 2$,
with the same underlying $F \leftrightarrow_l F'$, and
$0 < m + 3 h(F, \T_i) \lleq_{\T_j} l$ for $i, j \in \{1, 2\}$,
and we are given homomorphisms $\T_1 \rightarrow \T_2$ and
$\T_1' \rightarrow \T_2'$ inducing the same homomorphism
$X^*(\T_2) = X^*(\T_2') \rightarrow X^*(\T_1') = X^*(\T_1)$,
the following diagram is commutative:
\[
\xymatrix{
\T_1(F)/\T_1(F)_m \ar[r] \ar[d] & \T_2(F)/\T_2(F)_m \ar[d] \\
\T_1'(F')/\T_1'(F')_m \ar[r] & \T_2'(F')/\T_2'(F')_m
},
\]
where the vertical arrows are as in \eqref{homomorphisms}, and the
horizontal arrows are induced by the homomorphisms $\T_1 \rightarrow \T_2$
and $\T_1' \rightarrow \T_2'$.

\item \label{Chai-Yu Kottwitz compatibility}
In the setting of \eqref{homomorphisms}, we have the following
compatibility with the Chai-Yu isomorphisms and Kottwitz homomorphisms,
in the sense that the following diagram is commutative:
\begin{equation} \label{eqn: Chai-Yu Kottwitz compatibility}
\xymatrix{
\T(F)_b/\T(F)_m \ar[r] \ar[d] &
\T(F)/\T(F)_m \ar[r] \ar[d] & (X_*(\T)_{I_F})^{\Gamma_{\kappa_F}} \ar[d] \\
\T'(F')_b/\T'(F')_m \ar[r] &
\T'(F')/\T'(F')_m \ar[r] & (X_*(\T')_{I_{F'}})^{\Gamma_{\kappa_{F'}}} \\
},
\end{equation}
where $\T(F)_b$ (resp., $\T(F')_b$) denotes the maximal bounded subgroup of $ \T(F)$ (resp., $\T(F')$), the left vertical arrow is induced by the Chai-Yu isomorphism
of \cite[Theorem 8.5]{CY01},
the middle vertical arrow is as in \eqref{homomorphisms}, the right vertical
arrow is induced by the $\Gamma_F/I_F^l = \Gamma_{F'}/I_{F'}^l$-equivariant
identification $X_*(\T) = X_*(\T')$,
and the second horizontal arrow of either row is the Kottwitz
homomorphism.

\item \label{LLC compatibility}
In the setting of \eqref{homomorphisms},
if $F$ and $F'$ are complete and their residue field $\kappa_F = \kappa_{F'}$
is finite, and we assume the stronger inequality
$0 < m + 4 h(F, \T) \lleq_{\T} l$, we have the following compatibility
with the local Langlands correspondence for tori. We have a commutative diagram 
\begin{equation} \label{eqn: LLC compatibility}
\xymatrix{
\Hom(\T(F)/\T(F)_m, \CC^{\times})
\ar@{^{(}->}[r] \ar[d] & H^1(W_F/I_F^l, \hat \T) \ar[d]_{\cong}^{\Del_l} \\
\Hom(\T'(F')/\T'(F')_m, \CC^{\times})
 \ar@{^{(}->}[r] & H^1(W_{F'}/I_{F'}^l, \hat \T')
},
\end{equation}
where the horizontal arrows are given by the local Langlands correspondence
for tori, the left vertical arrow is induced by
the isomorphism $\T(F)/\T(F)_m \rightarrow \T'(F')/\T'(F')_m$ of
\eqref{homomorphisms}, and the right vertical arrow is obtained by combining
the Deligne isomorphism $W_F/I_F^l \cong W_{F'}/I_{F'}^l$ together with
the $\Gamma_F/I_F^l \cong \Gamma_{F'}/I_{F'}^l$-equivariant
identification
$\hat \T = \Hom(X_*(\T), \CC^{\times}) = \Hom(X_*(\T'), \CC^{\times})
= \hat \T'$. Here, to make sense of the top row (to which the bottom row
is analogous), part of the assertion, implicitly,
is that the image of the subset
\[\Hom(\T(F)/\T(F)_m, \CC^{\times}) \subset
\Hom_{\mathrm{cts}}(\T(F), \CC^{\times})\] under the
local Langlands correspondence
is contained in the subset of $H^1(W_F, \hat \T)$ obtained by inflation
from $H^1(W_F/I_F^l, (\hat \T)^{I_F^l}) = H^1(W_F/I_F^l, \hat \T)$.
\end{enumerate}
If further $\T$ is weakly induced in the sense of \cite{KP23},
i.e., satisfies the condition \textup{(T)} of \cite{Yu15}, i.e.,
becomes an induced torus after base-change to some tamely
ramified extension, one can replace $h(F, \T), h(F, \T_1)$
and $h(F, \T_2)$ by $0$ in the above statements.
\end{thm}

\subsection{The case of split tori}

The case of split tori is an obvious extension of the $\GL_1$-case
covered by \cite{Del84}, and is a very special case of \cite{Kaz86}.

\subsubsection{Deligne's triples.}
We first digress to remark on some fine print.
For all these considerations, choosing isomorphisms is important to ensure
that various constructions are well-defined. Thus, Deligne works
not with isomorphisms $\mO_F/\p_F^l \cong \mO_{F'}/\p_{F'}^l$ of
truncated discrete valuation rings, but rather
with slightly more rigidified data in the form of isomorphisms
$(\mO_F/\p_F^l, \p_F/\p_F^{l + 1}, \epsilon) \rightarrow
(\mO_{F'}/\p_{F'}^l, \p_{F'}/\p_{F'}^{l + 1}, \epsilon')$
of triples, where $\epsilon$ is the obvious map $\p_F/\p_F^{l + 1} \rightarrow
\mO_F/\p_F^l$, and $\epsilon'$ is analogous. Fixing such an isomorphism
is what lets one construct the Deligne isomorphism
$\Gamma_F/I_F^l \cong \Gamma_{F'}/I_{F'}^l$ of \cite{Del84}
and show that it is well-defined up to an inner conjugation.

\subsubsection{The case of $\GL_1$, from \cite{Del84}}
First, if $\T = \GL_1/F$ and $\T' = \GL_1/F'$ compatibly, the required
isomorphism, say when $m = l$, is the isomorphism
\begin{equation} \label{eqn: truncated multiplicative groups}
F^{\times}/(1 + \p_F^l) \rightarrow {F'}^{\times}/(1 + \p_{F'}^l)
\end{equation}
constructed by Deligne from the realization
$F \leftrightarrow_l F'$, in a canonical manner starting from the
truncated data, in \cite[Section 1.2]{Del84}. 
A more concrete but slightly less obviously canonical description
for \eqref{eqn: truncated multiplicative groups} requires it to
\begin{itemize}
\item restrict to $(\mO_F/\p_F^l)^{\times} \cong (\mO_{F'}/\p_{F'}^l)^{\times}$
on $\mO_F^{\times}/(1 + \p_F^l) = (\mO_F/\p_F^l)^{\times}$, and
\item send the image of a uniformizer $\varpi_F$ to that of
a uniformizer $\varpi_{F'}$ whenever $\varpi_F$ and $\varpi_{F'}$
are compatible under $\p_F/\p_F^{l + 1} \rightarrow \p_{F'}/\p_{F'}^{l + 1}$.
\end{itemize}

\subsubsection{General split tori}
When $\T$ and $\T'$ are split but
otherwise general, the datum relating $\T'$ to $\T$
amounts to just an isomorphism $X^*(\T) \rightarrow X^*(\T')$
of abelian groups, say $\chi \mapsto \chi'$. Then our
isomorphism $\T(F)/\T(F)_m \rightarrow \T'(F')/\T'(F')_m$
is defined so as to match the images of $t \in \T(F)$ and $t' \in \T'(F')$
precisely when for each $\chi \in X^*(\T)$ identifying with
$\chi' \in X^*(\T')$, the images of $\chi(t)$ and $\chi'(t')$ correspond
under \eqref{eqn: truncated multiplicative groups}.

For general tori, in the spirit of the above discussion,
we find it convenient to specify the isomorphism
$\T(F)/\T(F)_m \rightarrow \T'(F')/\T'(F')_m$ by
forcing compatibilities that characterize it.

\subsection{Standard and congruent isomorphisms}

In addition to the congruent filtration subgroups $\T(F)_m$,
we will need the ``naive'' filtration subgroups $\T(F)^{\naive}_r$ ($r \geq 0$):
\begin{equation} \label{eqn:naive}
 \T(F)^{\naive}_r = \{t \in \T(F) _b \mid \val_F(\chi(t) - 1) \geq r, \; \forall \,
\chi \in X^*(\T)\}, 
 \end{equation}
for a suitable extension $\val_F$ of the normalized discrete valuation on $F$.

When $(F, \T) \leftrightarrow_l (F', \T')$, one defines:
\begin{enumerate}[(i)]
\item for $0 < r \lleq_{\T} l$,
a standard isomorphism $\T(F)/\T(F)^{\naive}_r \rightarrow
\T'(F')/\T'(F')^{\naive}_r$ to be one that matches the images of $t \in \T(F)$
and $t' \in \T'(F')$ whenever $\chi(t)$ and $\chi'(t')$ have images
that match under a suitable extension of \eqref{eqn: truncated multiplicative
groups} (see Definition \ref{df: standard congruent Chai-Yu}\eqref{standard isomorphism} for more details).

\item a congruent isomorphism $\T(F)/\T(F)_m \rightarrow
\T'(F')/\T'(F')_m$ to be an isomorphism induced by a standard isomorphism
after passage to maximal unramified extensions
(see Definition \ref{df: standard congruent Chai-Yu}\eqref{congruent isomorphism}
for more details).
\end{enumerate}

Standard isomorphisms are unique when they exist,
and have good functoriality properties, compatibility
with the Kottwitz homomorphism, and
(in the case of complete fields with finite residue field)
compatibility with the local Langlands correspondence.
Congruent isomorphisms inherit
the first three of these four properties, and under stronger assumptions
the fourth too.

Like with \cite{Gan22}, we too make use of an argument following
the construction of the Kottwitz homomorphism 
for tori (see \cite[Proposition 11.1.1]{KP23}):
the following simple yet not entirely obvious fact gets us started
(Proposition \ref{pro: standard isomorphism strictly Henselian}).
\begin{pro}
If $(F, \T) \leftrightarrow_l (F', \T')$, $F$ is strictly Henselian, and
$0 < r \lleq_{\T} l$, then a standard isomorphism
$\T(F)/\T(F)^{\naive}_r \rightarrow \T'(F')/\T'(F')^{\naive}_r$ exists.
\end{pro}

The difficulty is that we cannot see an obvious way to descend this
to the non-strictly-Henselian case, without going through
congruent isomorphisms. This difficulty is what
motivates congruent isomorphisms for us, notwithstanding
their unpleasantness: indeed, if
$\tilde F/F$ is a maximal unramified extension and
$\tilde F \leftrightarrow_l \tilde F'$ ``lies over'' $F \leftrightarrow_l F'$,
then a ``Galois invariant'' isomorphism $\T(\tilde F)/\T(\tilde F)_m \rightarrow
\T'(\tilde F')/\T'(\tilde F')_m$, on taking Galois invariants, gives
us an isomorphism $\T(F)/\T(F)_m \rightarrow \T'(F')/\T'(F')_m$: 
$\T(\tilde F)_m$ has trivial $\Gal(\tilde F/F)$-cohomology, being
sort of pro-unipotent over the (perfect) residue field of $\mO_F$
(\cite[Proposition 13.8.1]{KP23}).

\subsection{Using the work of Chai and Yu}

Thus, the main question now becomes: when can a standard
isomorphism $\T(\tilde F)/\T(\tilde F)^{\naive}_r \rightarrow
\T'(\tilde F')/\T'(\tilde F')^{\naive}_r$ induce an
isomorphism
$\T(\tilde F)/\T(\tilde F)_m \rightarrow \T'(\tilde F')/\T'(\tilde F')_m$?
This seems to require a much deeper ingredient:
the spectacular work of Chai and Yu (\cite{CY01}).

Under the assumption $m + 3 h(F, \T) \lleq_{\T} l$
of Theorem \ref{thm: Ganapathy Kazhdan isomorphism},
Chai and Yu construct a canonical isomorphism
$\mathcal{T}^{\ft} \times_{\mO_F} \mO_F/\p_F^m
\rightarrow {\mathcal{T}'}^{\ft} \times_{\mO_{F'}} \mO_{F'}/\p_{F'}^m$,
where $\mathcal{T}^{\ft}$ is the finite type N\'eron model of $\T$,
and ${\mathcal{T}'}^{\ft}$ that of $\T'$. The properties that characterize
their isomorphism are implicit in their construction.
A careful examination of their construction, together with
a few arguments that are tedious but not difficult
(see Proposition \ref{pro: Chai-Yu isomorphism CY01}),
tells us that their isomorphism
$\mathcal{T}^{\ft} \times_{\mO_F} \mO_F/\p_F^m
\rightarrow {\mathcal{T}'}^{\ft} \times_{\mO_{F'}} \mO_{F'}/\p_{F'}^m$
is characterized by the fact that, upon evaluating
on $\mO_{\tilde F}/\p_{\tilde F}^m = \mO_{\tilde F'}/\p_{\tilde F'}^m$,
it is induced by a restriction of a standard isomorphism
$\T(\tilde F)/\T(\tilde F)^{\naive}_{m + h(F, \T)} \rightarrow
\T'(\tilde F')/\T'(\tilde F')^{\naive}_{m + h(F, \T)}$. This is the
key observation that lets us construct congruent (and hence also
standard) isomorphisms outside the strictly Henselian
and ``weakly induced'' cases.

Thus, more generally, we define a Chai-Yu isomorphism to be an
isomorphism $\mathcal{T}^{\ft} \times_{\mO_F} \mO_F/\p_F^m
\rightarrow {\mathcal{T}'}^{\ft} \times_{\mO_{F'}} \mO_{F'}/\p_{F'}^m$
that, when evaluated at $\mO_{\tilde F}/\p_{\tilde F}^m = \mO_{\tilde F'}/\p_{\tilde F'}^m$,
is induced by a standard isomorphism
$\T(\tilde F)/\T(\tilde F)^{\naive}_r \rightarrow
\T'(\tilde F')/\T'(\tilde F')^{\naive}_r$
for some $0 < r \lleq_{\T} l$ such that $\T(\tilde F)^{\naive}_r \subset
\T(\tilde F)_m$ and $\T'(\tilde F')^{\naive}_r \subset
\T'(\tilde F')_m$. The existence of a Chai-Yu isomorphism easily
gives the existence of a congruent isomorphism that is
tautologically compatible with the Chai-Yu isomorphism.

\subsection{The organization of the paper}
In Section \ref{sec: notation}, we define notation and recall
some material, as well as provide some simple arguments such
as an explanation as to why various results of \cite{Del84}, though
stated for ``local fields'' (quotient
fields of complete discrete valuation rings with perfect residue
field) automatically extend to the case of Henselian discretely valued fields
with perfect residue fields. This
extension is convenient because our arguments involve passage
to maximal unramified extensions, which preserves Henselian-ness
but not completeness. We also review the result of Chai and Yu
of interest to us, stating
their isomorphism and articulating the characterization implicit
in their work, in Theorem \ref{thm: CY01}.

In Section \ref{sec: standard congruent Chai-Yu}, we study standard
isomorphisms, congruent isomorphisms and Chai-Yu isomorphisms.
Much of the content of this section has been summarized above.

In Section \ref{sec: weakly induced tori}, we restrict to the
case of tori that are weakly induced in the sense of \cite{KP23}.
These tori are much simpler than general tori. In this
case, one can construct a Chai-Yu isomorphism under the milder,
natural, assumption $0 < m \lleq_{\T} l$
(Proposition \ref{pro: weakly induced tori Chai-Yu isomorphism}),
which is what lets us replace $h(F, \T)$ by $0$ in the statement
of Theorem \ref{thm: Ganapathy Kazhdan isomorphism}
for weakly induced tori.

Finally, in Section \ref{sec: putting things together},
we put things together and prove
Theorem \ref{thm: Ganapathy Kazhdan isomorphism}.

\smallskip

{\bf \noindent Acknowledgements.}
We are grateful to Radhika Ganapathy, whose work,
especially \cite{Gan15} and \cite{Gan22}, is the primary inspiration
for this paper. We also thank Tasho Kaletha, Arnaud Mayeux,
Gopal Prasad and Dipendra Prasad for helpful and educative discussions as well
as encouragement.  The recent appearance
in literature of the book \cite{KP23} proved propitious
to the writing of this paper, by allowing us to make several simplifications
and improvements.

\section{Notation and review.} \label{sec: notation}

\subsection{Discretely valued Henselian fields} \label{subsec: DVHF}

We will abbreviate ``discretely valued Henselian field" to ``DVHF''.
We will be interested  in DVHFs with perfect (but not necessarily
finite) residue fields, unlike
\cite{Del84}, which additionally imposes completeness.
This is because we will need to pass to maximal
unramified extensions $\tilde F$ of fields $F$ of interest
(see \cite[just before Corollary 2.4.6]{Ber93};
valued Henselian fields are the quasi-complete fields
of \cite[Definition 2.3.1 and Proposition 2.4.3]{Ber93}), and
doing so preserves only Henselian-ness, not completeness.

\subsubsection{Objects associated to a DVHF}
\label{subsubsec: DVHF}
For a DVHF $F$, we will denote by $\mO_F$ its ring of integers,
$\p_F \subset \mO_F$ the maximal ideal of $\mO_F$, $\kappa_F =
\mO_F/\p_F$ its residue field, and $\val_F$ the
normalized discrete valuation of $F$ as well as its own extension to
any algebraic extension of $F$. Given a
field $F$, it will often be implicitly understood that a separable closure
$F^{\sep}$ has been chosen. For a DVHF $F$, we will write $\Gamma_F$ and
$I_F$ respectively for the absolute Galois group
$\Gal(F^{\sep}/F)$ and the inertia group of $F$,
and $\Gamma_{E/F}$ for the Galois group
of any Galois extension $E/F$. Let $\Gamma_{\kappa_F} = \Gamma_F/I_F$; it
is isomorphic to the absolute Galois group of $\kappa_F$.
In case $\kappa_F$ is finite, $\Gamma_{\kappa_F}$
is topologically generated by a Frobenius element, and we will write
$W_F \subset \Gamma_F$ for the Weil group, namely the inverse image
in $\Gamma_F$ of the subgroup of $\Gamma_{\kappa_F}$ abstractly generated the
Frobenius element. Given a DVHF $F$ with perfect residue field,
we will often be interested in separable finitely ramified extensions $E/F$,
which may not be finite or Galois. In such a situation, we will use without
further comment that $E$ is also a DVHF with perfect residue field
(thus, $\kappa_E$ is  algebraically closed if $E$ contains a maximal unramified
extension of $F$), and denote by $e(E/F)$ the associated ramification degree.

\subsubsection{Passage to completion} \label{subsubsec: completion}
Write $\hat F$ for the completion of any DVHF $F$; it is a complete DVHF.
Let $F$ be a DVHF. For the following, we refer to
\cite[Proposition 2.1.6]{KP23} and \cite[Proposition 2.4.1]{Ber93}.
If $E/F$ is a finite separable extension, noting that $E$ is 
also a DVHF, the obvious map
$E \otimes_F \hat F \rightarrow \hat E$ of rings is an isomorphism.
We have an equivalence of categories between the category
$\ext F$ of finite separable extensions
of $F$ and the analogous category $\ext \hat F$,
given by $E \mapsto \hat E$, so that choosing an embedding of
$F^{\sep}$ into ${\hat F}^{\sep} := (\hat F)^{\sep}$
gives a canonical identification $\Gamma_F \rightarrow \Gamma_{\hat F}$.

\subsubsection{Ramification theory}
\label{subsubsec: ramification theory}
We assume the setting and notation from Subsubsection
\ref{subsubsec: completion} above, but assume also that $\kappa_F$
is perfect. Let $E/F$ be a finitely ramified separable extension.
For each $r \in [-1, \infty)$,
we have the ``lower ramification (equivalence) relation'' $\Xi_r$ on
$\Hom_{F-\mathrm{alg}}(E, F^{\sep})$ as in \cite[(A.3.3)]{Del84} (whose
$R_u$ is our $\Xi_r$), under which $\sigma$ and $\tau$ are equivalent
if and only if for some (or equivalently, any) finitely ramified
Galois extension $M/F$ contained in $F^{\sep}$ and containing
a normal closure of $E$,
we have $\val_M(\sigma(x) - \tau(x)) \geq e(M/E) (r + 1)$ for all
$x \in \mO_E$. When $E = M$ is Galois, the lower ramification
subgroup $\Gal(E/F)_r \subset \Gal(E/F)$ is the equivalence class
of the identity element under $\Xi_r$, which is then just ``lies in the
same $\Gal(E/F)_r$-coset''.

Choosing any $M$ as above, $\Gal(M/F)$ has a well-defined
transitive action on the set of $\Xi_r$-equivalence classes, so that
they all have the same cardinality. This cardinality
is easily seen to be bounded above
by $e(E/F)$ for $r > -1$: if $\sigma, \tau \in \Hom_{F-\mathrm{alg}}(E, F^{\sep})$
belong to the same class under $\Xi_r$ with $r > -1$, it is an easy exercise
to see that they agree on the maximal
unramified subextension of $E/F$.
For $r > -1$, let $1 \leq g_r \leq e(E/F)$ be the
cardinality of each $\Xi_r$-equivalence class (our $g_r$
is the $r_u$ of \cite{Del84}). This also lets us define the Herbrand function
associated to each finitely ramified separable extension
$E/F$ by the familiar integral as in \cite[(A.4.3)]{Del84},
for $r \in  [0, \infty)$:
\begin{equation} \label{eqn: Herbrand function}
e(E/F)^{-1} r \leq \varphi_{E/F}(r) = \int_0^r dt/(g_0/g_t) \leq r.
\end{equation}
$\varphi_{E/F}$ is a piecewise linear
self-homeomorphism of $[0, \infty)$. In fact, one can show that this also
defines a self-homeomorphism of $[-1, \infty)$, but we will only
be interested in its values on $[0, \infty)$.
Let its inverse be $\psi_{E/F}$, another self-homeomorphism of $[0, \infty)$.

The following allow us to reduce the study of
$\varphi_{E/F}$ and $\psi_{E/F}$ to the case
where $E/F$ is finite and $F$ is complete:
\begin{itemize}
\item For any subextension $E_{\circ}/F$ of $E/F$ with
$e(E_{\circ}/F) = e(E/F)$ --- note that
there exist finite such $E_{\circ}/F$ ---
we have $\varphi_{E/F} = \varphi_{E_{\circ}/F}$
and $\psi_{E/F} = \psi_{E_{\circ}/F}$.
To see this, make use of the same argument that was
used above to prove that $g_r \leq e(E/F)$ for $r > -1$, to see that
the value of $g_r$ associated to $E/F$ equals
that associated to $E_{\circ}/F$.

\item If $E/F$ is finite, it is easy to see that the identification
\[\Hom_{F-\mathrm{alg}}(E, F^{\sep}) \rightarrow
\Hom_{\hat F-\mathrm{alg}}(\hat E, (\hat F)^{\sep})\]
respects each of the equivalence
relations $\Xi_r$, so that $\varphi_{E/F} = \varphi_{\hat E/\hat F}$
and $\psi_{E/F} = \psi_{\hat E/\hat F}$
as functions $[0, \infty) \rightarrow [0, \infty)$.
\end{itemize}

We claim that for finitely ramified separable extensions
$E_2/E_1/F$, we have
$\varphi_{E_2/F} = \varphi_{E_1/F} \circ \varphi_{E_2/E_1}$,
or equivalently $\psi_{E_2/F} = \psi_{E_2/E_1} \circ \psi_{E_1/F}$.
To see this, reduce using the arguments above
to the case where $E_2/F$ is finite, and then to the case where
$F, E_1$ and $E_2$ are complete, then to the case where
$E_2/F$ is Galois, and use \cite[(A.4.1) and Proposition A.4.2]{Del84}.

For $E/F$ finitely ramified separable, define
the ``upper ramification relations'' $\Xi^r := \Xi_{\psi_{E/F}(r)}$.
If $E/F$ is Galois, this is the
``belongs to the same coset'' relation for the upper ramification subgroup
$\Gal(E/F)^r := \Gal(E/F)_{\psi_{E/F}(r)} \subset \Gal(E/F)$.

When $E/F$ is finite, it is easy to check that
$\Hom_{F-\mathrm{alg}}(E, F^{\sep}) \rightarrow
\Hom_{\hat F-\mathrm{alg}}(\hat E, (\hat F)^{\sep})$
preserves the ``lower ramification relations'' $\Xi_r$, and  hence
(using $\psi_{E/F} = \psi_{\hat E/\hat F}$) also the
``upper ramification relations'' $\Xi^r$. The $\Xi^r$ have the following
advantage: if $E/F$ is a finite extension and $M/F$ is a finite Galois
extension in $F^{\sep}$ containing a normal closure of $E$,
then $\Xi^r$ is the same as ``lies in the same
$\Gal(M/F)^r$-orbit'': to see this,
pass to completion and use \cite[(A.3.2) and the last sentence
of A.4]{Del84}. This nice behavior under quotients lets us
give $\Gamma_F$ an upper ramification filtration $\{I_F^r\}_{r \geq 0}$:
by definition, the upper ramification filtration subgroup
$I_F^r \subset \Gamma_F$
is the subgroup of elements that map to $\Gal(M/F)^r$ for each
finite Galois subextension $M/F$ of $F^{\sep}/F$. Each such
map $I_F^r \rightarrow \Gal(M/F)^r$ is
then seen to be surjective. The isomorphism
$\Gamma_F \rightarrow \Gamma_{\hat F}$ maps
$I_F^r$ to $I_{\hat F}^r$ and quotients to
an isomorphism $\Gamma_F/I_F^r \rightarrow \Gamma_{\hat F}/I_{\hat F}^r$,
for each $r \geq 0$. The objects associated to $F$
that we have defined above ($\psi_{E/F}, I_F^r$ etc.) are intrinsic, and
their definitions did not make use of the embedding $F^{\sep} \hookrightarrow
{\hat F}^{\sep}$.

\subsubsection{At most $l$-ramified extensions}
\label{subsubsec: at most l ramified extensions}
Let $F$ be a DVHF with
perfect residue field, and $l$ a nonnegative integer.
A separable (algebraic) extension $E/F$ will be called
at most $l$-ramified if for every finite subextension $E_{\circ}/F$
of it, the relation $\Xi^l$ associated to $E_{\circ}/F$ is trivial.
Using \cite[Proposition A.6.1]{Del84}, this can be shown to be
equivalent to requiring
that $I_F^l$ fixes any $F$-algebra embedding $E \hookrightarrow F^{\sep}$. 
Let $(\ext F)^l$ denote the category of finite at most
$l$-ramified (and hence separable by definition) extensions of $F$.
The discussion in the previous subsubsection
implies that the functor
$E \mapsto \hat E$ from Subsubsection \ref{subsubsec: completion}
induces an equivalence of categories $(\ext F)^l \rightarrow (\ext \hat F)^l$.
The category of 
the ind-objects of  $(\ext F)^l$ is equivalent to the category of the algebraic
(not necessarily finite or finitely ramified)
at most $l$-ramified extensions of $F$.

\begin{remark} \label{rmk: at most l ramified extensions}
Here are some properties of a finitely ramified
at most $l$-ramified extension $E/F$, that will be used
without further comment in what follows:
\begin{enumerate}[(i)]
\item \label{l(1)} $l(1) := \psi_{E/F}(l)$ is an integer,
equal to $e(E/F) (l + 1) - \val_F(\text{the different of $E/F$}) - 1$
if $E/F$ is finite (reduce to the case of $E/F$ finite,
and see around \cite[Proposition A.6.1]{Del84};
this uses that $E/F$ is at most $l$-ramified). Note also
that $l \leq l(1) \leq l e(E/F)$, with the latter equality
holding if and only if $E/F$ is tamely ramified
(use, e.g., \eqref{eqn: Herbrand function} and the previous sentence).

\item \label{IFl IEl(1)} $I_F^l = I_E^{l(1)}$: reduce to the case where $E/F$ is
finite, and note that for any finite Galois extension $M/E$
we have $\Gal(M/F)^l \subset \Gal(M/E)$, and then: 
\[ \Gal(M/F)^l = \Gal(M/E) \cap \Gal(M/F)_{\psi_{M/F}(l)}
= \Gal(M/E)_{\psi_{M/E} \circ \psi_{E/F}(l)}
= \Gal(M/E)^{\psi_{E/F}(l)}.
\]
\end{enumerate}
\end{remark}

\subsection{The Krasner-Deligne theory (\cite{Del84})}

\subsubsection{Deligne's triples}
\label{subsubsec: Trl}
In Subsubsections \ref{subsubsec: Trl F}
and \ref{subsubsec: close local fields} below, we will
write $\mathfrak{T}$ for the category of triples
$(A, \m, \epsilon)$ as in \cite[Sections 1.1 and 1.4]{Del84}: $A$ is
a truncated discrete valuation ring with perfect residue field,
$\m$ is a free $A$-module of rank $1$, and $\epsilon \colon \m \rightarrow A$
is an $A$-module epimorphism from $\m$ to the maximal ideal of $A$.
For an object $S = (A, \m, \epsilon)$ in this category, $l(S)$ will denote
the length of $A$ as an $A$-module.

For each object $S = (A, \m, \epsilon)$ of $\mathfrak{T}$ of length say $l$,
Deligne has defined a category $(\ext S)^l$ in
\cite[Definition 2.7]{Del84}, which we will refer to in this subsection.
Its objects are the ``finite flat'' objects over $S$
in $\mathfrak{T}$ satisfying an ``at most $l$-ramified'' condition,
but its morphisms are only certain equivalence
classes of morphisms in $\mathfrak{T}$.

\subsubsection{Deligne's triple $\Tr_l F$ associated to $F$}
\label{subsubsec: Trl F}
If $F$ is a DVHF with perfect residue field,
and $l$ a positive integer, we will
write $\Tr_l F$ for the object $(\mO_F/\p_F^l, \p_F/\p_F^{l + 1},
\epsilon)$ in $\mathfrak{T}$, where
$\epsilon \colon \p_F/\p_F^{l + 1} \rightarrow \mO_F/\p_F^l$ is induced
by the inclusion $\p_F \subset \mO_F$.

\cite[Theorem 2.8]{Del84} says that if $F$ is complete, then
$\Tr$ extends to a well-defined functor from $(\ext F)^l$
to the category $(\ext \Tr_l F)^l$ (see Subsubsection \ref{subsubsec: Trl}),
which is in fact an equivalence of categories.
However, even when $F$ is not complete, using the equivalence of categories
$(\ext F)^l \rightarrow (\ext \hat F)^l$
(Subsubsection \ref{subsubsec: at most l ramified extensions})
and the tautological isomorphism $\Tr_l F \rightarrow \Tr_l \hat F$,
we formally get an equivalence of categories $(\ext F)^l \rightarrow
(\ext \Tr_l F)^l$. Moreover, 
it is immediate that this equivalence of categories has an intrinsic description
independent of $\hat F$, exactly as the one used for
$\hat F$ in \cite[Theorem 2.8]{Del84}. In particular, any isomorphism
$\Tr_l F \rightarrow \Tr_l F'$ determines an equivalence
of categories $(\ext F)^l \rightarrow (\ext F')^l$; this takes
each $E$ in $(\ext F)^l$ to some $E'$ in $(\ext F')^l$
such that we have an isomorphism
$\Tr_{e l} E = \Tr_{e l} E'$ in the category
$(\ext \Tr_l F)^l = (\ext \Tr_l F')^l$, where $e = e(E/F) = e(E'/F')$.

\subsubsection{Close local fields} \label{subsubsec: close local fields}
The notation $F \leftrightarrow_l F'$ will mean not only that
$F$ and $F'$ are DVHFs with perfect residue fields that are
$l$-close in the sense that
we have an isomorphism $\mO_F/\p_F^l \cong \mO_{F'}/\p_{F'}^l$ of rings,
but also that the following additional data have been chosen:
\begin{itemize}
\item An identification $\Tr_l F = \Tr_l F'$ has been chosen in $\mathfrak{T}$,
as also an equivalence of categories
$U  \colon (\ext F)^l \rightarrow (\ext F')^l$ as in
Subsubsection \ref{subsubsec: Trl F} above,
which then as in \cite[Section 3.5]{Del84},
determines the inner class of an isomorphism
$\Gamma_F/I_F^l \cong \Gamma_{F'}/I_{F'}^l$.

\item An identification $\Gamma_F/I_F^l = \Gamma_{F'}/I_{F'}^l$
from the inner class mentioned above has been
chosen by means of a choice of a fixed
isomorphism $U((F^{\sep})^{I_F^l}) \rightarrow ({F'}^{\sep})^{I_{F'}^l}$
over $F'$, as follows: 
\[
\Gamma_F/I_F^l = \Aut_{F-\mathrm{alg}}(({F}^{\sep})^{I_F^l})
\overset{U}{\rightarrow} \Aut_{F'-\mathrm{alg}}(U((F^{\sep})^{I_F^l}))
= \Aut_{F'-\mathrm{alg}}(({F'}^{\sep})^{I_{F'}^l}) = \Gamma_{F'}/I_{F'}^l,
\]
where $U\colon (\ext F)^l \cong (\ext F')^l$ is now 
extended to the level of the ind-objects.

We refer to \cite[Section 3.5, especially Section 3.5(c)]{Del84}
for some of the details, which do not need the assumption that $F$ is
complete.
\end{itemize}
This involves choices of $F^{\sep}$ and ${F'}^{\sep}$
among other things, changing which will change associated objects
in an appropriate sense, e.g., up to an inner automorphism for Galois groups.

\subsubsection{A variation} \label{subsubsec: close local fields variation}
Given $F \leftrightarrow_l F'$, it will be helpful to consider the
following variant of $U$. Let $(\ext F)^{l, +}$ be the category of
embeddings $E \hookrightarrow F^{\sep}$, where $E/F$ is
finitely ramified and at most $l$-ramified. Similarly, we have
$(\ext F')^{l, +}$. These categories have the following advantage which
is important for us: between any two objects in them is either
a unique morphism, or none at all, depending on whether the stabilizer
of the former in $\Gamma_F$ or $\Gamma_{F'}$
contains that of the latter. Clearly, $U$, considered at the
level of ind-objects, defines an equivalence of categories
$U^+ \colon(\ext F)^{l, +} \rightarrow (\ext F')^{l, +}$; it sends
$E \hookrightarrow F^{\sep}$ to
$U(E) \hookrightarrow U((F^{\sep})^{I_F^l}) \rightarrow
({F'}^{\sep})^{I_{F'}^l} \hookrightarrow {F'}^{\sep}$, where
$U((F^{\sep})^{I_F^l}) \rightarrow ({F'}^{\sep})^{I_{F'}^l}$ is part
of the datum $F \leftrightarrow_l F'$.
Moreover, it is an easy exercise to describe when
$E' \hookrightarrow {F'}^{\sep}$ is isomorphic in
$(\ext F')^{l, +}$ to $U^+(E \hookrightarrow F^{\sep})$:
this is so if and only if $E \hookrightarrow F^{\sep}$ and
$E' \hookrightarrow {F'}^{\sep}$ have the same stabilizer
in $\Gamma_F/I_F^l = \Gamma_{F'}/I_{F'}^l$. \\
Upshot: the datum of an extension $E'/F$ and an isomorphism
$U(E) \rightarrow E'$, is the same as that of an embedding
$E' \hookrightarrow {F'}^{\sep}$ with the same stabilizer
as $E \hookrightarrow F^{\sep}$ in $\Gamma_F/I_F^l = \Gamma_{F'}/I_{F'}^l$.

\subsubsection{Close extensions of close local fields}
\label{subsubsec: TrE}

Given $F \leftrightarrow_l F'$, we will typically need to
work with realizations $E \leftrightarrow_{l(1)} E'$ involving
finitely ramified at most $l$-ramified extensions
$E$ and $E'$ of $F$ and $F'$ with $U(E) \cong E'$,
where $l(1)= \psi_{E/F}(l)$ is an integer.
We will use without further mention that, for any such $E'$, we have
$\psi_{E/F}(l) = \psi_{E'/F'}(l)$:
use the discussion in \cite[Section 1.5.3]{Del84}
(expressing $\psi_{E/F}(l)$ in terms of truncated data).

Let us describe how a choice of an isomorphism $U(E) \rightarrow E'$
gives a realization $E \leftrightarrow_{l(1)} E'$.
As in \cite[Construction 3.4.1]{Del84}
(and using a direct limit argument to reduce to the case of finite extensions),
this choice determines an isomorphism
$\Tr_{l(1)} E \rightarrow \Tr_{l(1)} E'$;
by Subsubsection \ref{subsubsec: Trl F}
above this does not need the assumption that $F$ is complete.
Moreover, we have obvious choices of the additional data
needed to upgrade this to a realization
$E \rightarrow_{l(1)} E'$, as follows. An analogue $U_E$
of $U$ for this isomorphism can be obtained by
restricting $U$ to extensions of $E$ and thinking of
$U(E) \rightarrow E'$ as an identification, which we can also
use to choose $F^{\sep}$ and ${F'}^{\sep}$ as
algebraic closures of $E$ and $E'$, and get an
identification $U_E((F^{\sep})^{I_E^{l(1)}})
= U((F^{\sep})^{I_F^l}) \rightarrow ({F'}^{\sep})^{I_{F'}^l}
= ({F'}^{\sep})^{I_{E'}^{l(1)}}$.

\begin{remark} \label{rmk: properties E l(1) E'}
The above construction of $E \leftrightarrow_{l(1)} E'$ has the
following properties, which will be used without further mention
in what follows:
\begin{enumerate}[(i)]
\item \label{GammaE GammaF GammaE' GammaF'}
$\Gamma_E$ and $\Gamma_{E'}$, as realized in
$E \leftrightarrow_{l(1)} E'$, identify with the stabilizers of
$E \hookrightarrow F^{\sep}$ in $\Gamma_F$ and $E' \hookrightarrow {F'}^{\sep}$
in $\Gamma_{F'}$, and the resulting isomorphism
$\Gamma_E/I_E^{l(1)} \rightarrow \Gamma_{E'}/I_{E'}^{l(1)}$ is simply
the restriction of $\Gamma_F/I_F^l \rightarrow \Gamma_{F'}/I_{F'}^l$.
\item \label{GammaEF GammaE'F'} We get a bijection
\begin{equation} \label{eqn: GammaEF GammaE'F'}
\Gamma_{E/F} \rightarrow (\Gamma_F/I_F^l)/(\Gamma_E/I_E^{l(1)}) \rightarrow
(\Gamma_{F'}/I_{F'}^l)/(\Gamma_{E'}/I_{E'}^{l(1)}) \rightarrow
\Gamma_{E'/F'},
\end{equation}
which is an isomorphism of groups if $E/F$ is Galois,
in which case these groups act compatibly on $\Tr_{l(1)} E = \Tr_{l(1)} E'$
over $\Tr_l F = \Tr_l F'$.
\item \label{extensions transitivity}
Suppose $L/F$ is a finitely ramified at most $l$-ramified
extension `containing' $E/F$. Then any extension $L'/F'$ together with
an isomorphism $U(L) \rightarrow L'$ determines an extension $L'/E'$
(via $E' \rightarrow U(E) \rightarrow U(L) \rightarrow L'$),
together with an isomorphism $U_E(L) = U(L) \rightarrow L'$,  and vice versa.
Given any such $L'$ and $U(L) \rightarrow L'$,
it is easily verified that the realization $L \leftrightarrow_{\psi_{L/F}(l)}
L'$, obtained using the above construction starting from
$F \leftrightarrow_l F'$ and $U(L) \rightarrow  L'$,
is the same as the realization
$L \leftrightarrow_{\psi_{L/E}(\psi_{E/F}(l))} L'$, obtained using
the above construction starting from $E \leftrightarrow_{\psi_{E/F}(l)} E'$
and $U_E(L) = U(L) \rightarrow L'$.
\end{enumerate}
\end{remark}

However, it seems inconvenient to keep track of fixed isomorphisms
$U(E) \rightarrow E'$, or even to keep referring to $U$.
This is why we had Subsubsection
\ref{subsubsec: close local fields variation}: the discussion
there shows that, given $E \hookrightarrow F^{\sep}$ in
$(\ext F)^{l, +}$, any $E' \hookrightarrow {F'}^{\sep}$ with
the same stabilizer as it in $\Gamma_F/I_F^l = \Gamma_{F'}/I_{F'}^l$
determines a unique isomorphism $U(E) \rightarrow E'$, and
hence a realization $E \leftrightarrow_{l(1)} E'$.

\begin{notn} \label{notn: compatible embeddings}
If a realization $F \leftrightarrow_l F'$ is understood, and
we talk of compatible embeddings $E \hookrightarrow F^{\sep}$
and $E' \hookrightarrow {F'}^{\sep}$, we will mean that
$E/F$ and $E'/F'$ are finitely ramified at most $l$-ramified
extensions, and that $E \hookrightarrow F^{\sep}$
and $E' \hookrightarrow {F'}^{\sep}$ have the same stabilizer
in $\Gamma_F/I_F^l = \Gamma_{F'}/I_{F'}^l$. Thus, the compatible
embeddings $E \hookrightarrow F^{\sep}$
and $E' \hookrightarrow {F'}^{\sep}$ give a realization
$E \leftrightarrow_{l(1)} E'$, where $l(1) = \psi_{E/F}(l)$,
``lying over'' $F \leftrightarrow_l F'$.
\end{notn}

\begin{remark} \label{rmk: properties E l(1) E' rephrasing}
By the above discussion above, Remark \ref{rmk: properties E l(1) E'}
can be stated in terms of compatible embeddings, with, in particular,
Remark \ref{rmk: properties E l(1) E'}\eqref{extensions transitivity}
taking the following shape (with a slight change of notation):
if $L_1/L_2/F$ are finitely ramified at most
$l$-ramified extensions, and 
$L_i \hookrightarrow F^{\sep}$ and $L_i' \hookrightarrow {F'}^{\sep}$
are compatible embeddings for $i = 1, 2$, then $L_1'
\hookrightarrow {F'}^{\sep}$ is the composite of
$L_2' \hookrightarrow {F'}^{\sep}$ with an embedding
$L_1' \hookrightarrow L_2'$.
Moreover, with $l_1 = \psi_{L_1/F}(l)$ and
$l_2 = \psi_{L_2/F}(l) = \psi_{L_2/L_1}(l_1)$, 
the realization $L_2 \leftrightarrow_{l_2} L_2'$
produced from $F \leftrightarrow_l F'$ is the same as
the $L_2 \leftrightarrow_{\psi_{L_2/L_1}(l_1)} L_2'$
produced from the $L_1 \leftrightarrow_{l_1} L_1'$ in
turn produced from $F \leftrightarrow_l F'$.
\end{remark}
\subsubsection{Relating the multiplicative groups of close local fields}
\label{subsubsec: close local fields multiplicative groups}
Suppose that $F \leftrightarrow_l F'$, and that we have embeddings
$L_1 \hookrightarrow L_2 \hookrightarrow F^{\sep}$, with
$L_1/F$ and $L_2/F$ finitely ramified and at most $l$-ramified.
Assume that for $i = 1, 2$,
$L_i \hookrightarrow F^{\sep}$ and $L_i' \hookrightarrow {F'}^{\sep}$
are compatible embeddings, so that $L_1' \hookrightarrow {F'}^{\sep}$
factors through $L_2' \hookrightarrow {F'}^{\sep}$
(Remark \ref{rmk: properties E l(1) E' rephrasing}). Set
$l_i = \psi_{L_i/F}(l)$ for $i = 1, 2$.
The inclusion $L_1^{\times} \hookrightarrow L_2^{\times}$ induces
$L_1^{\times}/(1 + \p_{L_1}^{l_1}) \rightarrow
L_2^{\times}/(1 + \p_{L_2}^{l_2})$, because $l_2 =
\psi_{L_2/L_1}(l_1) \leq l_1 e(L_2/L_1)$,
by Remark \ref{rmk: at most l ramified extensions}\eqref{l(1)}.
Part of the datum defining the map $\Tr_{l_1} L_1 \rightarrow \Tr_{l_2} L_2$
(as in \cite[Section 1.4]{Del84}) is an $\mO_{L_1}/\p_{L_1}^{l_1}$-linear map
$\p_{L_1}/\p_{L_1}^{l_1 + 1}
\rightarrow (\p_{L_2}/\p_{L_2}^{l_2 + 1})^{\otimes e(L_2/L_1)}$ that
sends a generator of the source to a generator of the target as an
$\mO_{L_2}/\p_{L_2}^{l_2}$-module.
The description of $L_i^{\times}/(1 + \p_{L_i}^{l_i})$ ($i = 1, 2$) in terms
of $\Tr_{l_i} L_i$, given in \cite[Section 1.2]{Del84} 
as the group of homogeneous units of	
the graded $\mO_{L_i}/\p_{L_i}^{l_i}$-algebra
$\bigoplus_{n \in \ZZ} (\p_{L_i}/\p_{L_i}^{l_i + 1})^{\otimes n}$,
implies that the map $L_1^{\times}/(1 + \p_{L_1}^{l_1}) \rightarrow
L_2^{\times}/(1 + \p_{L_2}^{l_2})$ can be described
in terms of the extension $\Tr_{l_1} L_1 \rightarrow \Tr_{l_2} L_2$,
as obtained by putting together the various
$(\p_{L_1}/\p_{L_1}^{l_1 + 1})^{\otimes n}
\rightarrow (\p_{L_2}/\p_{L_2}^{l_2 + 1})^{\otimes e(L_2/L_1) n}$
(this does not use the completeness of $F$ or of the $L_i$).
Using the isomorphism $\Tr_{l_i} L_i \rightarrow \Tr_{l_i} L_i'$,
we get an isomorphism $L_i^{\times}/(1 + \p_{L_i}^{l_i}) \rightarrow
{L_i'}^{\times}/(1 + \p_{L_i'}^{l_i})$.
Now it is clear that, whenever $F \leftrightarrow_l F'$ and we
have $L_1 \leftrightarrow_{l_1} L_1'$ and
$L_2 \leftrightarrow_{l_2} L_2'$ as above,
we have an obvious commutative diagram:
\[
\xymatrix{
F^{\times}/(1 + \p_F^l) \ar[r] \ar[d] & L_1^{\times}/(1 + \p_{L_1}^{l_1})
\ar[r] \ar[d] & L_2^{\times}/(1 + \p_{L_2}^{l_2}) \ar[d] \\
{F'}^{\times}/(1 + \p_{F'}^l) \ar[r] & L_1'^{\times}/(1 + \p_{L_1'}^{l_1})
\ar[r] & {L_2'}^{\times}/(1 + \p_{L_2'}^{l_2})
}.
\]
By the discussion at the end of Subsubsection \ref{subsubsec: TrE},
if some $L_i/F$ is Galois, then the vertical arrow
in the above diagram involving $L_i$
is invariant under $\Gamma_{L_i/F} = \Gamma_{L_i'/F'}$.

\begin{notn} \label{notn: lleq L}
Let $F$ be a DVHF with perfect residue field,
and $L/F$ be an algebraic extension. For $r, l > 0$, we say that
$r \lleq_L l$ (or $l \ggeq_L r$),
if $L/F$ is finitely ramified, at most $l$-ramified,
and satisfies that $r \leq \psi_{L/F}(l)/e(L/F)$
(usually $l$ will be an integer for our purposes).
\end{notn}

Note that, if $F \leftrightarrow_l F'$,
$r \lleq_L l$, and $L \hookrightarrow F^{\sep}$
and $L' \hookrightarrow {F'}^{\sep}$ are
compatible extensions, then with $l(1)=\psi_{E/F}(l)$, the isomorphisms
$\mO_L/\p_L^{l(1)} \rightarrow \mO_{L'}/\p_{L'}^{l(1)}$ and
$L^{\times}/(1 + \p_L^{l(1)}) \rightarrow {L'}^{\times}/(1 + \p_{L'}^{l(1)})$
induce isomorphisms $\mO_L/\p_L^{\lceil e(L/F) r \rceil} \mO_L
\rightarrow \mO_{L'}/\p_{L'}^{\lceil e(L/F) r \rceil} \mO_{L'}$
and $L^{\times}/(1 + \p_L^{\lceil e(L/F) r \rceil})
\rightarrow {L'}^{\times}/(1 + \p_{L'}^{\lceil e(L/F) r \rceil})$.

\subsection{Notation related to tori over DVHFs}

\begin{notn} \label{notn: close local fields tori}
\begin{enumerate}[(i)]
\item Henceforth we will subscripting to indicate base-change: the
base-change of a scheme $\X/S$ to $S'$ will be denoted by $\X_{S'}$.

\item \label{character lattice}
For any torus $\T$ over a field $F$ with a chosen separable
closure $F^{\sep}$, we will denote by $X^*(\T)$ and $X_*(\T)$
respectively the character lattice and the cocharacter lattice
of the base-change $\T_{F^{\sep}}$ of $\T$ to $F^{\sep}$,
viewed with the obvious action of $\Gal(F^{\sep}/F)$ on these.
If an embedding $E \hookrightarrow F^{\sep}$ is understood from the context,
where $E/F$ a separable extension splitting $\T$, we may use it to view
each $\chi \in X^*(\T)$ as a homomorphism $\T_E \rightarrow \GG_m/E$,
$\chi(t)$ as an element of $E^{\times}$ for
$\chi \in X^*(\T)$ and $t \in \T(E)$, etc.

\item If $\T$ is a torus over a DVHF $F$ with perfect residue field,
its ft-N\'eron model and connected N\'eron model
(see \cite[Definition B.8.9]{KP23}) will be
denoted by $\mathcal{T}^{\ft}$ and $\mathcal{T}$, respectively
(thus, we follow \cite{Gan22} in writing $\mathcal{T}$
for the $\mathcal{T}^0$ of \cite{KP23}).

\item If $\T$ is a torus over a DVHF $F$ with perfect residue field,
then $\T(F)_b \subset \T(F)$ will denote its maximal bounded subgroup;
thus, $\mathcal{T}^{\ft}(\mO_F)$ identifies with $\T(F)_b$.

\item \label{minimal congruent filtration} Let $\T$ be a torus over a
DVHF $F$ with perfect residue field. We will consider three filtrations
of $\T(F)$:
\[\{\T(F)^{\naive}_r\}_{r \geq 0},\quad\{\T(F)^{\std}_r\}_{r \geq 0} \quad\text{and} \quad\{\T(F)_r\}_{r \geq 0}.\]
For $\{\T(F)^{\naive}_r\}_{r \geq 0}$ (defined in \eqref{eqn:naive}), we 
have for each $r \geq 0$: 
\begin{equation} \label{eqn:naive-filtration}
\T(F)^{\naive}_r = \{t \in \T(F)_b \mid \val_F(\chi(t) - 1) \geq r, \; \forall \,
\chi \in X^*(\T)\}
\end{equation}
(here, each $\chi(t)$ is valued in $F^{\sep}$, and we recall that
the normalized valuation $\val_F$ is canonically extended to algebraic
extensions of $F$). The filtration $\{\T(F)^{\std}_r\}_{r \geq 0}$, is the standard or Moy-Prasad
filtration of $\T(F)$ (see \cite[Definition B.5.1]{KP23}): 
\begin{equation} \label{eqn:standard-filtration}
\T(F)^{\std}_r = \{t \in \T(F)^0 \mid \val_F(\chi(t) - 1) \geq r, \; \forall \,
\chi \in X^*(\T)\},
\end{equation}
where $\T(F)^0$ is the Iwahori subgroup of $\T(F)$ as defined in \cite[Definition~2.5.13]{KP23}, 
a subgroup of finite index of $\T(F)_b$ (which can be strictly contained in $\T(F)_b$).
The filtration $\{\T(F)_r\}_{r \geq 0}$ is the minimal congruent filtration of $\T(F)$,
originally introduced by Yu in \cite{Yu15},
but interpreted as in \cite[Section B.10]{KP23}. 
We have $\T(F)^{\std}_r = \T(F)_0 \cap \T(F)^{\naive}_r$ for each $r \geq 0$,
and one can check that $\T(F)^{\naive}_0 = \T(F)_b$.
If $r = m$ is an integer, then $\T(F)_m$ is the group of $\mO_F$-points of what
is defined in \cite[Definition A.5.12]{KP23} as the $m$-th congruence
subgroup scheme of $\mathcal{T}$
(this is what the ``congruent'' of ``minimal congruent filtration'' refers to).
By \cite[Remark A.5.14]{KP23}, it also equals
$\ker(\mathcal{T}(\mO_F) \rightarrow \mathcal{T}(\mO_F/\p_F^m))$,
and hence is what \cite{Gan22} denotes as $T_m$.

\item \label{filtration subgroups strange} In the setting of
\eqref{minimal congruent filtration} above, if $E/F$ is a finitely
ramified separable algebraic extension,
we set $\T(E)^{\naive}_r := \T_E(E)^{\naive}_r$.
In slight contrast, if $E/F$ is unramified (algebraic), we let $\T(E)_r$
be $\mathcal{T}_r(\mO_E)$, where $\mathcal{T}_r$ is the $r$-th minimal
congruence filtration group scheme of $\T$ 
from \cite[Definition B.10.8(3)]{KP23} (thus, $\T(E)_r \cap \T(F)
= \T(F)_r$, but we do not know if $\T(E)_r$ equals
$\T_E(E)_r$).

\item \label{lleq T}
Let $\T$ be a torus over a DVHF $F$ with perfect residue field, and
let $r, l > 0$.
We say that $r \lleq_{\T} l$ (or $l \ggeq_{\T} r$), if
there exists a (finitely ramified at most $l$-ramified) extension $L/F$ that splits $\T$, such that $r \lleq_L l$
(see Notation \ref{notn: lleq L};
usually $l$ will be an integer for our purposes).

\item \label{close local fields tori}
If we write $(F, \T) \leftrightarrow_l (F', \T')$, we will
mean that $F \leftrightarrow_l F'$, and that
$\T, \T'$ are tori over $F, F'$ such that the actions of
$I_F^l$ on $X^*(\T)$ and $I_{F'}^l$ on $X^*(\T')$ are trivial
(in other words, $\T$ and $\T'$ are ``at most $l$-ramified''), and
that an isomorphism $X^*(\T) \rightarrow X^*(\T')$ has been
chosen that is equivariant for the actions
of $\Gamma_F/I_F^l = \Gamma_{F'}/I_{F'}^l$ (recall that the identification
$\Gamma_F/I_F^l = \Gamma_{F'}/I_{F'}^l$	is part of the datum
defining $F \leftrightarrow_l F'$).

\item \label{close local fields tori restriction}
Suppose $F \leftrightarrow_l F'$, and that
$E \hookrightarrow F^{\sep}$ and $E' \hookrightarrow {F'}^{\sep}$
are compatible embeddings, so (see Subsubsection \ref{subsubsec: TrE})
we have $E \leftrightarrow_{l(1)} E'$, where $l(1) = \psi_{E/F}(l)$.
Suppose further that $(F, \T) \leftrightarrow_l (F', \T')$
and $(E, \S) \leftrightarrow_{l(1)} (E', \S')$
extend our $F \leftrightarrow_l F'$ and $E \leftrightarrow_{l(1)} E'$,
respectively. Then we will implicitly work with realizations
$(E, \T_E) \leftrightarrow_{l(1)} (E', \T'_{E'})$
and $(F, \R := \Res_{E/F} \S) \leftrightarrow_l (F', \R' := \Res_{E'/F'} \S')$,
extending $E \leftrightarrow_{l(1)} E'$ and $F \leftrightarrow_l F'$,
obtained from the following identifications:
\begin{enumerate}[(a)]
\item $X^*(\T_E) = X^*(\T) \rightarrow X^*(\T') = X^*(\T'_{E'})$,
which is equivariant for $\Gamma_F/I_F^l = \Gamma_{F'}/I_{F'}^l$ and hence
for $\Gamma_E/I_E^{l(1)} = \Gamma_{E'}/I_{E'}^{l(1)}$; and
\item $X^*(\R) = \Ind_{\Gamma_E}^{\Gamma_F} X^*(\S) = 
\Ind_{\Gamma_E/I_E^{l(1)}}^{\Gamma_F/I_F^l} X^*(\S)
= \Ind_{\Gamma_{E'}/I_{E'}^{l(1)}}^{\Gamma_{F'}/I_{F'}^l} X^*(\S')
= \Ind_{\Gamma_{E'}}^{\Gamma_{F'}} X^*(\S') = X^*(\R')$,
where we recall that $\S$ and $\S'$ are at most $l(1)$-ramified,
use the discussion of
\eqref{GammaE GammaF GammaE' GammaF'} and \eqref{GammaEF GammaE'F'}
of Remark \ref{rmk: properties E l(1) E'}, and use
the canonical identifications
$X^*(\R) = \Ind_{\Gamma_E}^{\Gamma_F} X^*(\S)
= \ZZ[\Gamma_F] \otimes_{\ZZ[\Gamma_E]} X^*(\S)$
and $X^*(\R') = \Ind_{\Gamma_{E'}}^{\Gamma_{F'}} X^*(\S)$
reviewed in Remark \ref{rmk: restriction of scalars tori} below.
\end{enumerate}
\end{enumerate}
\end{notn}

As usual, the above notation will be adapted in obvious ways: for any torus
$\T_1$ over any DVHF $F^{\flat}$,
we will make sense of $\mathcal{T}^{\ft}_1$ and
$\T_1(F^{\flat})_m$, etc.

\begin{remark} \label{rmk: restriction of scalars tori}
Let $E/F$ be a finite separable field extension, $\S/E$ a
torus, and $\R := \Res_{E/F} \S$. Let us recall the canonical
realization $X^*(\R) =  \Ind_{\Gamma_E}^{\Gamma_F} X^*(\S)$, where
$\Gamma_F = \Gal(F^{\sep}/F)$ and $\Gamma_E = \Gal(F^{\sep}/E)$.
We have a ``universal'', surjective, homomorphism $\R_E \rightarrow \S_E$,
which, at the level of $A$-valued points for an $E$-algebra $A$, is
the map $\R(A) = \S(E \otimes_F A) \rightarrow \S(A)$, obtained by
applying the functor $\S$ to the multiplication map
$E \otimes_F A \rightarrow A$ of $E$-algebras (where $E \otimes_F A$
is an $E$-algebra via the first factor). This map has a well-known
universal property: for any multiplicative type group
scheme $\T$ over $F$, base-changing to $E$ followed by
composition with $\R_E \rightarrow \S_E$
gives us a functorial bijection between homomorphisms $\T \rightarrow \R$
and homomorphisms $\T_E \rightarrow \S_E$. Hence, composition with the injection
$X^*(\S) \rightarrow X^*(\R)$ dual to $\R_E \rightarrow \S_E$
gives a functorial identification
$\Hom_{\Gamma_E}(X^*(\S), X^*(\T)) \rightarrow \Hom_{\Gamma_F}(X^*(\R),
X^*(\T))$ (the notation $X^*(\T)$ extends to the case where $\T$
is a multiplicative type group scheme).
Hence Frobenius reciprocity gives an identification 
\[ X^*(\R) = \Ind_{\Gamma_E}^{\Gamma_F} X^*(\S)
= \ZZ[\Gamma_F] \otimes_{\ZZ[\Gamma_E]} X^*(\S). \]
\end{remark}

\begin{remark} \label{rmk: naive filtration}
We will often use without further comment the following nice property of
the naive filtration (Notation \ref{notn: close local fields tori}\eqref{minimal congruent filtration}): for any
injective homomorphism $\T_1 \hookrightarrow \T_2$ of tori over
a DVHF $F$ with perfect residue field, $r \geq 0$,
and a finitely ramified separable extension $E/F$,
$\T_1(F)^{\naive}_r = \T_1(F) \cap \T_2(F)^{\naive}_r
=\T_1(F) \cap \T_2(E)^{\naive}_{e(E/F) r}$.
For the first equality, use that $X^*(\T_2) \rightarrow X^*(\T_1)$
is surjective. For the second, recall that by the convention in
Notation \ref{notn: close local fields tori}\eqref{filtration subgroups strange},
each $\T_i(E)_s$ is defined using the normalized discrete valuation on $E$.
\end{remark}

\subsection{Weil restriction for tori across close local fields}

The following lemma is implicit in the last sentence
of \cite[Section 3.6]{CY01}.
\begin{lm} \label{lm: tori automorphisms}
Suppose $(F, \T) \leftrightarrow_l (F', \T')$.
Let $L \hookrightarrow F^{\sep}$ and $L' \hookrightarrow {F'}^{\sep}$
be compatible embeddings, with $L/F$ finite,
so that by Notation \ref{notn: close local fields tori}\eqref{close local fields tori restriction},
we have a realization $(F, \R) \leftrightarrow_l (F', \R')$,
where $\R = \Res_{L/F} \T_L$ and $\R' = \Res_{L'/F'} \T_{L'}'$.
Then the ``diagonal'' inclusions
$\T \hookrightarrow \R$ and $\T' \hookrightarrow \R'$
induce the same (necessarily surjective)
homomorphisms $X^*(\R) = X^*(\R') \rightarrow X^*(\T') = X^*(\T)$.
\end{lm}
\begin{proof}
$L'/F'$ is also finite, so the statement makes sense.
By the universal property of $\Res_{L/F}$, the ``diagonal'' map
$\T \hookrightarrow \R$ is the
unique homomorphism that, when base-changed to $L$ and composed with
$\R_L \rightarrow \T_L$, yields the identity map $\T_L \rightarrow \T_L$.
Dually, $X^*(\R) \rightarrow X^*(\T)$ is the unique
homomorphism of $\Gamma_F$-modules that, when viewed as a homomorphism
of $\Gamma_L$-modules and pre-composed with the
``universal'' $\Gamma_L$-module homomorphism $X^*(\T) \hookrightarrow X^*(\R)$,
yields the identity. In the previous sentence, we can replace
$\Gamma_F$ by $\Gamma_F/I_F^l = \Gamma_{F'}/I_{F'}^l$
and $\Gamma_E$ by $\Gamma_E/I_E^{l(1)} = \Gamma_{E'}/I_{E'}^{l(1)}$,
where $l(1) = \psi_{L/F}(l)$. Since an analogous assertion applies
for $(F', \T')$, we are done.
\end{proof}

\subsection{A review of some results of Chai and Yu}
\label{subsec: CY01}

Unfortunately, we will need to quote from the proofs, and not just
the lemmas, of \cite{CY01}. Therefore, we summarize what we will
need from that paper in this subsection.

\begin{notn} \label{notn: h}
\begin{enumerate}[(i)]
\item \label{h} If $\T$ is a torus over a DVHF $F$ with perfect residue field,
and $L/F$ is a finite Galois extension splitting $\T$,
we will denote by $h(F, \T, L)$ the nonnegative integer
$h(\mO_F, \mO_L, \Gamma_F, X_*(\T))$ defined
as in \cite[Section 8.1, just before the lemma]{CY01}. If $L/F$ is a
minimal splitting extension for $\T$, i.e.,
isomorphic to the fixed field of the kernel of
$\Gamma_F \rightarrow \Aut(X^*(\T))$, we write $h(F, \T) = h(F, \T, L)$.
By \cite[Lemma 8.1]{CY01},
whenever $(F, \T) \leftrightarrow_l (F', \T')$,
$L \hookrightarrow F^{\sep}$ and $L' \hookrightarrow {F'}^{\sep}$
are compatible extensions, and
$h(F, \T, L) < \lfloor \psi_{L/F}(l)/e(L/F) \rfloor$, we have
$h(F, \T, L) = h(F', \T', L')$ and $h(F, \T) = h(F, \T')$:
indeed, if we set
$h = h(F, \T, L)$ and $e = e(L/F)$, then $L'/F'$ splits
$\T'$ (minimally if $L$ does), and we have an identification
$\Tr_{e (h + 1)} L \cong \Tr_{e (h + 1)} L'$ over $\Tr_{h + 1} F
\cong \Tr_{h + 1} F'$, feeding into the hypothesis
of \cite[Lemma 8.1]{CY01}.

\item \label{Chai-Yu ResLFTL} Suppose $(F, \T) \leftrightarrow_l (F', \T')$.
Let $L/F$ be an at most $l$-ramified
finite Galois extension splitting $\T$,
and assume that $L \hookrightarrow F^{\sep}$
and $L' \hookrightarrow {F'}^{\sep}$ are compatible embeddings. Recall that
we have $(F, \R := \Res_{L/F} \T_L) \leftrightarrow_l
(F', \R' := \Res_{L'/F'} \T_{L'})$
(see Notation \ref{notn: close local fields tori}\eqref{close local fields tori restriction}).
Whenever $0 < m \lleq_L l$, \cite[the proof of Proposition 8.4(ii)]{CY01}
gives us an isomorphism
\begin{equation} \label{eqn: Chai-Yu ResLFTL}
\mathcal{R}^{\ft} \times_{\mO_F} \mO_F/\p_F^m \rightarrow
{\mathcal{R}'}^{\ft} \times_{\mO_{F'}} \mO_{F'}/\p_{F'}^m,
\end{equation}
(we recall that $\mathcal{R}^{\ft}$ (resp., ${\mathcal{R}'}^{\ft}$) is the finite type N\'eron model of $\R$ (resp., $\R'$)).
Explicitly, the identification
$\mO_L/\p_F^m \mO_L = \mO_{L'}/\p_{F'}^m \mO_{L'}$
(from $L \leftrightarrow_{\psi_{L/F}(l)} L'$) allows us
to identify both sides, at the level of $A$-valued points
for an $\mO_F/\p_F^m = \mO_{F'}/\p_{F'}^m$-algebra $A$, with
\[ A \mapsto \Hom(X^*(\T), ((\mO_L/\p_F^m \mO_L) \otimes_{\mO_F} A)^{\times})
= \Hom(X^*(\T'), ((\mO_{L'}/\p_{F'}^m \mO_{L'}) \otimes_{\mO_{F'}} A)^{\times}).
\]
\end{enumerate}
\end{notn}

\begin{remark} \label{rmk: Chai-Yu ResLFTL}
Assume the setting of Notation \ref{notn: h}\eqref{Chai-Yu ResLFTL}.
The following is from \cite[Section 8.1]{CY01}.
\begin{enumerate}[(i)]
\item \label{embedding R in an affine space}
Note that $\Res_{\mO_L/\mO_F} \GG_m$ has an obvious realization
as a closed subscheme of
the affine space $\mathbb{A}^{[L : F] + 1}/\mO_F$
associated to the free $\mO_F$-module 
$\mO_L \oplus \mO_F$ of rank $[L : F] + 1$ (with some chosen basis).
Using a {\em compatible} basis of $\mO_{L'} \oplus \mO_{F'}$, we have
$\Res_{\mO_{L'}/\mO_{F'}} \GG_m \hookrightarrow \mathbb{A}^{[L : F] + 1}/\mO_{F'}$.
Note that the obvious isomorphism
$(\mathbb{A}^{[L : F] + 1}/\mO_F) \times_{\mO_F} \mO_F/\p_F^m \rightarrow
(\mathbb{A}^{[L : F] + 1}/\mO_{F'}) \times_{\mO_{F'}} \mO_{F'}/\p_{F'}^m$
restricts to the isomorphism
$\Res_{\mO_L/\mO_F} \GG_m \times_{\mO_F} \mO_F/\p_F^{m + 1}
\rightarrow \Res_{\mO_{L'}/\mO_{F'}} \GG_m \times_{\mO_{F'}} \mO_{F'}/\p_{F'}^{m + 1}$
defined as in (in fact as a special case of)
Notation \ref{notn: h}\eqref{Chai-Yu ResLFTL}.

\item Choosing bases $\{\chi_i = \chi_i'\}$ of $X^*(\T) = X^*(\T')$,
we can realize $\mathcal{R}^{\ft}$, which sends an $\mO_F$-algebra
$A$ to $\Hom(X^*(\T), (\mO_L \otimes_{\mO_F} A)^{\times})$, as a product
of copies of $\Res_{\mO_L/\mO_F} \GG_m$ indexed by $\{\chi_i = \chi_i'\}$,
giving using \eqref{embedding R in an affine space}
an embedding $\mathcal{R}^{\ft} \hookrightarrow
\mathbb{A}^{\dim \T ([L : F] + 1)}/\mO_F$, as a closed subscheme.
Similarly, we get ${\mathcal{R}'}^{\ft} \hookrightarrow
\mathbb{A}^{\dim \T ([L : F] + 1})/\mO_{F'}$. It is immediate
that the obvious isomorphism
$(\mathbb{A}^{[L : F] + 1}/\mO_F) \times_{\mO_F} \mO_F/\p_F^m \rightarrow
(\mathbb{A}^{[L : F] + 1}/\mO_{F'}) \times_{\mO_{F'}} \mO_{F'}/\p_{F'}^m$
restricts to the isomorphism
$\mathcal{R}^{\ft} \times_{\mO_F} \mO_F/\p_F^m
\rightarrow {\mathcal{R}'}^{\ft} \times_{\mO_{F'}} \mO_{F'}/\p_{F'}^m$
of Notation \ref{notn: h}\eqref{Chai-Yu ResLFTL}.
\end{enumerate}
\end{remark}

The following is one of the main results of \cite{CY01}:
\begin{thm}[Chai and Yu] \label{thm: CY01}
Let $(F, \T) \leftrightarrow_l (F', \T')$.
Suppose $m$ is a positive integer such that $m + 3 h(F, \T, L) \lleq_L l$, for 
a fixed at most $l$-ramified finite Galois extension $L/F$ splitting $\T$ (which
exists since $(F, \T) \leftrightarrow_l (F', \T')$).
Let $L \hookrightarrow F^{\sep}$ and $L' \hookrightarrow {F'}^{\sep}$
be compatible embeddings, so that
$(F, \R := \Res_{L/F} \T_L) \leftrightarrow_l
(F', \R' := \Res_{L'/F'} \T_{L'})$. Set $h = h(F, \T, L)$.
Then there exists a unique isomorphism $\mcT^{\ft} \times_{\mO_F} \mO_F/\p_F^m
\rightarrow {\mcT'}^{\ft} \times_{\mO_{F'}} \mO_{F'}/\p_{F'}^m$
satisfying the following property: for some (or equivalently, any)
compatible embeddings $\tilde F \hookrightarrow  F^{\sep}$
and $\tilde F' \hookrightarrow {F'}^{\sep}$, where
$\tilde F/F$ is a maximal unramified extension,
evaluating this isomorphism
(resp., \eqref{eqn: Chai-Yu ResLFTL}) at
$\mO_{\tilde F}/\p_{\tilde F}^m = \mO_{\tilde F'}/\p_{\tilde F'}^m$
(resp., $\mO_{\tilde F}/\p_{\tilde F}^{m + h}
= \mO_{\tilde F'}/\p_{\tilde F'}^{m + h}$)
gives us the right-most (resp., the left-most) vertical arrow of
an obvious commutative diagram
\begin{equation} \label{eqn: CY01 main result}
\xymatrix{
\R(\tilde F)_b/\R(\tilde F)_{m + h} \ar[d]
& \T(\tilde F)_b/\T(\tilde F)^{\naive}_{m + h} \ar@{_{(}->}[l] \ar[r] \ar[d]
& \T(\tilde F)_b/\T(\tilde F)_m \ar[d] \\
\R'(\tilde F')_b/\R'(\tilde F')_{m + h} &
\T'(\tilde F')_b/\T'(\tilde F')^{\naive}_{m + h}
\ar@{_{(}->}[l] \ar[r] & \T'(\tilde F')_b/\T'(\tilde F')_m
}.
\end{equation}
Here, we have identified $\mcT^{\ft}(\mO_{\tilde F})$ with
$\T(\tilde F)_b$, and hence $\mcT^{\ft}(\mO_{\tilde F}/\p_{\tilde F}^m)$
with $\T(\tilde F)_b/\T(\tilde F)_m$, etc. That the left vertical arrow
induces the middle vertical arrow, and the inclusions
$\T(\tilde F)^{\naive}_{m + h} \subset \T(\tilde F)_m$
and $\T'(\tilde F')^{\naive}_{m + h} \subset \T'(\tilde F')_m$
needed to make sense of the right horizontal arrows in the two rows,
are part of the assertion.
\end{thm}

The above description is not present right at the point of
statement of \cite[Theorem~8.5]{CY01}, but can be assembled from
various parts of \cite{CY01}. To help the reader do so, we will give
more references later in this subsection. To this end, we now make
some preparation.

The following assertion is a special case of
\cite[Proposition 4.2]{CY01}, understood
using \cite[the proof of Proposition 8.4(iii)]{CY01}:
\begin{pro}[Chai and Yu] \label{pro: CY01 Proposition 4.2}
Let $F \leftrightarrow_l F'$. Let $\X, \X'$ be
smooth algebraic schemes over $\mO_F, \mO_{F'}$, $\W \subset \X_{\kappa_F}$
and $\W' \subset \X_{\kappa_{F'}}'$ closed smooth subschemes,
and $\Y, \Y'$ the dilatations of $\W, \W'$ on $\X, \X'$.
Assume that, for some $0 \leq m < l$, we are given an isomorphism
$\X \times_{\mO_F} \mO_F/\p_F^{m + 1} \rightarrow \X' \times_{\mO_{F'}}
\mO_{F'}/\p_{F'}^{m + 1}$ over $\mO_F/\p_F^{m + 1} = \mO_{F'}/\p_{F'}^{m + 1}$,
that on tensoring with $\kappa_F = \kappa_{F'}$ identifies $\W$ with $\W'$.
Then there is a unique isomorphism
$\Y \times_{\mO_F} \mO_F/\p_F^m \rightarrow \Y' \times_{\mO_{F'}} \mO_{F'}/\p_{F'}^m$
over $\mO_F/\p_F^m = \mO_{F'}/\p_{F'}^m$ with the following property: 
for some (or equivalently, any) compatible embeddings
$\tilde F \hookrightarrow F^{\sep}$
and $\tilde F' \hookrightarrow {F'}^{\sep}$, where
$\tilde F/F$ is a maximal unramified extension,
it maps the image of $y \in \Y(\mO_{\tilde F}) \subset \X(\mO_{\tilde F})$
in $\Y \times_{\mO_F} \mO_F/\p_F^m$
(by which we mean its image in $\Y(\mO_{\tilde F}/\p_{\tilde F}^m)$)
to that of $y' \in \Y'(\mO_{\tilde F'}) \subset \X(\mO_{\tilde  F'})$
in $\Y' \times_{\mO_{\tilde F'}} \mO_{\tilde F'}/\p_{\tilde F'}^m$ whenever
$\X \times_{\mO_F} \mO_F/\p_F^{m + 1} \rightarrow \X' \times_{\mO_{F'}}
\mO_{F'}/\p_{F'}^{m + 1}$ does so.
\end{pro}
\begin{proof}[References for the proof]
First we address the uniqueness, assuming the existence.
It is immediately seen that
$\tilde F'/F'$ is automatically a maximal unramified extension.
Given $y \in \Y(\mO_{\tilde F}) \subset \X(\mO_{\tilde F})$,
its image in $\X(\mO_{\tilde F}/\p_{\tilde F}^{m + 1})
= \X'(\mO_{\tilde F'}/\p_{\tilde F'}^{m + 1})$ has image
in $\X(\kappa_{\tilde F}) = \X'(\kappa_{\tilde F'})$ that belongs to
$\W = \W'$. Thus, any $y' \in \X'(\mO_{\tilde F'})$ that lifts this image
(such $y'$ exist as $\X'$ is smooth)
belongs to $\Y'(\mO_{\tilde F'}) \subset \X'(\mO_{\tilde F'})$.
Thus, $\Y \times_{\mO_F} \mO_F/\p_F^m \rightarrow \Y' \times_{\mO_{F'}} \mO_{F'}/\p_{F'}^m$
is pinned down on the image of $\Y(\mO_{\tilde F})$ in
$\Y(\mO_{\tilde F}/\p_{\tilde F}^m)$. Now the uniqueness follows
by the schematic density of this
set of points, as asserted in \cite[Lemma 8.5.1]{CY01}, which applies
since $\Y$ is smooth (see \cite[Section 3.2, Proposition~3]{BLR}, or \cite[Lemma A.5.10]{KP23} or \cite[Proposition~2.16]{MRR}).
Note that this argument does not use the completeness assumption
of \cite[Section 8]{CY01}.

Now we discuss the existence. The non-dependence
on $\tilde F \hookrightarrow F^{\sep}$ and
$\tilde F' \hookrightarrow {F'}^{\sep}$ is easy.
Write $\X = \spec C$ and $\X' = \spec C'$, for an $\mO_F$-algebra $C$ and
an $\mO_{F'}$-algebra $C'$. Then $\X \times_{\mO_F} \mO_F/\p_F^{m + 1}
\rightarrow \X' \times_{\mO_{F'}} \mO_{F'}/\p_{F'}^{m + 1}$
is dual to the inverse of an isomorphism
$C/\p_F^{m + 1} C \rightarrow C'/\p_{F'}^{m + 1} C'$. Let
the subscheme $\W \subset \X_{\kappa_F} \subset \X$
be defined by an ideal $I_W$ generated by $\varpi_F, g_1, \dots, g_s$, where
$\varpi_F \in \mO_F$ is a uniformizer, and $g_1, \dots, g_s \in C$.
Then the subscheme $\W' \subset \X'$ is defined by
the ideal $I_{W'}$ generated by
$\varpi_{F'}, g_1', \dots, g_s'$, where $\varpi_{F'} \in \mO_{F'}$
is a uniformizer matching $\varpi_F$ under
$\mO_F/\p_F^{m + 1} = \mO_{F'}/\p_{F'}^{m + 1}$, and where $g_i'$ matches
$g_i$ under $C/\p_F^{m + 1} C \rightarrow C'/\p_{F'}^{m + 1} C'$
for $1 \leq i \leq s$.

Over the isomorphism
$C/\p_F^{m + 1} C = C'/\p_{F'}^{m + 1} C'$ of
$\mO_F/\p_F^{m + 1} = \mO_{F'}/\p_{F'}^{m + 1}$-algebras
lies an isomorphism
\begin{equation} \label{eqn: for CY01 Proposition 4.2}
\frac{C[x_1, \dots, x_s]}{(\varpi_F x_1 - g_1, \dots, \varpi_F x_s - g_s)}
\otimes_{\mO_F} \mO_F/\p_F^{m + 1}
\rightarrow
\frac{C'[x_1, \dots, x_s]}{(\varpi_{F'} x_1 - g_1', \dots, \varpi_{F'} x_s - g_s)}
\otimes_{\mO_{F'}} \mO_{F'}/\p_{F'}^{m + 1},
\end{equation}
induced by sending $x_i$ to $x_i$ for each $i$. The ring
$C[x_1, \dots, x_s]/(\varpi_F x_1 - g_1, \dots, \varpi_F x_s - g_s)$,
{\em modulo its $\varpi_F^{\infty}$-torsion}, is the coordinate ring
$\mO_F[\Y]$ of $\Y$ (see,  e.g., 
\cite[the discussion of Remark A.5.9]{KP23}).
A similar assertion applies with $F'$ in place of $F$. While
$\mO_F[\Y] \otimes_{\mO_F} \mO_F/\p_F^m$ and
$\mO_{F'}[\Y'] \otimes_{\mO_{F'}} \mO_{F'}/\p_{F'}^m$ are quotients
of the left-hand side and the right-hand side of
the isomorphism \eqref{eqn: for CY01 Proposition 4.2}, it is not obvious that
\eqref{eqn: for CY01 Proposition 4.2} induces an isomorphism between
these quotients. Nevertheless, it does so, by \cite[Proposition 4.2]{CY01},
as we now explain. 

We will superscript with $[m]$ to denote base-change to $\mO_F/\p_F^m$ or $\mO_{F'}/\p_{F'}^m$. In particular, $I_W^{[m]}
= I_W \otimes_{\mO_F} \mO_F/\p_F^m$.  Since $I_W \supset \varpi_F \mO_F[X]$, we get
$\varpi_F^m I_W\supset \varpi_F^{m + 1} \mO_F[X]$. This allows
us to identify $I_W^{[m]} =I_W/\varpi_F^m I_W$ with $I_{W'}^{[m]} =I_{W'}/\varpi_{F'}^m I_{W'}$.
One considers the following diagram:
\[
\xymatrix{
(C[x_1, \dots, x_s]/(\varpi_F x_1 - g_1, \dots, \varpi_F x_s - g_s))^{[m]}
\ar[r] \ar[d]\! &\! \left( \bigoplus_{t \geq 0} \Sym^t_{C^{[m]}}
I_W^{[m]} \right)_{(\varpi_F)}
\ar[r] \ar[d]\! &\! (\mO_F[\Y])^{[m]} \ar[d] \\
(C'[x_1, \dots, x_s]/(\varpi_{F'} x_1 - g_1', \dots, \varpi_{F'} x_s - g'_s))^{[m]}
\ar[r] \!&\! \left( \bigoplus_{t \geq 0} \Sym^t_{{C'}^{[m]}} I_{W'}^{[m]} \right)_{(\varpi_{F'})}
\ar[r] \!&\! (\mO_{F'}[\Y'])^{[m]}
}
\]
where the top middle term refers to the homogeneous localization
of $\bigoplus_{t \geq 0} \Sym^t_{C^{[m]}} I_W^{[m]}$ at the homogeneous element
of degree $1$ given by the image of $\varpi_F \in I_W$,
so that its spectrum is an open subset of
the scheme $\mathrm{Bl}'(\X, \W) \times_{\mO_F} \mO_F/\p_F^m$ defined as in
\cite[Section 4.2.1]{CY01}.
Thus, the right square and its commutativity,
and the fact that the right vertical arrow is an isomorphism, follow
from the canonicity description of \cite[Section 4.2.1]{CY01}. The top
left arrow maps the image of $x_i$ to $g_i/\varpi_F$ for each $i$,
and the bottom left arrow is similar, so that the commutativity
of the left square is clear.

Thus, \eqref{eqn: for CY01 Proposition 4.2} indeed quotients to an isomorphism
$\mO_F[\Y] \otimes_{\mO_F} \mO_F/\p_F^m \rightarrow
\mO_{F'}[\Y'] \otimes_{\mO_{F'}} \mO_{F'}/\p_{F'}^m$, i.e., an
isomorphism $\Y \times_{\mO_F} \mO_F/\p_F^m \rightarrow
\Y' \times_{\mO_{F'}} \mO_{F'}/\p_{F'}^m$. It remains to show that this
morphism is as described in the proposition, so let $y, y'$ be
as in it. It suffices to show that $f(y) = f'(y')$ in
$\mO_{\tilde F}/\p_{\tilde F}^m = \mO_{\tilde F'}/\p_{\tilde F'}^m$
whenever $f$ and $f'$ match under \eqref{eqn: for CY01 Proposition 4.2}.
Without loss of generality, $f$ and $f'$ are both represented
by $x_i$ for some $1 \leq i \leq s$. But since we have
$\varpi_F f = g_i$ and $\varpi_{F'} f' = g_i'$, this follows from
the fact that $g_i$ and $g_i'$ match each other under
$C/\p_F^{m + 1} C \rightarrow C'/\p_{F'}^{m + 1} C'$, as do
$\varpi_F$ and $\varpi_{F'}$ under
$\mO_F/\p_F^{m + 1} \rightarrow \mO_{F'}/\p_{F'}^{m + 1}$,
so that $g_i(y)$ and $g_i'(y')$ match under
$\mO_{\tilde F}/\p_{\tilde F}^{m + 1} \rightarrow \mO_{\tilde F'}/\p_{\tilde F'}^{m + 1}$.
\end{proof}

\begin{proof}[References for the proof of Theorem \ref{thm: CY01}]
The uniqueness follows as in the proof of
Proposition \ref{pro: CY01 Proposition 4.2}.
Using the latter argument of \cite[Remark 8.6]{CY01} ($\mO_F/\p_F^m$, 
$\mathcal{T}^{\ft} \times_{\mO_F} \mO_F/\p_F^m$
and $\mathcal{R}^{\ft} \times_{\mO_F} \mO_F/\p_F^m$ remain unchanged
when $F$ is replaced by its completion), together with 
\cite[Proposition 2.3.4(2)]{KP23}, from which it follows that
$\mcT^{\ft}(\mO_{\tilde F}) = \T(\tilde F)_b$ is dense in
the analogous group associated to the completion of $F$
(so that the source and
target of middle vertical arrow are unchanged when we replace $F$ by
its completion), we may and shall assume
that $F$ and $F'$ are complete. This is being done so that
we may use results from \cite[Section 8]{CY01}.

We will denote by $\underline \T^i$ and $\underline \R^i$
the integral models $\underline T^i$ for $\T$
and $\underline R^i$ for $\R$ from \cite[Sections 3.2, 3.4, 3.6]{CY01}.
By definition, $\underline \R^0 := \mathcal{R}^{\ft}$, where
$\R := \Res_{L/F} \T_L$, and $\underline \T^0$
(see \cite[Section 3.6]{CY01}) is the schematic closure of $\T$ in
$\underline \R^0$. Thus, $\underline \T^0$
is the standard model of $\T$ in the sense of \cite[Section B.4]{KP23}.
If $\underline \T^i$ and $\underline \R^i$ are defined, then
$\underline \T^{i + 1}$ (resp., $\underline \R^{i + 1}$)
is the dilatation of a smooth subscheme
$\Z^i \subset \underline \T^i \times_{\mO_F} \kappa_F$ on $\underline \T^i$
(resp., $\W^i \subset \underline \R^i \times_{\mO_F} \kappa_F$ on
$\underline \R^i$). Further, $\underline \T^i$ can be realized as
the schematic closure of $\T$ in $\underline \R^i$
(\cite[Lemma 3.5]{CY01}).
We have similar objects $(\underline \T')^i$ and
$(\underline \R')^i$ associated to $\T'/F'$.

By \cite[Corollary 8.2.4]{CY01}, or rather its proof, thanks to the
inequalities $\psi_{L/F}(l)/e(L/F) > 2h$
and $\psi_{L/F}(l)/e(L/F) - h \geq m + h$, the isomorphism
$\underline \R^0 \times_{\mO_F} \mO_F/\p_F^{m + h} \rightarrow (\underline \R')^0 \times_{\mO_{F'}} \mO_{F'}/\p_{F'}^{m + h}$
given by \eqref{Chai-Yu ResLFTL} of Notation \ref{notn: h}
restricts to an isomorphism
$\underline \T^0 \times_{\mO_F} \mO_F/\p_F^{m + h} \rightarrow
(\underline \T')^0 \times_{\mO_{F'}} \mO_{F'}/\p_{F'}^{m + h}$,
as the following two sentences explain. Recall the chains of inclusions of
closed subschemes $\underline \T^0 \subset \underline \R^0
\subset \mathbb{A}^{\dim \T \cdot ([L : F] + 1)}/\mO_F$
and $(\underline{\T'})^0 \subset (\underline{\R'})^0
\subset \mathbb{A}^{\dim \T \cdot ({[L : F] + 1)}}/\mO_{F'}$
from \cite[Section 8.1 and the beginning of Section 8.3]{CY01},
reviewed in Remark \ref{rmk: Chai-Yu ResLFTL}.
The isomorphism
$\underline \R^0 \times_{\mO_F} \mO_F/\p_F^{m + h} \rightarrow (\underline \R')^0 \times_{\mO_{F'}} \mO_{F'}/\p_{F'}^{m + h}$
from \eqref{Chai-Yu ResLFTL} of Notation \ref{notn: h} has been observed
to be a restriction of the obvious identification of affine spaces
$\mathbb{A}^{\dim \T \cdot ([L : F] + 1)}/(\mO_F/\p_F^{m + h})
= \mathbb{A}^{\dim \T \cdot ([L : F] + 1)}/(\mO_{F'}/\p_{F'}^{m + h})$,
and the isomorphism
$\underline \T^0 \times_{\mO_F} \mO_F/\p_F^{m + h} \rightarrow (\underline \T')^0 \times_{\mO_{F'}} \mO_{F'}/\p_{F'}^{m + h}$
is constructed in \cite[Corollary 8.2.4]{CY01} to also satisfy this property.

Since $\underline \T^0(\mO_{\tilde F}) \subset \underline \R^0(\mO_{\tilde F})$
identifies with the inclusion $\T(\tilde F)_b \subset \R(\tilde F)_b$,
and $\T(\tilde F)^{\naive}_{m + h} = \T(\tilde F) \cap
\R(\tilde F)_{m + h}$, and similar assertions hold for $F'$,
the above paragraph implies that we indeed have
a well-defined middle vertical arrow
of \eqref{eqn: CY01 main result} making the
left square of that diagram commute.

A positive integer $\delta$ is introduced in \cite[Section 8.5]{CY01},
and it is observed that $\delta \leq h$. As observed there
(the invocation of \cite[Lemma 5.5]{CY01})
and using our assumption that $m + 3h \leq l$, working with $F'$
instead of $F$ does not change $\delta$. Thus, as observed at
\cite[the beginning of Section 8.5.2]{CY01}, $\mcT^{\ft}$ equals
$\underline \T^h$ and ${\mcT'}^{\ft}$ equals $(\underline{\T}')^h$.

$\underline \R^h$ and $(\underline \R')^h$
are obtained from $\underline \R^0$ and $(\underline \R')^0$ by
a series of $h$ dilatations. At the $i$-th step, one
inductively assumes given an identification
$\underline \R^{i - 1} \times_{\mO_F} \mO_F/\p_F^{m + h + 1 - i}
\rightarrow (\underline \R')^{i - 1} \times_{\mO_{F'}} \mO_{F'}/\p_{F'}^{m + h + 1 - i}$,
and dilatates $\underline \R^{i - 1}$ and $(\underline \R')^{i - 1}$
with respect to the same subscheme of
$\underline \R^{i - 1} \times_{\mO_F} \kappa_F
= \underline (\R')^{i - 1} \times_{\mO_{F'}} \kappa_{F'}$
to get $\underline \R^i$ and $(\underline \R')^i$, yielding
by Proposition \ref{pro: CY01 Proposition 4.2} an
isomorphism
$\underline \R^i \times_{\mO_F} \mO_F/\p_F^{m + h - i}
\rightarrow (\underline \R')^i \times_{\mO_{F'}} \mO_{F'}/\p_{F'}^{m + h - i}$
(see \cite[Proposition 8.4]{CY01}).
It therefore follows from $h$-many applications
of Proposition \ref{pro: CY01 Proposition 4.2}
that there is a unique isomorphism
$\underline \R^h \times_{\mO_F} \mO_F/\p_F^m
\rightarrow (\underline \R')^h \times_{\mO_{F'}} \mO_{F'}/\p_{F'}^m$
that maps the image of $y \in \underline \R^h(\mO_{\tilde F}) \subset
\mathcal{R}^{\ft}(\mO_{\tilde F})$ to
that of $y' \in (\underline \R')^h(\mO_{\tilde F'}) \subset
{\mathcal{R}'}^{\ft}(\mO_{\tilde F'})$ whenever the image
of $y$ maps to that of $y'$ under
$\underline \R^0 \times_{\mO_F} \mO_F/\p_F^{m + h} \rightarrow (\underline \R')^0 \times_{\mO_{F'}} \mO_{F'}/\p_{F'}^{m + h}$ (these applications are justified
by the fact that the $\underline \R^i$ and the $(\underline \R')^i$
are smooth, unlike the $\underline \T^i$ and the $(\underline \T')^i$).

It is argued in \cite[Section 8.5.2]{CY01} that
$\underline \R^h \times_{\mO_F} \mO_F/\p_F^m
\rightarrow (\underline \R')^h \times_{\mO_{F'}} \mO_{F'}/\p_{F'}^m$
restricts to an isomorphism
$\mcT^{\ft} \times_{\mO_F} \mO_F/\p_F^m
= \underline \T^h \times_{\mO_F} \mO_F/\p_F^m
\rightarrow (\underline \T')^h \times_{\mO_{F'}} \mO_{F'}/\p_{F'}^m
= {\mcT'}^{\ft} \times_{\mO_{F'}} \mO_{F'}/\p_{F'}^m$.
Hence, under this isomorphism, the image of
$y \in \mcT^{\ft}(\mO_{\tilde F}) \subset \mathcal{R}^{\ft}(\mO_{\tilde F})$
maps to that of
$y' \in {\mcT'}^{\ft}(\mO_{\tilde F'})
\subset {\mathcal{R}'}^{\ft}(\mO_{\tilde F'})$
whenever, under the middle vertical arrow of
\eqref{eqn: CY01 main result}, the image of $y$ maps to that
of $y'$. This makes sense of the right square of
\eqref{eqn: CY01 main result} and gives its commutativity.
\end{proof}

Note that in the above proof, the implicit assertions such as that
$\T(\tilde F)^{\naive}_{m + h} \subset \T(\tilde F)_m$ have
been implicitly taken care of.
This could be compared with the easier containment
$\T(\tilde F)^{\naive}_r \supset \T(\tilde F)_r$
(use \cite[Propositions B.10.4 and B.10.13]{KP23}),
which can be proper for tori that are not ``weakly induced''.
In any case, let us record the containment
$\T(\tilde F)^{\naive}_{m + h} \subset \T(\tilde F)_m$,
since it applies in greater generality 
(without assuming $m + 3 h(F, \T) \lleq_{\T} l$),
and does not need the strength of \cite[Proposition 4.2]{CY01}
(i.e., of Proposition \ref{pro: CY01 Proposition 4.2}):
\begin{lm} \label{lm: MP intersection}
Let $\T$ be a torus over a DVHF $F$ with perfect residue field.
Assume that $h := h(F, \T)
< \lfloor \psi_{L/F}(l)/e(L/F) \rfloor$ for
some at most $l$-ramified finite Galois extension $L/F$ splitting $\T$.
Then for any positive integer $m$, we have
$\T(\tilde F)^{\naive}_{m + h} = \T(\tilde F)^{\std}_{m + h}
\subset \T(\tilde F)_m$.
\end{lm}
\begin{remark} \label{rmk: MP intersection}
In the situation of the above lemma, we claim that whenever
$F \leftrightarrow_{m + h} F'$, and
$\tilde F \hookrightarrow F^{\sep}$ and $\tilde F'
\hookrightarrow {F'}^{\sep}$ are compatible embeddings, where $\tilde F/F$ is
a maximal unramified extension,
we also have an analogous containment
$\T'(\tilde F')^{\naive}_{m + h} = \T'(\tilde F')^{\std}_{m + h}
\subset \T'(\tilde F')_m$:
this is because $h(F, \T) = h(F', \T')$ by the discussion
of Notation \ref{notn: h}\eqref{h}.
\end{remark}
\begin{proof}[Proof of Lemma \ref{lm: MP intersection}]
It suffices to show that $\T(\tilde F)^{\naive}_{m + h} \subset \T(\tilde F)_m$,
since it will then follow that
\[ \T(\tilde F)^{\std}_{m + h}
\subset \T(\tilde F)^{\naive}_{m + h} = \T(\tilde F)^{\naive}_{m + h} \cap
\T(\tilde F)_m \subset \T(\tilde F)^{\naive}_{m + h} \cap
\T(\tilde F)_0 = \T(\tilde F)^{\std}_{m + h}. \]

Let $\R = \Res_{L/F} \T$.
Since the connected N\'eron model
$\mathcal{T}$ is obtained from ${\mathcal{T}}^{\ft}$ by 
dilatating with respect to the identity component of the special fiber,
it follows that the group $\T(\tilde F)_m
= \ker(\mathcal{T}(\mO_{\tilde F}) \rightarrow
\mathcal{T}(\mO_{\tilde F}/\p_{\tilde F}^m))$
also equals
$\ker(\mathcal{T}^{\ft}(\mO_{\tilde F}) \rightarrow
\mathcal{T}^{\ft}(\mO_{\tilde F}/\p_{\tilde F}^m))$.
The kernel of $\mathcal{R}^{\ft}(\mO_{\tilde F}) \rightarrow
\mathcal{R}^{\ft}(\mO_{\tilde F}/\p_{\tilde F}^{m + h})$ equals
$\R(\tilde F)_{m + h}$, which equals $\R(\tilde F)^{\naive}_{m + h}$
by \cite[Corollary B.10.13]{KP23} (and the equality
$\mathcal{R} = \mathcal{R}^{\ft}$).

We will use notation from the above (outline of) proof.
It suffices to show that, if $t \in \underline \T^0(\mO_{\tilde F})$
has trivial image in $\mathcal{R}^{\ft}(\mO_{\tilde F}/\p_{\tilde F}^{m + h})
= \underline \R^0(\mO_{\tilde F}/\p_{\tilde F}^{m + h})$,
then $t$ has trivial image in
$\underline \R^h(\mO_{\tilde F}/\p_{\tilde F}^m)$
(and hence belongs to $\ker(\mathcal{T}^{\ft}(\mO_{\tilde F}) \rightarrow
\mathcal{T}^{\ft}(\mO_{\tilde F}/\p_{\tilde F}^m))$).
Indeed, one shows by induction on $0 \leq i \leq h$ that
$t$ has trivial image in
$\underline \R^i(\mO_{\tilde F}/\p_{\tilde F}^{m + h - i})$.
The induction step is as in the arguments
around and below \eqref{eqn: for CY01 Proposition 4.2}
in the proof of Proposition \ref{pro: CY01 Proposition 4.2}:
if the coordinate ring of $\R^i$ is $C$, write
the coordinate ring of $\R^{i + 1}$ as a quotient of
$C[x_1, \dots, x_s]/(\varpi_F x_1 - g_1,
\dots, \varpi_F x_s - g_s)$, and note that the induction hypothesis
that $f(t) \in \varpi_F^{m + h - i} \mO_{\tilde F}$ for all
$f \in C$ implies $f(t) \in \varpi_F^{m + h - i - 1} \mO_{\tilde F}$
for all $f \in C[x_1, \dots, x_s]/(\varpi_F x_1 - g_1,
\dots, \varpi_F x_s - g_s)$.
\end{proof}

Note that the proof of the above lemma shows that we can replace
$h = h(F, \T)$ by any $i$ such that
$\underline \T^i = \mathcal{T}^{\ft}$, with $\underline \T^i$
as in the proof of Theorem \ref{thm: CY01} (e.g., $i$ could be
the $\delta$ of that proof).

\section{Standard, congruent and Chai-Yu isomorphisms}
\label{sec: standard congruent Chai-Yu}

\subsection{The definition of standard and congruent isomorphisms}

\begin{notn} \label{notn: standard congruent Chai-Yu}
Let $(F, \T) \leftrightarrow_l (F', \T')$, and let
$r > 0$. Elements $t \T(F)_r^{\naive} \in \T(F)/\T(F)^{\naive}_r$ and
$t' \T'(F')_r^{\naive} \in \T'(F')/\T'(F')^{\naive}_r$
(or by abuse of notation, $t \in \T(F)$ and $t' \in \T'(F')$)
are said to be standard correspondents of each other if
for some (and hence by Lemma \ref{lm: standard isomorphism well-defined}
below, any) compatible embeddings $L \hookrightarrow F^{\sep}$ and
$L' \hookrightarrow {F'}^{\sep}$, such that $L$ splits $\T$ and
$r \lleq_L l$ (see Notation \ref{notn: lleq L}), the following
holds: for every $\chi = \chi' \in X^*(\T') = X^*(\T)$,
$\chi(t)$ and $\chi'(t')$ have images that match under
the isomorphism $L^{\times}/(1 + \p_L^{\lceil e(L/F) r \rceil})
\rightarrow {L'}^{\times}/(1 + \p_{L'}^{\lceil e(L/F) r \rceil})$
(described below Notation \ref{notn: lleq L}).
\end{notn}

\begin{remark}
\begin{enumerate}[(i)]
\item Note that the condition in Notation \ref{notn: standard congruent Chai-Yu}
does not change if $t$ or $t'$ is replaced by another element
of $t \T(F)^{\naive}_r$ or $t' \T'(F')^{\naive}_r$ (this follows from
the definition of the naive filtration subgroups $\T(F)^{\naive}_r$ and
$\T'(F')^{\naive}_r$).

\item We will see in Lemma \ref{lm: standard isomorphism well-defined} below
that, in the setting of Notation \ref{notn: standard congruent Chai-Yu},
every element of $\T(F)/\T(F)^{\naive}_r$ has either a unique
standard correspondent in $\T'(F')/\T'(F')^{\naive}_r$,
or none at all (and vice versa).
This will sometimes be used without further comment in what follows.
\end{enumerate}
\end{remark}

\begin{df} \label{df: standard congruent Chai-Yu}
Let $(F, \T) \leftrightarrow_l (F', \T')$.
\begin{enumerate}[(i)]
\item \label{standard isomorphism}
Let $r > 0$ be a positive real number.
An isomorphism 
$\T(F)/\T(F)^{\naive}_r \rightarrow \T'(F')/\T'(F')^{\naive}_r$
of abelian groups is said to be a standard isomorphism if
it maps every element of its source to a standard correspondent of
it.

\item \label{congruent isomorphism} Let $m$ be a positive integer.
An isomorphism $\T(F)/\T(F)_m \rightarrow \T'(F')/\T'(F')_m$ is
said to be a congruent isomorphism if for some
(or equivalently by Lemma \ref{lm: congruent isomorphism}\eqref{congruent isomorphism well-defined} below, any)
compatible embeddings $\tilde F \hookrightarrow F^{\sep}$ and
$\tilde F' \hookrightarrow {F'}^{\sep}$, where $\tilde F/F$ is
a maximal unramified extension, and some $r > 0$ such that
$\T(\tilde F)^{\naive}_r \subset \T(\tilde F)_m$ and
$\T'(\tilde F')^{\naive}_r \subset \T'(\tilde F')_m$,
the given isomorphism is a restriction of an isomorphism
$\T(\tilde F)/\T(\tilde F)_m \rightarrow \T'(\tilde F')/\T'(\tilde F')_m$,
which in turn is induced by a standard isomorphism
$\T(\tilde F)/\T(\tilde F)^{\naive}_r \rightarrow \T'(\tilde F')/\T'(\tilde F')^{\naive}_r$ associated to $(\tilde F, \T_{\tilde F}) \leftrightarrow_l
(\tilde F', \T'_{\tilde F'})$.

\item \label{Chai-Yu isomorphism}
Let $m$ be a positive integer.
An isomorphism $\mcT^{\ft} \times_{\mO_F} \mO_F/\p_F^m \rightarrow
{\mcT'}^{\ft} \times_{\mO_{F'}} \mO_{F'}/\p_{F'}^m$ is
said to be a Chai-Yu isomorphism if for some
(or equivalently by
Lemma \ref{lm: congruent isomorphism}\eqref{congruent isomorphism well-defined} below, any) compatible embeddings
$\tilde F \hookrightarrow F^{\sep}$ and
$\tilde F' \hookrightarrow {F'}^{\sep}$, where $\tilde F/F$ is a
maximal unramified extension,
and some $r > 0$ such that $\T(\tilde F)^{\naive}_r \subset \T(\tilde F)_m$ and
$\T'(\tilde F')^{\naive}_r \subset \T'(\tilde F')_m$,
the isomorphism
$\T(\tilde F)_b/\T(\tilde F)_m \rightarrow \T'(\tilde F')_b/\T'(\tilde F')_m$
obtained by evaluating the given isomorphism at
$\mO_{\tilde F}/\p_{\tilde F}^m = \mO_{\tilde F'}/\p_{\tilde F'}^m$
is induced
by a ``restricted standard isomorphism''
$\T(\tilde F)_b/\T(\tilde F)^{\naive}_r
\rightarrow \T'(\tilde F')_b/\T'(\tilde F')^{\naive}_r$, namely, one
that maps each element of its source to a standard correspondent of it,
for the realization $(\tilde F, \T_{\tilde F}) \leftrightarrow_l
(\tilde F', \T'_{\tilde F'})$.
\end{enumerate}
\end{df}

The following remark helps make sense of the above definition.
\begin{remark} \label{rmk: standard congruent Chai-Yu}\
\begin{enumerate}[(i)]
\item Note that, by definition, a standard isomorphism
$\T(F)/\T(F)^{\naive}_r \rightarrow \T'(F')/\T'(F')^{\naive}_r$
does not exist unless 
$0 \leq r  \lleq _{\T} l$, i.e.,
$r \leq \psi_{L/F}(l)/e(L/F)$ for some finitely ramified at most
$l$-ramified separable extension $L/F$ splitting $\T$. Similarly,
a congruent isomorphism
$\T(\tilde F)/\T(\tilde F)_m \rightarrow \T'(\tilde F')/\T'(\tilde F')_m$
or a Chai-Yu isomorphism $\mcT^{\ft} \times_{\mO_F} \mO_F/\p_F^m \rightarrow
{\mcT'}^{\ft} \times_{\mO_{F'}} \mO_{F'}/\p_{F'}^m$
does not exist unless there exist $r \geq 0$ such
that $l \ggeq_{\T} r$, $\T(\tilde F)^{\naive}_r \subset
\T(\tilde F)_m$ and $\T'(\tilde F')^{\naive}_r \subset \T'(\tilde F')_m$
(these latter conditions force $r \geq m$, since
$\T(F)_r \subset \T(F)^{\naive}_r$).

\item In Definition \ref{df: standard congruent Chai-Yu}\eqref{congruent isomorphism},
the notion of restriction from $\T(\tilde F)/\T(\tilde F)_m$ to
$\T(F)/\T(F)_m$ makes sense, because
$\T(\tilde F)_m \cap \T(F) = \T(F)_m$
(see Notation \ref{notn: close local fields tori}\eqref{filtration subgroups strange}).
For $L/F$ finitely ramified separable with $e = e(L/F)$,
since $\T(L)^{\naive}_{e r} \cap \T(F) = \T(F)^{\naive}_r$
(Remark \ref{rmk: naive filtration}), we may restrict from
$\T(L)/\T(L)^{\naive}_{e r}$ to $\T(F)/\T(F)^{\naive}_r$.

\item To relate the definition of a Chai-Yu isomorphism to
isomorphisms constructed in \cite{CY01}, see
Proposition \ref{pro: Chai-Yu isomorphism CY01}
below.
\end{enumerate}
\end{remark}

\subsection{Some first properties of standard isomorphisms}

\begin{lm} \label{lm: standard isomorphism well-defined}
Let $(F, \T) \leftrightarrow_l (F', \T')$, and let $r > 0$. 
If $t \in \T(F)$ and $t' \in \T'(F')$ satisfy
the conditions in the definition of a standard correspondent
(Notation \ref{notn: standard congruent Chai-Yu})
with respect to some choice of $L \hookrightarrow F^{\sep}$
and $L' \hookrightarrow {F'}^{\sep}$ (with $L/F$ finitely ramified,
such that $L$ splits $\T$ and $r \lleq_L l$), then it satisfies those
conditions with respect to any other such choice, say
$L_1 \hookrightarrow F^{\sep}$ and
$L_1' \hookrightarrow {F'}^{\sep}$. Moreover, every element
of $\T(F)/\T(F)^{\naive}_r$ has either a unique standard corresondent
in $\T'(F')/\T'(F')^{\naive}_r$, or none at all.
\end{lm}
\begin{proof}
We first prove the former assertion.
Assume first that $L/F$ is minimal, i.e.,
isomorphic to the fixed field of $\ker(\Gamma_F \rightarrow \Aut(X^*(\T)))$.
In this case, $L \hookrightarrow F^{\sep}$ factors as the composite
of $L_1 \hookrightarrow F^{\sep}$ and some $F$-algebra embedding
$L \hookrightarrow L_1$. Moreover, $L'$ is automatically a minimal
splitting extension for $\T'$, and we similarly get
$L' \hookrightarrow L_1' \hookrightarrow {F'}^{\sep}$. In
this case, the lemma follows from the commutative diagram
of Subsubsection \ref{subsubsec: close local fields multiplicative groups}.
A similar argument, with the roles of $L$ and $L_1$ swapped, reduces
the case of general $L$ to that of minimal $L$: note
that replacing $L$ with a smaller splitting extension
increases $\psi_{L/F}(l)/e(L/F)$, and hence
preserves the relation $r \lleq_L l$.

To see the uniqueness assertion, setting $e = e(L/F) = e(L'/F')$
and recalling that $\T(F)/\T(F)^{\naive}_r \subset \T(L)/\T(L)^{\naive}_{e r}$,
use the identifications
\[
\T(L)/\T(L)^{\naive}_{e r}
\overset{\cong}{\rightarrow}
\Hom(X^*(\T), L^{\times}/(1 + \p_L^{\lceil e r \rceil}))
\overset{\cong}{\rightarrow}
\Hom(X^*(\T'), {L'}^{\times}/(1 + \p_{L'}^{\lceil e r \rceil}))
\\ \overset{\cong}{\rightarrow} \T'(L')/\T'(L')^{\naive}_{e r},
\]
where the first map is $t \mapsto (\chi \mapsto \chi(t))$, and note
that $t$ and $t'$ are standard correspondents if and only if they
define the same element of this identified object.
\end{proof}

\begin{remark} \label{rmk: unique standard isomorphism}
Assume the setting of the above lemma.
\begin{enumerate}[(i)]
\item \label{unique standard isomorphism} If $\T/F$ is split,
a unique standard isomorphism
$\T(F)/\T(F)^{\naive}_r \rightarrow \T'(F')/\T'(F')^{\naive}_r$ exists:
\[ \T(F)/\T(F)^{\naive}_r
\rightarrow
\Hom(X^*(\T), F^{\times}/1 + \p_F^{\lceil r \rceil})
=
\Hom(X^*(\T'), {F'}^{\times}/1 + \p_{F'}^{\lceil r \rceil})
\rightarrow \T'(F')/\T'(F')^{\naive}_r. \]

\item \label{unique standard isomorphism general T}
For general $\T$, there is either a unique
standard isomorphism \[\T(F)/\T(F)^{\naive}_r \rightarrow
\T'(F')/\T'(F')^{\naive}_r,\] or none at all. Indeed, choose 
compatible embeddings $L \hookrightarrow F^{\sep}$ and
$L' \hookrightarrow {F'}^{\sep}$, with $L/F$ a splitting extension
for $\T$ such that $r \lleq_L l$. Set
$e = e(L/F) = e(L'/F')$. If the standard isomorphism
$\T(L)/\T(L)^{\naive}_{e r}
\rightarrow \T'(L')/\T'(L')^{\naive}_{e r}$
associated to $(L, \T_L) \leftrightarrow_{\psi_{L/F}(l)} (L', \T'_{L'})$
(made sense of using Notation \ref{notn: close local fields tori}\eqref{close local fields tori restriction})
restricts to an isomorphism
$\T(F)/\T(F)^{\naive}_r \rightarrow \T'(F')/\T'(F')^{\naive}_r$,
then this restriction defines a standard
isomorphism.
If not, there is no standard isomorphism
$\T(F)/\T(F)^{\naive}_r \rightarrow \T'(F')/\T'(F')^{\naive}_r$.
\end{enumerate}
\end{remark}

\begin{lm} \label{lm: standard isomorphism functorial}
Standard isomorphisms have the following functoriality.
Let $(F, \T_i) \leftrightarrow_l (F', \T_i')$ for $i = 1, 2$,
with the same underlying $F \leftrightarrow_l F'$, and let $r > 0$.
Let $f : \T_1 \rightarrow \T_2$ and $f' : \T_1' \rightarrow \T_2'$ be
homomorphisms inducing the same homomorphism
$X^*(\T_2') = X^*(\T_2) \rightarrow X^*(\T_1) = X^*(\T_1')$ at the level
of character lattices. Then, if standard isomorphisms 
$\T_i(F)/\T_i(F)^{\naive}_r \rightarrow \T_i'(F')/\T_i'(F')^{\naive}_r$
exist for $i = 1, 2$, they are the vertical arrows of the following
commutative diagram:
\[
\xymatrix{
\T_1(F)/\T_1(F)^{\naive}_r \ar[r]^{f} \ar[d] & \T_2(F)/\T_2(F)^{\naive}_r \ar[d] \\
\T_1'(F')/\T_1'(F')^{\naive}_r \ar[r]^{f'} & \T_2'(F')/\T_2'(F')^{\naive}_r
}.
\]
\end{lm}
\begin{remark}
The lemma would be immediate if we had $r \lleq_L l$ for
some $L/F$ splitting both $\T_1$ and $\T_2$.
We are not making this assumption, hence the longer proof.
\end{remark}
\begin{proof}[Proof of Lemma \ref{lm: standard isomorphism functorial}]
The assertion of the lemma is equivalent to the following statement:
if $t_1 \in \T_1(F)$ is a standard correspondent of $t_1' \in \T_1'(F')$,
then $t_2 := f(t_1) \in \T_2(F)$ is a standard correspondent of
$t_2' := f'(t_1') \in \T_2'(F')$.

First, we consider a slightly different situation. We make
the stronger assumption that there exists a finite separable
extension $L/F$, splitting both $\T_1$ and $\T_2$, such
that $r \lleq_L l$. However, we do not impose the assumption that
either of the standard isomorphisms is well-defined.
Choose compatible embeddings
$L \hookrightarrow F^{\sep}$ and $L' \hookrightarrow {F'}^{\sep}$.

Under these assumptions, the following claim is formal:
if $t_1 \in \T_1(F)$ is a standard correspondent of
$t_1' \in \T_1'(F')$, then $f(t_1) \in \T_2(F)$ is a standard
correspondent of $f'(t_1') \in \T_2'(F')$.

Now consider the general case.
Let $\T_3 \subset \T_2$
be the image of $f$, and $\T_3' \subset \T_2'$ that of $f'$.
Thus, $X^*(\T_3) \subset X^*(\T_1)$ is the image of
$X^*(\T_2) \rightarrow X^*(\T_1)$, and similarly with $X^*(\T_3')$.
We have an obvious realization $(F, \T_3) \leftrightarrow_l (F', \T_3')$.

We have $r \lleq_{\T_3} l$, since any splitting field of
either of $\T_1$ or $\T_2$ also splits $\T_3$. However,
we cannot assume that there exists a standard isomorphism
$\T_3(F)/\T_3(F)^{\naive}_r \rightarrow \T_3'(F')/\T_3'(F')^{\naive}_r$.

Nevertheless, if $t_1 \in \T_1(F)$ is a standard correspondent
of $t_1' \in \T_1'(F')$, then letting $t_2, t_3$ be the images
of $t_1$ in $\T_2(F)$ and $\T_3(F)$,
and $t_2'$ and $t_3'$ those of $t_1'$ in $\T_2'(F')$ and $\T_3'(F')$:
\begin{itemize}
\item $t_3$ is a standard correspondent of $t_3'$: apply the above
claim with $\T_1 \rightarrow \T_2$ and $\T_1' \rightarrow \T_2'$ replaced
by $\T_1 \rightarrow \T_3$ and $\T_1' \rightarrow \T_3'$; and
\item hence $t_2$ is a standard correspondent of $t_2'$: apply
the above claim with $\T_1 \rightarrow \T_2$ and
$\T_1' \rightarrow \T_2'$ replaced
by $\T_3 \rightarrow \T_2$ and $\T_3' \rightarrow \T_2'$.
\end{itemize}
As observed earlier, this implies the lemma.
\end{proof}

\begin{lm} \label{lm: standard isomorphism decreasing r}
Let $(F, \T) \leftrightarrow_l (F', \T')$, and let $r > 0$.
Assume that there is a standard isomorphism
$\T(F)/\T(F)^{\naive}_r \rightarrow \T'(F')/\T'(F')^{\naive}_r$. Then
for all $0 < s \leq r$, this standard isomorphism
induces a standard isomorphism $\T(F)/\T(F)^{\naive}_s
\rightarrow \T'(F')/\T'(F')^{\naive}_s$,
which further restricts to a ``restricted standard isomorphism''
$\T(F)_b/\T(F)^{\naive}_s \rightarrow \T'(F')_b/\T'(F')^{\naive}_s$,
uniquely characterized by the fact that it sends each element of
its source to a standard correspondent of it.
\end{lm}
\begin{proof}
Easy, using the following two facts.
First, whenever $L \leftrightarrow_{l_1} L'$,
the resulting isomorphism
$L^{\times}/(1 + \p_L^{l_1}) \rightarrow {L'}^{\times}/(1 + \p_{L'}^{l_1})$
induces, for all $0 < s \leq l_1$, an isomorphism
$L^{\times}/(1 + \p_L^{\lceil s \rceil})
\rightarrow {L'}^{\times}/(1 + \p_{L'}^{\lceil s \rceil})$.
Secondly, given $t \in \T(F)$ and a finitely ramified separable
extension $L \hookrightarrow F^{\sep}$ that splits $\T$,
we have $t \in \T(F)_b$ if and only if for all $\chi \in X^*(\T)$,
the element $\chi(t)$ of $L^{\times}$ belongs to $\mO_L^{\times}$.
\end{proof}

\begin{lm} \label{lm: standard isomorphism functorial 2}
Standard isomorphisms also have the following functoriality.
Let $(F, \T) \leftrightarrow_l (F', \T')$.
For $i = 1, 2$, let $E_i \hookrightarrow F^{\sep}$ and
$E_i' \hookrightarrow {F'}^{\sep}$ be compatible embeddings, and assume that
there is a factorization
$E_1 \hookrightarrow F^{\sep} = (E_2 \hookrightarrow F^{\sep})
\circ (E_1 \hookrightarrow E_2)$, giving an analogous
factorization $E_1' \hookrightarrow {F'}^{\sep}
= (E_2' \hookrightarrow {F'}^{\sep})
\circ (E_1' \hookrightarrow E_2')$. Then, if for $i = 1, 2$,
standard isomorphisms
$\T(E_i)/\T(E_i)^{\naive}_{r_i} \rightarrow \T'(E_i')/\T'(E_i')^{\naive}_{r_i}$
exist (as before, using
Notation \ref{notn: close local fields tori}\eqref{close local fields tori restriction}
to make sense of the
$(E_i, \T_{E_i}) \leftrightarrow_{\psi_{E_i/F}(l)} (E_i', \T'_{E_i'})$),
and $r_2 \leq e(E_2/E_1) r_1$, then 
these are the vertical arrows of the following commutative diagram, whose
horizontal arrows are induced by $\T(E_1) \rightarrow \T(E_2)$
and $\T'(E_1') \rightarrow \T'(E_2')$:
\[
\xymatrix{
\T(E_1)/\T(E_1)^{\naive}_{r_1}
\ar[r] \ar[d] & \T(E_2)/\T(E_2)^{\naive}_{r_2} \ar[d] \\
\T'(E_1')/\T'(E_1')^{\naive}_{r_1}
\ar[r] & \T'(E_2')/\T'(E_2')^{\naive}_{r_2}
}.
\]
Note that we do not assume the existence of any standard isomorphism
\[\T(F)/\T(F)^{\naive}_r \rightarrow \T'(F')/\T'(F')^{\naive}_r.\]
\end{lm}
\begin{proof}
Since $r_2 \leq e(E_2/E_1) r_1$, the horizontal arrows are
well-defined. By Lemma \ref{lm: standard isomorphism decreasing r},
we may decrease $r_1$ if necessary, to assume that $r_2 = e(E_2/E_1) r_1$.

Choose compatible embeddings $L \hookrightarrow F^{\sep}$
and $L' \hookrightarrow {F'}^{\sep}$ for the realization
$E_2 \leftrightarrow_{l_2} E_2'$ (see Subsubsection \ref{subsubsec: TrE}),
splitting $\T$, and such that $r_2 \lleq_{\T_{E_2}} l_2$.
Thus, we have
\[ r_1 = r_2/e(E_2/E_1) \leq \psi_{L/E_2}(l_2)/e(L/E_1)
= \psi_{L/E_2}(\psi_{E_2/F}(l))/e(L/E_1) = \psi_{L/E_1}(l_1)/e(L/E_1). \]
Thus, we also have $r_1 \lleq_{\T_{E_1}} l_1$.
Note that $L \hookrightarrow F^{\sep}$
and $L' \hookrightarrow {F'}^{\sep}$
are also compatible embeddings for $F \leftrightarrow_l F'$ (since
their stabilizers in $\Gamma_F/I_F^l = \Gamma_{F'}/I_{F'}^l$
are contained in $\Gamma_{E_2}/I_{E_2}^{l_2} = \Gamma_{E_2'}/I_{E_2'}^{l_2}$),
and hence also for $E_1 \leftrightarrow_{l_1} E_1'$.
All these descriptions give
the same realization $L \leftrightarrow_{l^{\circ}} L'$,
where $l^{\circ} = \psi_{L/E_2}(l_2) = \psi_{L/E_1}(l_1)$
(see Remark \ref{rmk: properties E l(1) E' rephrasing}).

Thus, by Remark \ref{rmk: unique standard isomorphism},
both the vertical arrows are obtained by restriction from
the standard isomorphism
$\T(L)/\T(L)^{\naive}_{e(L/E_i) r_i}
\rightarrow \T'(L')/\T'(L')^{\naive}_{e(L/E_i) r_i}$
(with $e(L/E_i) r_i$ independent of $i$), and the lemma follows. 
\end{proof}

\begin{cor} \label{cor: standard isomorphism functorial 2}
Let $(F, \T) \leftrightarrow_l (F', \T')$.
Let $E \hookrightarrow F^{\sep}$
and $E' \hookrightarrow {F'}^{\sep}$ be compatible embeddings,
and let $r > 0$. Assume that $E/F$ is Galois
(and hence so is $E'/F'$), and that a standard isomorphism
$\T(E)/\T(E)^{\naive}_r \rightarrow \T'(E')/\T'(E')^{\naive}_r$,
associated to $(E, \T_E) \leftrightarrow_{\psi_{E/F}(l)} (E', \T'_{E'})$,
exists. Then this isomorphism is equivariant for the action
of $\Gamma_{E/F} = \Gamma_{E'/F'}$.
\end{cor}
\begin{proof}
This is a special case of Lemma \ref{lm: standard isomorphism functorial 2}.
\end{proof}

\subsection{Further properties of standard isomorphisms}

In this subsection, we prove less obvious properties of standard
isomorphisms: their existence when $F$ is strictly Henselian,
and compatibility with Kottwitz homomorphisms and the
local Langlands correspondence.

\begin{pro} \label{pro: standard isomorphism strictly Henselian}
Let $(F, \T) \leftrightarrow_l (F', \T')$, and let
$0 < r \lleq_{\T} l$. If further $F$ is strictly Henselian
(and hence so is $F'$), a standard isomorphism
$\T(F)/\T(F)^{\naive}_r \rightarrow \T'(F')/\T'(F')^{\naive}_r$ exists.
\end{pro}
\begin{proof}
Choose a finite separable extension $L/F$, splitting $\T$,
such that $r \lleq_L l$. Without loss of generality, $L/F$ is minimal such
(making $L$ smaller increases $\psi_{L/F}(l)/e(L/F)$),
and hence Galois. Let $L \hookrightarrow F^{\sep}$
and $L' \hookrightarrow {F'}^{\sep}$ be compatible embeddings, so
$L \leftrightarrow_{\psi_{L/F}(l)} L'$. Abbreviate $e  := e(L/F)$.

We have a standard isomorphism
$\T(L)/\T(L)^{\naive}_{er} \rightarrow \T'(L')/\T'(L')^{\naive}_{er}$,
since $L$ splits $\T$ and $e r \leq \psi_{L/F}(l)$.
It suffices to show that this isomorphism restricts to an isomorphism 
$\T(F)/\T(F)^{\naive}_r\to\T'(F')/\T'(F')^{\naive}_r$.

The isomorphism
$\T(L)/\T(L)^{\naive}_{er} \rightarrow \T'(L')/\T'(L')^{\naive}_{er}$
is equivariant for $\Gamma_{L/F} = \Gamma_{L'/F'}$,
by Corollary \ref{cor: standard isomorphism functorial 2}, and hence
for the action of the product $N_{L/F} = N_{L'/F'}$ 
of the elements of $\Gamma_{L/F} = \Gamma_{L'/F'}$. Therefore,
it suffices to show that $N_{L/F} : \T(L) \rightarrow \T(F)$
and $N_{L'/F'} : \T'(L') \rightarrow \T'(F')$ are surjective,
or equivalently, that the analogous maps
$\Res_{L/F} \T_L \rightarrow \T$ and $\Res_{L'/F'} \T'_{L'} \rightarrow \T'$
are surjective respectively at the levels of $F$-rational points and
$F'$-rational points. This is well-known:
the kernel $\T_0$ of $N_{L/F} :
\Res_{L/F} \T_L \rightarrow \T$
is connected, and $H^1(F, \T_0) = 0$ by
\cite[Corollary 2.3.7 and Lemma 2.5.4]{KP23}, since $F$ is
strictly Henselian and $\kappa_F$ is perfect.
\end{proof}

\begin{pro} \label{pro: standard isomorphism Kottwitz homomorphism}
Standard isomorphisms have the following compatibility with
Kottwitz homomorphisms. Let $(F, \T) \leftrightarrow_l (F', \T')$,
and assume that a standard isomorphism
$\T(F)/\T(F)^{\naive}_r \rightarrow \T'(F')/\T'(F')^{\naive}_r$ exists.
Then it is the left vertical arrow of the following commutative diagram, whose
horizontal arrows are the relevant Kottwitz homomorphisms, and
whose right vertical arrow is an isomorphism induced by the
$\Gamma_F/I_F^l = \Gamma_{F'}/I_{F'}^l$-equivariant identification
$X_*(\T) = X_*(\T')$.
\[
\xymatrix{
\T(F)/\T(F)^{\naive}_r \ar[r] \ar[d] & (X_*(\T)_{I_F})^{\Gamma_{\kappa_F}}
\ar[d] \\
\T'(F')/\T'(F')^{\naive}_r \ar[r] & (X_*(\T')_{I_{F'}})^{\Gamma_{\kappa_{F'}}}
}
\]
 (subscripting with $I_F$ or $I_{F'}$ stands for taking
the group of $I_F$-coinvariants or $I_{F'}$-coinvariants).
\end{pro}
\begin{proof}
To make sense of the right vertical arrow, use that
the identification $\Gamma_F/I_F^l = \Gamma_{F'}/I_{F'}^l$
restricts to an identification $I_F/I_F^l = I_{F'}/I_{F'}^l$
and induces an identification $\Gamma_{\kappa_F} = \Gamma_{\kappa_{F'}}$.

If $\T$ is split, the lemma follows from 
Remark \ref{rmk: unique standard isomorphism}\eqref{unique standard isomorphism}
and the following factorization of the Kottwitz homomorphism:
\[
\T(F) = \Hom(X^*(\T), F^{\times}) \overset{\val}{\rightarrow}
\Hom(X^*(\T), \ZZ) = X_*(\T). \]

Choose compatible embeddings $\tilde F \hookrightarrow F^{\sep}$ and
$\tilde F' \hookrightarrow {F'}^{\sep}$, with $\tilde F/F$ a maximal
unramified extension.
By Proposition \ref{pro: standard isomorphism strictly Henselian},
we have a standard isomorphism \[\T(\tilde F)/\T(\tilde F)^{\naive}_r
\rightarrow \T'(\tilde F')/\T'(\tilde F')^{\naive}_r\] associated
to $(\tilde F, \T_{\tilde F}) \leftrightarrow_l (\tilde F', \T'_{\tilde F'})$,
which by Lemma \ref{lm: standard isomorphism functorial 2} restricts to the  
standard isomorphism $\T(F)/\T(F)^{\naive}_r \rightarrow \T'(F')/\T'(F')^{\naive}_r$.
Since the Kottwitz homomorphism too is defined by restricting from
the maximal unramified extension, we may now
replace $(F, \T) \leftrightarrow_l (F', \T')$ by
$(\tilde F, \T_{\tilde F}) \leftrightarrow_l (\tilde F', \T'_{\tilde F'})$,
and assume that $F$ is strictly Henselian.

Now we are in the setting of
Proposition \ref{pro: standard isomorphism strictly Henselian}.
Let $L \hookrightarrow F^{\sep}$ and $L' \hookrightarrow {F'}^{\sep}$
be as in the proof of that proposition, and set $e = e(L/F)$.
We have a diagram
\[
\xymatrix{
& \T(L)/\T(L)^{\naive}_{e r} \ar[dl]_{N_{L/F}}
\ar@{.>}[dd] \ar[rr] & &
\T'(L')/\T'(L')^{\naive}_{e r} \ar[dd]
\ar[dl]_{N_{L'/F'}}
\\
\T(F)/\T(F)^{\naive}_r \ar[rr] \ar[dd] & & \T'(F')/\T'(F')^{\naive}_r \ar[dd] & \\
& X_*(\T_L) \ar@{.>}[rr] \ar@{.>}[ld] & & X_*(\T_{L'}') \ar[ld] \\
X_*(\T)_{I_F} \ar[rr] & & X_*(\T')_{I_{F'}}
},
\]
whose `top face' is given by the proof of
Proposition \ref{pro: standard isomorphism strictly Henselian},
vertical arrows are the appropriate Kottwitz homomorphisms,
and the `bottom' face consists of obvious maps.

The proof of Proposition \ref{pro: standard isomorphism strictly Henselian}
also gives the commutativity of the `top face'. The two `side faces' are
commutative by \cite[Lemma 11.1.4]{KP23}.
The `hind face' (the four terms involving $L$) is commutative
since the split case of the lemma is known. The `bottom face' is clearly
commutative.

Since $N_{L/F}$ is surjective, as we saw in
the proof of Proposition \ref{pro: standard isomorphism strictly Henselian},
it is now easy to see that the `front face' is commutative
as well, which is exactly the commutative diagram
the lemma seeks to prove.
\end{proof}

\begin{pro} \label{pro: standard isomorphism LLC}
Standard isomorphisms have the following compatibility with the $\mathrm{LLC}$.
Let $(F, \T) \leftrightarrow_l (F', \T')$, and assume that
a standard isomorphism $\T(F)/\T(F)^{\naive}_r \rightarrow \T'(F')/\T'(F')^{\naive}_r$ exists
(in particular, $r \lleq_{\T} l$). Assume further that
$F$ and $F'$ are complete with finite residue field. Then we have the following
commutative diagram analogous to \eqref{eqn: LLC compatibility},
whose left vertical arrow is induced by the given isomorphism,
and the right vertical arrow is induced by the isomorphism
$\hat \T =
X^*(\T) \otimes \CC^{\times} = X^*(\T') \otimes \CC^{\times} = \hat \T'$
of modules over $W_F/I_F^l = W_{F'}/I_{F'}^l$:
\begin{equation} \label{eqn: standard isomorphism LLC}
\xymatrix{
\Hom(\T(F)/\T(F)^{\naive}_r, \CC^{\times})
\ar@{^{(}->}[r]^{\ \ \ \ \ \ \mathrm{LLC}} \ar[d]
& H^1(W_F/I_F^l, \hat \T) \ar[d]_{\cong}^{\Del_l} \\
\Hom(\T'(F')/\T'(F')^{\naive}_r, \CC^{\times})
 \ar@{^{(}->}[r]^{\ \ \ \ \ \ \ \ \mathrm{LLC}} & H^1(W_{F'}/I_{F'}^l, \hat \T')
}.
\end{equation}
\end{pro}
\begin{proof}
First suppose $\T = \GG_m/F$, and $\T' = \GG_m/F'$ compatibly. In this
case, the lemma is easy to see: briefly, the local class field theory
map $W_F \rightarrow F^{\times}$ sends $I_F^l$ to $1 + \p_F^l$, making
the horizontal arrows well-defined (as $r \leq l$), and the
commutativity of the square follows from the analogous statement
for local class field theory, proved by Deligne
(see \cite[Proposition 3.6.1]{Del84}). From this, the case where
$\T$ is split follows, so we will consider the split case as known.

Let $L/F$ be a finite (say minimal, and hence) Galois extension splitting
$\T$, with $r \lleq_L l$. Let $\chi \colon \T(F)/\T(F)^{\naive}_r \rightarrow
\CC^{\times}$ have image $\chi' \colon \T'(F')/\T'(F')^{\naive}_r \rightarrow \CC^{\times}$
under the left vertical arrow.

Write $e = e(L/F)$. Since $\T(F)/\T(F)^{\naive}_r \hookrightarrow \T(L)/\T(L)^{\naive}_{e r}$,
we can extend $\chi$ to a homomorphism
$\chi_1 \colon \T(L)/\T(L)^{\naive}_{er} \rightarrow \CC^{\times}$. It is clear that
the homomorphism $\chi_1' : \T'(L')/\T'(L')^{\naive}_{er} \rightarrow \CC^{\times}$
obtained by transferring $\chi_1$ under the standard isomorphism
$\T(L)/\T(L)^{\naive}_{er} \rightarrow \T'(L')/\T'(L')^{\naive}_{er}$ (which exists
since $L$ splits $\T$ and $r \lleq_L l$) has $\chi'$ as its restriction to
$\T'(F')/\T'(F')^{\naive}_r \subset \T'(L')/\T'(L')^{\naive}_{er}$
(see Remark \ref{rmk: unique standard isomorphism}).

Let $\varphi_{\chi} \in H^1(W_F, \hat \T), \varphi_{\chi'} \in
H^1(W_{F'}, \hat \T'), \varphi_{\chi_1} \in H^1(W_L, \hat \T_L)$,
$\varphi_{\chi_1'} \in H^1(W_{L'}, \hat \T'_{L'})$ be the local
Langlands parameters of $\chi, \chi', \chi_1, \chi_1'$.
Write $l_1 = \psi_{L/F}(l) = \psi_{L'/F'}(l)$.

Since $\T_L$ is split, and since $e r \leq \psi_{L/F}(l) = l_1$,
it follows from the split case (discussed at the beginning
of this proof) that
$\varphi_{\chi_1} \in H^1(W_L/I_L^{l_1}, \hat \T) \subset H^1(W_L, \hat \T)$
(the inflation map), that $\varphi_{\chi_1'} \in H^1(W_{L'}/I_{L'}^{l_1})
\subset H^1(W_{L'}, \hat \T')$, and that $\varphi_{\chi_1'}$ is
the image of $\varphi_{\chi_1}$ under the obvious isomorphism
$H^1(W_L/I_L^{l_1}, \hat \T) \rightarrow H^1(W_{L'}/I_{L'}^{l_1}, \hat \T')$.

We will show that $\varphi_{\chi}$ is the image of
$\varphi_{\chi_1}$ under the corestriction
map $H^1(W_L, \hat \T) \rightarrow H^1(W_F, \hat \T)$: since $L/F$
is Galois, this follows from the construction of the local Langlands
correspondence for tori in \cite[Section 7.7]{Yu09} (see especially
the definition of $\varphi_T$ in (c) there). This corestriction
map is a composite $H^1(W_L, \hat \T) \rightarrow
H^1(W_F, \Ind_{W_L}^{W_F} \hat \T) \rightarrow H^1(W_F, \hat \T)$, where
the first map is the isomorphism given by Shapiro's lemma, and the
second is induced by an appropriate surjection
$\Ind_{W_L}^{W_F} \hat \T \rightarrow \hat \T$, involving
a certain sum over representatives for $W_F/W_L$:
see \cite[towards the end of Section 2.5]{Ser02}
(though this reference treats profinite groups, the same applies
in our context; $W_L \subset W_F$ is of finite index). Restricted to
$H^1(W_L/I_L^{l_1}, \hat \T) \subset H^1(W_L, \hat \T)$,
this map is a similarly defined composite
$H^1(W_L/I_L^{l_1}, \hat \T) \rightarrow
H^1(W_F/I_F^l, \Ind_{W_L/I_L^{l_1}}^{W_F/I_F^l} \hat \T)
\rightarrow H^1(W_F/I_F^l, \hat \T)$, where this time one
uses, as one clearly may, the Shapiro's lemma isomorphism
$H^1(W_L/I_L^{l_1}, \hat \T) \rightarrow H^1(W_F/I_F^l, \Ind_{W_L/I_L^{l_1}}^{W_F/I_F^l} \hat \T)$, and a sum over representatives for
$(W_F/I_F^l)/(W_L/I_L^{l_1})$.
Similarly, $\varphi_{\chi'}$ is the image of $\varphi_{\chi_1'}$
under a composite $H^1(W_{L'}/I_{L'}^{l_1}, \hat \T') \rightarrow
H^1(W_{F'}/I_{F'}^l, \Ind_{W_{L'}/I_{L'}^{l_1}}^{W_{F'}/I_{F'}^l} \hat \T')
\rightarrow H^1(W_{F'}/I_{F'}^l, \hat \T')$.

Now, using the identification
$W_F/I_F^l = W_{F'}/I_{F'}^l \supset W_{L'}/I_{L'}^{l_1}
= W_L/I_L^{l_1}$, the identification $\hat \T = \hat \T'$ as modules over
$W_F/I_F^l = W_{F'}/I_{F'}^l$, and also using the observation above relating
$\varphi_{\chi_1}$ and $\varphi_{\chi_1'}$,
we conclude that $\varphi_{\chi}$ is indeed the image of $\varphi_{\chi'}$
under the isomorphism
$H^1(W_F/I_F^l, \hat \T) \rightarrow H^1(W_{F'}/I_{F'}^l, \hat \T')$,
finishing the proof of the proposition.
\end{proof}
\subsection{Properties of congruent and Chai-Yu isomorphisms}

\begin{lm} \label{lm: congruent isomorphism}
Let $(F, \T) \leftrightarrow_l (F', \T')$, and let $m$ be a positive integer.
\begin{enumerate}[(i)]
\item \label{congruent isomorphism well-defined} If an isomorphism
$\T(F)/\T(F)_m \rightarrow \T'(F')/\T'(F')_m$ satisfies the conditions
of Definition \ref{df: standard congruent Chai-Yu}\eqref{congruent isomorphism}
with respect to some choice of $\tilde F \hookrightarrow F^{\sep},
\tilde F' \hookrightarrow {F'}^{\sep}$ and $r > 0$, then so does it
with respect to any other such choice (as in the definition).
Thus, there is either a unique congruent isomorphism
$\T(F)/\T(F)_m \rightarrow \T'(F')/\T'(F')_m$,
or none at all. A similar assertion applies to Chai-Yu isomorphisms
(Definition \ref{df: standard congruent Chai-Yu}\eqref{Chai-Yu isomorphism}).

\item \label{congruent isomorphism functorial} 
Congruent isomorphisms have the following functoriality.
Let $(F, \T_i) \leftrightarrow_l (F', \T_i')$ for $i = 1, 2$, with
the same underlying $F \leftrightarrow_l F'$, and let
$m$ be a positive integer.
Let $f : \T_1 \rightarrow \T_2$ and $f' : \T_1' \rightarrow \T_2'$ be
homomorphisms inducing the same homomorphism
$X^*(\T_2') = X^*(\T_2) \rightarrow X^*(\T_1) = X^*(\T_1')$ at the level
of character lattices.
Assume that congruent isomorphisms 
$c_i : \T_i(F)/\T_i(F)_m \rightarrow \T_i'(F')/\T_i'(F')_m$
exist for $i = 1, 2$, and form the following
diagram:
\[
\xymatrix{
\T_1(F)/\T_1(F)_m \ar[r]^{f} \ar[d]_{c_1} & \T_2(F)/\T_2(F)_m \ar[d]^{c_2} \\
\T_1'(F')/\T_1'(F')_m \ar[r]^{f'} & \T_2'(F')/\T_2'(F')_m
}.
\]
Then this diagram is commutative under the following
additional assumption: ``the same $r$ can be used to define both
the congruent isomorphisms'', i.e., there exists
$r > 0$ such that for $i = 1, 2$, we have $r \lleq_{\T_i} l$,
$\T_i(\tilde F)^{\naive}_r \subset \T_i(\tilde F)_m$
and $\T_i'(\tilde F')^{\naive}_r \subset \T_i'(\tilde F')_m$,
for some choice of compatible embeddings $\tilde F \hookrightarrow F^{\sep}$ and
$\tilde F' \hookrightarrow {F'}^{\sep}$, where $\tilde F/F$
is a maximal unramified extension.

\item \label{congruent isomorphism sufficient condition}
Recall the following necessary condition for 
the existence of a
congruent isomorphism 
\[\T(F)/\T(F)_m \rightarrow \T'(F')/\T'(F')_m:\] for some
compatible embeddings $\tilde F \hookrightarrow F^{\sep}$ and
$\tilde F' \hookrightarrow {F'}^{\sep}$, where $\tilde F/F$
is a maximal unramified extension,
and some $r \lleq_{\T} l$ such that $\T(\tilde F)^{\naive}_r \subset
\T(\tilde F)_m$ and $\T'(\tilde F')^{\naive}_r \subset \T'(\tilde F')_m$,
a standard isomorphism $\T(\tilde F)/\T(\tilde F)^{\naive}_r \rightarrow
\T'(\tilde F')/\T'(\tilde F')^{\naive}_r$ induces an isomorphism
$\T(\tilde F)/\T(\tilde F)_m \rightarrow \T'(\tilde F')/\T'(\tilde F')_m$.
This condition is also sufficient.

\item \label{congruent isomorphism s}
Suppose there exists a congruent isomorphism
$\T(F)/\T(F)_m \rightarrow \T'(F')/\T'(F')_m$. Then for all
$0 < s \leq m$, it induces a standard isomorphism
$\T(F)/\T(F)^{\naive}_s \rightarrow \T'(F')/\T'(F')^{\naive}_s$.
\end{enumerate}
\end{lm}
\begin{proof}
To see \eqref{congruent isomorphism well-defined},
combine Lemma \ref{lm: standard isomorphism functorial 2}
(to see the non-dependence on
$\tilde F \hookrightarrow F^{\sep}$ and $\tilde F' \hookrightarrow {F'}^{\sep}$)
with Lemma \ref{lm: standard isomorphism decreasing r}
(to see the non-dependence on $r$). For the assertion concerning
Chai-Yu isomorphisms, one also uses
\cite[Lemma 8.5.1]{CY01}.

Now we come to \eqref{congruent isomorphism functorial}. For $i = 1, 2$,
the relation $r \lleq_{\T_i} l$ easily implies $r \lleq_{(\T_i)_{\tilde F}} l$,
and hence Proposition \ref{pro: standard isomorphism strictly Henselian}
gives a standard isomorphism
$\T_i(\tilde F)/\T_i(\tilde F)^{\naive}_r \rightarrow
\T_i'(\tilde F')/\T_i'(\tilde F')^{\naive}_r$. Thus,
\eqref{congruent isomorphism functorial} follows from	
Lemma \ref{lm: standard isomorphism functorial} and the
fact that we can work with the given $r$
(by \eqref{congruent isomorphism well-defined}).

Now we come to \eqref{congruent isomorphism sufficient condition}.
It follows from Corollary \ref{cor: standard isomorphism functorial 2} that
the standard isomorphism 
\[\T(\tilde F)/\T(\tilde F)^{\naive}_r \rightarrow
\T'(\tilde F')/\T'(\tilde F')^{\naive}_r,\] and hence also the isomorphism
$\T(\tilde F)/\T(\tilde F)_m \rightarrow \T'(\tilde F')/\T'(\tilde F')_m$,
is invariant under $\Gamma_{\tilde F/F} = \Gamma_{\tilde F'/F'}$.
Thus, it suffices to show that
$(\T(\tilde F)/\T(\tilde F)_m)^{\Gamma_{\tilde F/F}}
= \T(F)/\T(F)_m$ (then the analogous assertion for $F'$ will be true as well).
This in turn follows if we show that
$H^1(\Gamma_{\tilde F/F}, \T(\tilde F)_m) = 0$,
which is a special case of \cite[Proposition 13.8.1]{KP23}.
This gives \eqref{congruent isomorphism sufficient condition}.

It remains to prove \eqref{congruent isomorphism s}.
Choose compatible embeddings $\tilde F \hookrightarrow F^{\sep}$ and
$\tilde F' \hookrightarrow {F'}^{\sep}$, where $\tilde F/F$
is a maximal unramified extension.
Choose $r$ such that a standard isomorphism
$\T(\tilde F)/\T(\tilde F)^{\naive}_r \rightarrow
\T'(\tilde F')/\T'(\tilde F')^{\naive}_r$ induces an isomorphism
$\T(\tilde F)/\T(\tilde F)_m \rightarrow \T'(\tilde F')/\T'(\tilde F')_m$
that restricts to $\T(F)/\T(F)_m \rightarrow \T'(F')/\T'(F')_m$. Recall
that $r \geq m$ (since $\T(\tilde F)^{\naive}_r \subset \T(\tilde F)_m
\subset \T(\tilde F)^{\naive}_m$).
It suffices to show that, whenever
$t \in \T(F)$ and $t' \in \T'(F')$ have images that match
under $\T(F)/\T(F)_m \rightarrow \T'(F')/\T'(F')_m$, $t \T(F)^{\naive}_s$
and $t' \T'(F')^{\naive}_s$ are standard correspondents
(for ``level $s$''). An easy argument reduces this to showing that $t$ and $t'$
have images that match under the standard isomorphism
$\T(\tilde F)/\T(\tilde F)^{\naive}_s \rightarrow
\T'(\tilde F')/\T'(\tilde F')^{\naive}_s$ (which
exists by Lemma \ref{lm: standard isomorphism decreasing r} and the
fact that $s \leq m \leq r$).
Now we are done by Lemma \ref{lm: standard isomorphism decreasing r},
since there exist $t_{\circ} \in \T(\tilde F)_m \subset
\T(\tilde F)^{\naive}_m \subset \T(\tilde F)^{\naive}_s$, and
similarly $t_{\circ}' \in \T'(\tilde F')^{\naive}_s$,
such that $t t_{\circ}$ and $t' t_{\circ}'$
have images that match under the standard isomorphism
$\T(\tilde F)/\T(\tilde F)^{\naive}_r
\rightarrow \T'(\tilde F')/\T'(\tilde F')^{\naive}_r$.
\end{proof}

\begin{lm} \label{lm: Chai-Yu isomorphism well-defined}
Let $(F, \T) \leftrightarrow_l (F', \T')$, and let $m$ be a positive integer.
Then there exists associated to this data at most one Chai-Yu isomorphism
$\mcT^{\ft} \times_{\mO_F} \mO_F/\p_F^m \rightarrow
{\mcT'}^{\ft} \times_{\mO_{F'}} \mO_{F'}/\p_{F'}^m$.
\end{lm}
\begin{proof}
Combine the argument for congruent isomorphisms in
Lemma \ref{lm: congruent isomorphism}\eqref{congruent isomorphism well-defined}
with the schematic density of the image of $\T(\tilde F)_b
= \mathcal{T}^{\ft}(\mO_{\tilde F})$
in $\mcT^{\ft} \times_{\mO_F} \mO_F/\p_F^m$ (\cite[Lemma 8.5.1]{CY01}).
\end{proof}

\begin{pro} \label{pro: Chai-Yu congruent}
Let $(F, \T) \leftrightarrow_l (F', \T')$, and let $m$ be a positive integer.
If there exists a Chai-Yu isomorphism
$\mcT^{\ft} \times_{\mO_F} \mO_F/\p_F^m \rightarrow
{\mcT'}^{\ft} \times_{\mO_{F'}} \mO_{F'}/\p_{F'}^m$, then
there exists a congruent isomorphism
$\T(F)/\T(F)_m \rightarrow \T'(F')/\T'(F')_m$. Moreover, this
congruent isomorphism restricts to an isomorphism
$\T(F)_b/\T(F)_m \rightarrow \T'(F')_b/\T'(F')_m$
obtained by evaluating the Chai-Yu isomorphism at
$\mO_F/\p_F^m = \mO_{F'}/\p_{F'}^m$.
\end{pro}
\begin{proof}
Choose compatible embeddings $\tilde F \hookrightarrow F^{\sep}$ and
$\tilde F' \hookrightarrow {F'}^{\sep}$, where $\tilde F/F$ is
a maximal unramified extension.
By the definition of a Chai-Yu isomorphism
(Definition \ref{df: standard congruent Chai-Yu}\eqref{Chai-Yu isomorphism}),
for some $0 < r \lleq_{\T} l$, there exists
a ``restricted standard isomorphism'' 
$\T(\tilde F)_b/\T(\tilde F)^{\naive}_r
\rightarrow \T'(\tilde F')_b/\T'(\tilde F')^{\naive}_r$
that induces the isomorphism 
$\T(\tilde F)_b/\T(\tilde F)_m \rightarrow \T'(\tilde F')_b/\T'(\tilde F')_m$
obtained by evaluating the given Chai-Yu isomorphism
at $\mO_{\tilde F}/\p_{\tilde F}^m = \mO_{\tilde F'}/\p_{\tilde F'}^m$.
There also exists a standard isomorphism
$\T(\tilde F)/\T(\tilde F)^{\naive}_r
\rightarrow \T'(\tilde F')/\T'(\tilde F')^{\naive}_r$
by Proposition \ref{pro: standard isomorphism strictly Henselian}, which
restricts to the restricted standard isomorphism
$\T(\tilde F)_b/\T(\tilde F)^{\naive}_r
\rightarrow \T'(\tilde F')_b/\T'(\tilde F')^{\naive}_r$,
by Lemma \ref{lm: standard isomorphism decreasing r}.

The restricted standard isomorphism induces an isomorphism 
$\T(\tilde F)_b/\T(\tilde F)_m \rightarrow \T'(\tilde F')_b/\T'(\tilde F')_m$
and hence takes the image of $\T(\tilde F)_m$ to that
of $\T'(\tilde F')_m$. Hence so does the standard isomorphism as well,
which therefore induces an isomorphism
$\T(\tilde F)/\T(\tilde F)_m \rightarrow \T'(\tilde F')/\T'(\tilde F')_m$.
As in the proof of Lemma \ref{lm: congruent isomorphism}\eqref{congruent isomorphism sufficient condition},
using that
$H^1(\Gamma_{\tilde F/F}, \T(\tilde F)_m)
= 0 = H^1(\Gamma_{\tilde F'/F'}, \T'(\tilde F')_m)$, this isomorphism
$\T(\tilde F)/\T(\tilde F)_m \rightarrow \T'(\tilde F')/\T'(\tilde F')_m$
restricts to an isomorphism
$\T(F)/\T(F)_m \rightarrow \T'(F')/\T'(F')_m$, which is clearly
a congruent isomorphism that satisfies the latter assertion
of the lemma.
\end{proof}

Now we study the behavior of Chai-Yu isomorphisms
with respect to minimal congruent filtrations.
\begin{pro} \label{pro: minimal congruent Chai-Yu}
Let $(F, \T) \leftrightarrow_l (F', \T')$, and let $m$ be a positive integer.
Fix compatible embeddings $\tilde F \hookrightarrow F^{\sep}$ and
$\tilde F' \hookrightarrow {F'}^{\sep}$, where $\tilde F/F$ is
a maximal unramified extension.
Assume that, associated to $(F, \T) \leftrightarrow_l (F', \T')$,
there exists a Chai-Yu isomorphism
$\mathcal{T}^{\ft} \times_{\mO_F} \mO_F/\p_F^{m + 1} \rightarrow
{\mathcal{T}'}^{\ft} \times_{\mO_{F'}} \mO_{F'}/\p_{F'}^{m + 1}$
(``one higher level''), say induced by some isomorphism
$\T(\tilde F)_b/\T(\tilde F)^{\naive}_r
\rightarrow \T'(\tilde F')_b/\T'(\tilde F')^{\naive}_r$
as in Definition \ref{df: standard congruent Chai-Yu}\eqref{Chai-Yu isomorphism}
(where $m + 1 \leq r \lleq_{\T} l$). Then for all $0 \leq s \leq m$:
\begin{enumerate}[(i)]
\item \label{minimal congruent Chai-Yu} The isomorphism
$\T(\tilde F)_b/\T(\tilde F)_m \rightarrow \T'(\tilde F')_b/\T'(\tilde F')_m$,
obtained by evaluating the given Chai-Yu isomorphism
at $\mO_{\tilde F}/\p_{\tilde F}^m = \mO_{\tilde F'}/\p_{\tilde F'}^m$,
sends the image of $\T(\tilde F)_s$ to that of $\T'(\tilde F')_s$.
\item \label{minimal congruent Chai-Yu group scheme} At the level
of schemes, letting $\mathcal{T}_s$ and $\mathcal{T}'_s$ be the
minimal congruent filtration group schemes associated to $\T$ and $\T'$ of
level $s$, one has a unique isomorphism $\mathcal{T}_s \times_{\mO_F}
\mO_F/\p_F^{\lfloor m + 1 - s \rfloor} \rightarrow
\mathcal{T}'_s \times_{\mO_{F'}}
\mO_{F'}/\p_{F'}^{\lfloor m + 1 - s \rfloor}$ of schemes over
$\mO_F/\p_F^{\lfloor m + 1 - s \rfloor} = \mO_{F'}/\p_{F'}^{\lfloor m + 1 - s \rfloor}$,
under which the images of $t \in \mcT_s(\mO_{\tilde F}) = \T(\tilde F)_s$ and 
$t' \in \mcT_s'(\mO_{\tilde F'}) = \T'(\tilde F')_s$
correspond whenever the Chai-Yu isomorphism being considered
sends the image of $t$ in
$\mcT^{\ft}(\mO_{\tilde F}/\p_{\tilde F}^{m + 1}) = \T(\tilde F)_b/\T(\tilde F)_{m + 1}$ to that of $t'$ in
${\mcT'}^{\ft}(\mO_{\tilde F'}/\p_{\tilde F'}^{m + 1})
= \T'(\tilde F')_b/\T'(\tilde F')_{m + 1}$.
\end{enumerate}
\end{pro}
\begin{proof}
We have $m + 1 \lleq_{\T} l$ (as is implicit in the existence of the given
Chai-Yu isomorphism), i.e., $m + 1 \lleq_L l$ for some extension $L/F$ splitting
$\T$. If $\tilde L$ is a compositum of $L$ and $\tilde F$, we have
$\psi_{\tilde L/F}(l) = \psi_{\tilde L/L} \circ \psi_{L/F}(l) = \psi_{L/F}(l)$,
so $m + 1 \lleq_{\T_{\tilde F}} l$.
Fix compatible embeddings $\tilde L \hookrightarrow F^{\sep}$ and
$\tilde L' \hookrightarrow {F'}^{\sep}$, the former extending
$\tilde F \hookrightarrow F^{\sep}$. Hence
$\tilde F' \hookrightarrow {F'}^{\sep}$ factors through
$\tilde L' \hookrightarrow {F'}^{\sep}$ as well. 
Tautologically, $\tilde L \hookrightarrow F^{\sep}$ and
$\tilde L' \hookrightarrow {F'}^{\sep}$
are also compatible embeddings
for $\tilde F \leftrightarrow_{\psi_{\tilde F/F}(l) = l} \tilde F'$. 

Suppose that \eqref{minimal congruent Chai-Yu}
and the existence assertion of \eqref{minimal congruent Chai-Yu group scheme}
are known. Then each $t$ as in \eqref{minimal congruent Chai-Yu group scheme}
has a corresponding $t'$ (by \eqref{minimal congruent Chai-Yu})
and vice versa (by symmetry); therefore the uniqueness
assertion in \eqref{minimal congruent Chai-Yu group scheme} follows from
\cite[Lemma 8.5.1]{CY01}.

The proposition being trivial for $s = 0$, our first aim is to prove
just \eqref{minimal congruent Chai-Yu} for $0 < s < 1$.
By the definitions of
$\T(\tilde F)_s$ and $\T'(\tilde F')_s$ (see \cite[Definition B.10.8(2)]{KP23},
and the description involving dilatation in
\cite[the proof of Lemma B.10.9]{KP23}), and the fact that
$\T(\tilde F)_b/\T(\tilde F)_{m + 1} \rightarrow \T'(\tilde F')_b/\T'(\tilde F')_{m + 1}$
takes the image of $\T(\tilde F)_1$ to that of $\T'(\tilde F')_1$
(because the Chai-Yu isomorphism is a morphism
of schemes over $\mO_F/\p_F^{m + 1} = \mO_{F'}/\p_{F'}^{m + 1}$),
it suffices to show that for any homomorphism
$\S \rightarrow \T_{\tilde F}$ with $\S$ an induced torus over $\tilde F$, the
isomorphism $\T(\tilde F)_b/\T(\tilde F)_{m + 1}
\rightarrow \T'(\tilde F')_b/\T'(\tilde F')_{m + 1}$ takes the image of
$\S(\tilde F)_s$ in the source into that of $\S'(\tilde F')_s$ under
some homomorphism from an induced torus $\S'$ over $\tilde F'$
to $\T'_{\tilde F'}$ (we thank Kaletha for informing us that in
\cite[Definition B.10.8(2)]{KP23}, $R$ varies over
induced $K$-tori; this is why we take $\S$ to be an induced
torus over $\tilde F$ and not over $F$).

The map $\S \rightarrow \T_{\tilde F}$ factors through
the maximal $\tilde L$-split
($\tilde F$-torus) quotient of $\S$, since $X^*(\T) \rightarrow X^*(\S)$
has image inside $X^*(\S)^{\Gal(F^{\sep}/\tilde L)}$. $\Gal(F^{\sep}/\tilde F)$
permutes some basis for $X^*(\S)$, and hence also the set of
$\Gal(F^{\sep}/\tilde L)$-orbits of elements of this basis,
and hence also some basis of the character lattice
$X^*(\S)^{\Gal(F^{\sep}/\tilde L)}$ of the maximal $\tilde L$-split
quotient of $\S$.  Thus, the maximal $\tilde L$-split quotient of
$\S$ is an induced torus as well, with which we may now replace $\S$,
to assume that $\S$ is $\tilde L$-split, and in particular
at most $l$-ramified, and satisfying $m + 1 \lleq_{\S} l$.

This gives us a torus $\S'$ over $\tilde F'$, which is clearly induced
and splits over $\tilde L'$, and a homomorphism
$\S' \rightarrow \T'_{\tilde F'}$ such that
$X^*(\T') \rightarrow X^*(\S')$ identifies with the homomorphism
$X^*(\T) \rightarrow X^*(\S)$ dual to $\S \rightarrow \T_{\tilde F}$.
Since $\S'$ is induced, \eqref{minimal congruent Chai-Yu}
will follow if we show that $\T(\tilde F)/\T(\tilde F)_{m + 1}
\rightarrow \T'(\tilde F')/\T'(\tilde F')_{m + 1}$ takes the image of
$\S(\tilde F)_s$ in the source to that of $\S'(\tilde F')_s$ in the target.

Proposition \ref{pro: standard isomorphism strictly Henselian} shows that
standard isomorphisms $\S(\tilde F)/\S(\tilde F)^{\naive}_r
\rightarrow \S'(\tilde F')/\S'(\tilde F')^{\naive}_r$
and $\T(\tilde F)/\T(\tilde F)^{\naive}_r
\rightarrow \T'(\tilde F')/\T'(\tilde F')^{\naive}_r$
exist, the latter clearly restricting to the
isomorphism $\T(\tilde F)_b/\T(\tilde F)^{\naive}_r
\rightarrow \T'(\tilde F')_b/\T'(\tilde F')^{\naive}_r$
in the statement of the proposition.
Applying Lemma \ref{lm: standard isomorphism functorial}
in the context of the homomorphisms
$\S \rightarrow \T_{\tilde F}$ and
$\S' \rightarrow \T'_{\tilde F'}$, and using
that the standard isomorphism
$\S(\tilde F)/\S(\tilde F)^{\naive}_r
\rightarrow \S'(\tilde F')/\S'(\tilde F')^{\naive}_r$
identifies the images of $\S(\tilde F)_s$ and $\S'(\tilde F')_s$
(by Lemma \ref{lm: standard isomorphism decreasing r}),
it follows that the images of $\S(\tilde F)_s$ and $\S'(\tilde F')_s$ agree
in $\T(\tilde F)/\T(\tilde F)^{\naive}_r =
\T'(\tilde F')/\T'(\tilde F')^{\naive}_r$.

By the choice of $r$, the images
of $\S(\tilde F)_s$ and $\S'(\tilde F')_s$ in
$\T(\tilde F)/\T(\tilde F)_{m + 1}$ and
$\T'(\tilde F')/\T'(\tilde F')_{m + 1}$, respectively,
match under the isomorphism
$\T(\tilde F)_b/\T(\tilde F)_{m + 1} \rightarrow \T'(\tilde F')_b/\T'(\tilde F')_{m + 1}$.
Thus, \eqref{minimal congruent Chai-Yu} follows
for $0 < s < 1$, and hence for $0 \leq s < 1$.

Now let us prove \eqref{minimal congruent Chai-Yu group scheme}
for $0 \leq s < 1$; this is what necessitated needing a Chai-Yu
isomorphism of level $m + 1$. The case of $s = 0$ is immediate:
$\mathcal{T}^{\ft} \times_{\mO_F} \mO_F/\p_F^{m + 1} \rightarrow
{\mathcal{T}'}^{\ft} \times_{\mO_{F'}} \mO_{F'}/\p_{F'}^{m + 1}$
restricts to an isomorphism
$\mathcal{T} \times_{\mO_F} \mO_F/\p_F^{m + 1} \rightarrow
{\mathcal{T}'} \times_{\mO_{F'}} \mO_{F'}/\p_{F'}^{m + 1}$,
and we have $\mathcal{T} = \mathcal{T}_0$ and
$\mathcal{T}' = \mathcal{T}'_0$.
Hence we assume $s > 0$. Since \eqref{minimal congruent Chai-Yu}
is known in this case, with $m$
replaced by $m + 1$, $\mcT_s$ and $\mcT'_s$ are respectively the dilatations of
$\mcT = \mcT_0$ and $\mcT' = \mcT_0'$ with respect to the same subgroup
$\W_s$ of
$\mcT \times_{\mO_F} \kappa_F = \mcT' \times_{\mO_{\tilde F'}} \kappa_{F'}$
(identified using the Chai-Yu isomorphism).
Now the required isomorphism
$\mathcal{T}_s \times_{\mO_F} \mO_F/\p_F^m \rightarrow
\mathcal{T}'_s \times_{\mO_{F'}} \mO_{F'}/\p_{F'}^m$,
described as in \eqref{minimal congruent Chai-Yu group scheme},
follows from
Proposition \ref{pro: CY01 Proposition 4.2}, which applies since this subgroup
is reduced and hence smooth over $\kappa_F = \kappa_{F'}$, and since
$\mcT$ (over $\mO_F$) and $\mcT'$ (over $\mO_{F'}$) are smooth;
note that $m = \lfloor m + 1 - s \rfloor$.
For this step, we needed $m + 1$ in place of $m$. Note that
$\mcT_s$ is not a subgroup scheme of $\mcT = \mcT_0$, and
\cite[Proposition 4.2]{CY01} (summarized in
Proposition \ref{pro: CY01 Proposition 4.2}) is doing much work here.

Now consider general $s$ with $0 \leq s < m$.

Let us prove \eqref{minimal congruent Chai-Yu}.
If $t \in \T(\tilde F)_s$, and if $t' \in \mcT'(\mO_{\tilde F'})$
has the same image as $t$ in
$\mcT(\mO_{\tilde F}/\p_{\tilde F}^{m + 1}) =
\mcT'(\mO_{\tilde F'}/\p_{\tilde F'}^{m + 1})$, then since $t$ and $t'$
have the same image in the special fiber $\mcT(\mO_{\tilde F}/\p_{\tilde F}) =
\mcT'(\mO_{\tilde F'}/\p_{\tilde F'})$, it follows that
$t' \in \T'(\tilde F')_{s - \lfloor s \rfloor}$. Thus, by
\eqref{minimal congruent Chai-Yu group scheme} in the case where $0 \leq s < 1$
(applied with $s - \lfloor s \rfloor$ in place of $s$), $t$ and
$t'$ have the same image in
$\mcT_{s - \lfloor s \rfloor}(\mO_{\tilde F}/\p_{\tilde F}^m)
= \mcT_{s - \lfloor s \rfloor}'(\mO_{\tilde F'}/\p_{\tilde F'}^m)$.
The conditions $t \in \T(\tilde F)_s$ and
$t' \in \T'(\tilde F')_s$ both translate to this image
having trivial further image in
$\mcT_{s - \lfloor s \rfloor}(\mO_{\tilde F}/\p_{\tilde F}^{\lfloor s \rfloor})
= \mcT_{s - \lfloor s \rfloor}'(\mO_{\tilde F'}/\p_{\tilde F'}^{\lfloor s \rfloor})$. Thus, \eqref{minimal congruent Chai-Yu} follows.

Applying \eqref{minimal congruent Chai-Yu group scheme} with
$s - \lfloor s \rfloor$ in place of $s$,
and applying Proposition \ref{pro: CY01 Proposition 4.2}
$\lfloor s \rfloor$ times
\eqref{minimal congruent Chai-Yu group scheme} follows (use
that $\lfloor m + 1 - (s - \lfloor s \rfloor) \rfloor - \lfloor s \rfloor
= \lfloor m + 1 - s \rfloor$).
\end{proof}

\subsection{Relating to the work of Chai and Yu}

\begin{pro} \label{pro: Chai-Yu isomorphism CY01}
The isomorphism of Chai and Yu described in Theorem \ref{thm: CY01}
(the right-most vertical arrow of \eqref{eqn: CY01 main result})
is a Chai-Yu isomorphism. 
\end{pro}

The main input into the proof of the above proposition
is the following lemma.

\begin{lm} \label{lm: weakly induced Chai-Yu base-change}
Consider the setting of Notation \ref{notn: h}\eqref{Chai-Yu ResLFTL}.
Thus, $(F, \T) \leftrightarrow_l (F', \T')$, and we consider
$(F, \R := \Res_{L/F} \T_L) \leftrightarrow_l (F', \R' := \Res_{L'/F'} \T'_{L'})$,
where $L/F$ is an at most $l$-ramified
finite Galois extension splitting $\T$, and $L \hookrightarrow F^{\sep}$
and $L' \hookrightarrow {F'}^{\sep}$ are compatible embeddings.
Let $0 < m \lleq_L l$. Then the isomorphism
$\mathcal{R} \times_{\mO_F} \mO_F/\p_F^m \rightarrow
\mathcal{R}' \times_{\mO_{F'}} \mO_{F'}/\p_{F'}^m$ of
\eqref{eqn: Chai-Yu ResLFTL} is a Chai-Yu isomorphism.
\end{lm}
\begin{proof}
Note that $\mathcal{R}^{\ft} = \mathcal{R}$ and
${\mathcal{R}'}^{\ft} = \mathcal{R}'$.
Some of the proof will be written informally, for lightness of reading.

Let $\{\chi_i = \chi_i'\}_i$ be a basis for $X^*(\T) = X^*(\T')$. It
gives an isomorphism $\R = \prod_i \Res_{L/F} \GG_m,
\R' = \prod_i \Res_{L'/F'} \GG_m$. The realization
$(F, \R) \leftrightarrow_l (F', \R')$ is then, in an obvious sense, a product
of the obvious realizations $\prod_i (F, \Res_{L/F} \GG_m) \leftrightarrow_l
(F', \Res_{L'/F'} \GG_m)$.

Further, the isomorphism $\mathcal{R} \times_{\mO_F} \mO_F/\p_F^m
\rightarrow \mathcal{R}' \times_{\mO_{F'}} \mO_{F'}/\p_{F'}^m$
given by \eqref{eqn: Chai-Yu ResLFTL} then
becomes the product of the isomorphisms
$\Res_{\mO_L/\mO_F} \GG_m \times_{\mO_F} \mO_F/\p_F^m
\rightarrow \Res_{\mO_{L'}/\mO_{F'}} \GG_m \times_{\mO_{F'}} \mO_{F'}/\p_{F'}^m$,
each of which is given, at the level of $A$-points for an algebra
$A$ over $\mO_F/\p_F^m = \mO_{F'}/\p_{F'}^m$, by the identification
$((\mO_L/\p_F^m \mO_L) \otimes_{\mO_F} A)^{\times}
\rightarrow ((\mO_{L'}/\p_{F'}^m \mO_{L'}) \otimes_{\mO_{F'}} A)^{\times}$.
It is enough to prove that this isomorphism is a Chai-Yu isomorphism
for $(F, \Res_{L/F} \GG_m) \leftrightarrow_l (F', \Res_{L'/F'} \GG_m)$.

In other words, we may assume that $\T = \GG_m$, though the chosen
splitting extension used to define $\R = \Res_{L/F} \T_L = \Res_{L/F} \GG_m$
is still $L/F$. 

Let $\tilde F \hookrightarrow F^{\sep}$ and
$\tilde F' \hookrightarrow {F'}^{\sep}$ be compatible
extensions, with $\tilde F/F$ a maximal unramified extension.
Since $\R = \Res_{L/F} \GG_m$ is an induced torus, it is standard (and easy)
that $\R(\tilde F)^{\naive}_m = \R(\tilde F)_m$. Therefore, keeping in mind
Lemma \ref{lm: congruent isomorphism}\eqref{congruent isomorphism well-defined},
we may take $r = m$ in the definition of a Chai-Yu isomorphism. 
It is enough to show that the isomorphism
$\R(\tilde F)_b/\R(\tilde F)_m \rightarrow \R'(\tilde F')_b/\R'(\tilde F')_m$
obtained by evaluating $\mathcal{R} \times_{\mO_F} \mO_F/\p_F^m
\rightarrow \mathcal{R}' \times_{\mO_{F'}} \mO_{F'}/\p_{F'}^m$
at $\mO_{\tilde F}/\p_{\tilde F}^m = \mO_{\tilde F'}/\p_{\tilde F'}^m$
is a ``restricted standard isomorphism'' for
$(\tilde F, \R_{\tilde F}) \leftrightarrow_l (\tilde F', \R'_{\tilde F'})$.

Let $\tilde L \hookrightarrow F^{\sep}$ 
(resp., $\tilde L' \hookrightarrow {F'}^{\sep}$) be a compositum of
$L \hookrightarrow F^{\sep}$ and $\tilde F \hookrightarrow F^{\sep}$
(resp.,
$L' \hookrightarrow {F'}^{\sep}$ and $\tilde F' \hookrightarrow {F'}^{\sep}$).
It is then immediate that $\tilde L \hookrightarrow F^{\sep}$
and $\tilde L' \hookrightarrow {F'}^{\sep}$ have the same stabilizer
in $\Gamma_F/I_F^l = \Gamma_{F'}/I_{F'}^l$, i.e., are compatible embeddings
for $F \leftrightarrow_l F'$, and hence also
for for $\tilde F \leftrightarrow_l \tilde F'$. 
Note that $m \lleq_{\tilde L} l$ (since $\psi_{\tilde L/F} = \psi_{L/F}$), and
that $\tilde L \leftrightarrow_{\psi_{\tilde L/F}(l)} \tilde L'$
lies over both $L \leftrightarrow_{\psi_{L/F}(l)} L'$
and $\tilde F \leftrightarrow_l \tilde F'$.
Set $e = e(L/F) = e(\tilde L/\tilde F)$.
We use $\tilde L/\tilde F$ as a splitting extension for $\R_{\tilde F}$.

Thus, if $t \in \R(\tilde F)_b = \mathcal{R}(\mO_{\tilde F})$ and
$t' \in \R'(\tilde F')_b = \mathcal{R}'(\mO_{\tilde F'})$ have
the same image in $\mathcal{R}(\mO_{\tilde F}/\p_{\tilde F}^m)
= \mathcal{R}'(\mO_{\tilde F'}/\p_{\tilde F'}^m)$, it is enough to
show that for all $\chi = \chi' \in X^*(\R'_{\tilde F'}) =
X^*(\R_{\tilde F})$, $\chi(t) \in \mO_{\tilde L}^{\times}$
and $\chi'(t') \in \mO_{\tilde L'}^{\times}$ have the same image
in $(\mO_{\tilde L}/\p_{\tilde L}^{em})^{\times}
= (\mO_{\tilde L'}/\p_{\tilde L'}^{e m})^{\times}$, i.e.,
in $(\mO_{\tilde L}/\p_{\tilde F}^m \mO_{\tilde L})^{\times}
= (\mO_{\tilde L'}/\p_{\tilde F'}^m \mO_{\tilde L'})^{\times}$.

It is enough to prove this for $\chi = \chi'$ running over some basis
of $X^*(\R_{\tilde F}) = X^*(\R) = X^*(\R') = X^*(\R'_{\tilde F'})$.
We use the basis $\{\chi_{\sigma} = \chi'_{\sigma'} \mid
\sigma = \sigma' \in \Gamma_{L'/F'} = \Gamma_{L/F}\}$, where
for each $\tilde L$-algebra $A$,
$\chi_{\sigma} : (\Res_{L/F} \GG_m)(A) = (L \otimes_F A)^{\times}
\rightarrow A^{\times} = \GG_m(A)$ is a restriction of the 
map $L \otimes_F A \rightarrow A$ that takes $l \otimes a$ to
$\sigma(l) a$, and $\chi'_{\sigma'}$ has a similar description.
Taking $A = \tilde L$ and viewing $t$ as an element of
\[ \mathcal{R}(\mO_{\tilde F}) \subset \mathcal{R}(\mO_{\tilde L})
= (\mO_L \otimes_{\mO_F} \mO_{\tilde L})^{\times}
\subset (L \otimes_F \tilde L)^{\times}
= \R(\tilde L), \]
and similarly with $t'$, the lemma follows from the following
commutative diagram:
\[
\xymatrix{
(\mO_L/\p_F^m \mO_L) \otimes_{\mO_F}
(\mO_{\tilde L}/\p_{\tilde F}^m \mO_{\tilde L})
\ar[r] \ar[d] & \mO_{\tilde L}/\p_{\tilde F}^m \mO_{\tilde L} \ar[d] \\
(\mO_{L'}/\p_{F'}^m \mO_{L'}) \otimes_{\mO_{F'}} (\mO_{\tilde L'}/\p_{\tilde F'}^m \mO_{\tilde L'})
\ar[r] & \mO_{\tilde L'}/\p_{\tilde F'}^m \mO_{\tilde L'}
},
\]
where the top horizontal arrow sends $l \otimes \tilde l$
to $\sigma(l) \tilde l$, and the bottom horizontal arrow is analogous.
\end{proof}

\begin{proof}[Proof of Proposition \ref{pro: Chai-Yu isomorphism CY01}]
Since the map \eqref{eqn: Chai-Yu ResLFTL} is a Chai-Yu isomorphism
(Lemma \ref{lm: weakly induced Chai-Yu base-change}), the
left vertical arrow of \eqref{eqn: CY01 main result} takes
any element of its source to a standard correspondent
of it. In other words, it is a restriction of the standard isomorphism
$\R(\tilde F)/\R(\tilde F)_{m + h} = \R(\tilde F)/\R(\tilde F)^{\naive}_{m + h}
\rightarrow \R'(\tilde F')/\R'(\tilde F')^{\naive}_{m + h}
= \R'(\tilde F')/\R'(\tilde F')_{m + h}$, which exists by
Proposition \ref{pro: standard isomorphism strictly Henselian}.
Since the maps $X^*(\R) = X^*(\R') \rightarrow X^*(\T') = X^*(\T)$
that are dual to $\T \hookrightarrow \R$ and $\T' \hookrightarrow \R'$
coincide (Lemma \ref{lm: tori automorphisms}), it follows from
Proposition \ref{pro: standard isomorphism strictly Henselian}
and Lemma \ref{lm: standard isomorphism functorial} that the middle
vertical arrow of \eqref{eqn: CY01 main result} also sends each element
of its source to a standard correspondent of it
(in fact, this gives an alternate justification for the existence
of the middle vertical arrow of \eqref{eqn: CY01 main result}).
By definition (see Definition \ref{df: standard congruent Chai-Yu}\eqref{Chai-Yu isomorphism}),
this implies that $\mcT^{\ft} \times_{\mO_F} \mO_F/\p_F^m \rightarrow
{\mcT'}^{\ft} \times_{\mO_{F'}} \mO_{F'}/\p_{F'}^m$
is a Chai-Yu isomorphism.
\end{proof}

\section{The case of weakly induced tori}
\label{sec: weakly induced tori}

In this section, we will restrict to a class
of tori that includes all induced tori, namely,
the class of tori satisfying the beautiful condition (T)
identified in \cite{Yu15}, which, following \cite{KP23},
we will refer to as the class of weakly induced tori.
For these tori, the standard and minimal congruent filtrations coincide
(\cite[Corollary B.10.13]{KP23}).
We will show that for weakly induced tori,
standard, congruent and Chai-Yu isomorphisms exist in the
``best possible'' generality. It will follow that for these
tori, congruent isomorphisms are a special case of standard isomorphisms.

\subsection{Weakly induced tori} \label{subsec: weakly induced tori}

\begin{notn} \label{notn: weakly induced}
A torus $\T$ over a DVHF $F$ is said to be weakly
induced if it becomes an induced torus over some finite tamely ramified
extension of $F$. It is easy to see (\cite[Remark B.6.3]{KP23})
that $\T$ is weakly induced if and only if $X^*(\T)$ has a basis
that is permuted by the wild inertia group
$I_F^{> 0} := \bigcup_{r > 0} I_F^r \subset I_F \subset \Gamma_F$.
\end{notn}

The following lemma is one reason why weakly induced tori are easy to work with.
\begin{lm} \label{lm: weakly induced}
Let $\T$ be a weakly induced torus over a DVHF $F$ with perfect
residue field.
For any $r \geq 0$, we have
$\T(F)_r = \T(F)^{\std}_r$, and for any $r > 0$, we have
$\T(F)^{\std}_r = \T(F)^{\naive}_r$.
Consequently, using Remark \ref{rmk: naive filtration},
$\T(L_1)_r = (\T(L_2)_{e(L_2/L_1) r})^{\Gal(L_2/L_1)}$
and $\T_1(F)_r = \T_1(F) \cap \T_2(F)_r$ whenever
$r > 0$, $L_2/L_1/F$ is a chain of finitely ramified
separable field extensions with $L_2/L_1$ Galois,
and $\T_1 \hookrightarrow \T_2$ is an injective homomorphism
of weakly induced tori over $F$.
\end{lm}
\begin{proof}
For the equality $\T(F)_r = \T(F)^{\std}_r$, use
\cite[Corollary B.10.13]{KP23} (and intersect with $\T(F)$).
For the equality $\T(F)^{\std}_r = \T(F)^{\naive}_r$ when $r > 0$,
see \cite[Proposition B.6.4(3)]{KP23}.
\end{proof}

\subsection{Standard, congruent and Chai-Yu
isomorphisms for weakly induced tori}
\begin{pro} \label{pro: weakly induced tori isomorphism}
Let $(F, \T) \leftrightarrow_l (F', \T')$,
with $\T$ assumed to be weakly induced.
Suppose $0 < r \lleq_{\T} l$. Then there is a standard isomorphism
$\T(F)/\T(F)^{\naive}_r = \T(F)/\T(F)_r \rightarrow
\T'(F')/\T'(F')_r = \T'(F')/\T'(F')^{\naive}_r$. If $r = m$
is an integer, then this is also a congruent isomorphism.
\end{pro}
\begin{proof}
Since $\T$ is weakly induced over $F$, so is
$\T'$: $I_F^{> 0}$, acting through $I_F^{> 0}/I_F^l = I_{F'}^{> 0}/I_{F'}^l$,
permutes a basis of $X^*(\T) = X^*(\T')$.
Choose compatible embeddings $\tilde F \hookrightarrow F^{\sep}$
and $\tilde F' \hookrightarrow {F'}^{\sep}$, where
$\tilde F/F$ is a maximal unramified extension.
Consider the standard isomorphism $\T(\tilde F)/\T(\tilde F)_r
\rightarrow \T'(\tilde F')/\T'(\tilde F')_r$
associated to $(\tilde F, \T_{\tilde F}) \leftrightarrow_l
(\tilde F', \T'_{\tilde F'})$
(Proposition \ref{pro: standard isomorphism strictly Henselian}).
It is equivariant for $\Gamma_{\tilde F/F} = \Gamma_{\tilde F'/F'}$
(Corollary \ref{cor: standard isomorphism functorial 2}). Thus,
as in the proof of
Lemma \ref{lm: congruent isomorphism}\eqref{congruent isomorphism sufficient condition},
the first assertion follows if we show
that $H^1(\Gamma_{\tilde F/F}, \T(\tilde F)_r)
= 0 = H^1(\Gamma_{\tilde F'/F'}, \T'(\tilde F')_r)$.
This is a special case of \cite[Proposition 13.8.1]{KP23}.
The second assertion is immediate.
\end{proof}

\begin{pro} \label{pro: weakly induced tori Chai-Yu isomorphism}
Let $(F, \T) \leftrightarrow_l (F', \T')$,
with $\T$ a weakly induced torus over $F$.
Let $m$ be a positive integer, with 
$m \lleq_{\T} l$.  Then there is a unique Chai-Yu isomorphism
$\mathcal{T}^{\ft} \times_{\mO_F} \mO_F/\p_F^m
\cong \mathcal{T'}^{\ft} \times_{\mO_{F'}} \mO_{F'}/\p_{F'}^m$.
\end{pro}
\begin{proof}
Let $m \lleq_L l$ for some finite Galois extension $L/F$ splitting $\T$, and let
$L \hookrightarrow F^{\sep}$ and $L' \hookrightarrow {F'}^{\sep}$
be compatible embeddings. Form $(F, \R := \Res_{L/F} \T_L) \leftrightarrow_l
(F', \R' := \Res_{L'/F'} \T_{L'})$.

By Lemma \ref{lm: weakly induced Chai-Yu base-change}, there exists
a Chai-Yu isomorphism $\mathcal{R}^{\ft} \times_{\mO_F} \mO_F/\p_F^m
\rightarrow {\mathcal{R}'}^{\ft} \times_{\mO_{F'}} \mO_{F'}/\p_{F'}^m$.
On the other hand, we also know that
$\mcT^{\ft} \hookrightarrow \mathcal{R}^{\ft}$ and
${\mcT'}^{\ft} \hookrightarrow {\mathcal{R}'}^{\ft}$ are
closed immersions, as a special case of
\cite[Lemma B.7.11]{KP23} (this nontrivially uses the fact
that $\T$ and $\T'$ are weakly induced). This allows us to make
sense of the following claim: that the Chai-Yu isomorphism
$\mathcal{R}^{\ft} \times_{\mO_F} \mO_F/\p_F^m
\rightarrow {\mathcal{R}'}^{\ft} \times_{\mO_{F'}} \mO_{F'}/\p_{F'}^m$
restricts to an isomorphism
$\mcT^{\ft} \times_{\mO_F} \mO_F/\p_F^m
\rightarrow {\mcT'}^{\ft} \times_{\mO_{F'}} \mO_{F'}/\p_{F'}^m$
(which we will show to be the desired Chai-Yu isomorphism).

Since the image of $\mathcal{T}^{\ft}(\mO_{\tilde F})$ is schematically dense
in $\mathcal{T}^{\ft} \times_{\mO_F} \mO_F/\p_F^m$
(by \cite[Lemma 8.5.1]{CY01}), this claim follows if we show that
$\mathcal{R}^{\ft} \times_{\mO_F} \mO_F/\p_F^m \rightarrow
{\mathcal{R}'}^{\ft} \times_{\mO_{F'}} \mO_{F'}/\p_{F'}^m$
takes the image of $\mathcal{T}^{\ft}(\mO_{\tilde F}) = \T(\tilde F)_b$
isomorphically onto that
of ${\mathcal{T}'}^{\ft}(\mO_{\tilde F'}) = \T'(\tilde F')_b$.
In other words, if we show that
the map $\R(\tilde F)_b/\R(\tilde F)_m \rightarrow
\R'(\tilde F')_b/\R'(\tilde F')_m$, obtained by evaluating
$\mathcal{R}^{\ft} \times_{\mO_F} \mO_F/\p_F^m \rightarrow
{\mathcal{R}'}^{\ft} \times_{\mO_{F'}} \mO_{F'}/\p_{F'}^m$
at $\mO_{\tilde F}/\p_{\tilde F}^m = \mO_{\tilde F'}/\p_{\tilde F'}^m$,
induces an isomorphism $\T(\tilde F)_b/\T(\tilde F)_m
\rightarrow \T'(\tilde F')_b/\T'(\tilde F')_m$.

But by the definition of a Chai-Yu isomorphism,
this map $\R(\tilde F)_b/\R(\tilde F)_m \rightarrow
\R'(\tilde F')_b/\R'(\tilde F')_m$ is a restriction of a standard isomorphism
(use Proposition \ref{pro: standard isomorphism strictly Henselian}
and Lemmas \ref{lm: standard isomorphism decreasing r}
and \ref{lm: weakly induced}),
and hence restricts to an isomorphism
$\T(\tilde F)_b/\T(\tilde F)_m \rightarrow \T'(\tilde F')_b/\T'(\tilde F')_m$
that is also a restriction of a standard isomorphism
(combine Proposition \ref{pro: standard isomorphism strictly Henselian}
with Lemmas \ref{lm: standard isomorphism decreasing r},
\ref{lm: weakly induced}
and \ref{lm: standard isomorphism functorial}).

This not only proves that $\mathcal{R}^{\ft} \times_{\mO_F} \mO_F/\p_F^m
\rightarrow {\mathcal{R}'}^{\ft} \times_{\mO_{F'}} \mO_{F'}/\p_{F'}^m$
restricts to an isomorphism
$\mcT^{\ft} \times_{\mO_F} \mO_F/\p_F^m
\rightarrow {\mcT'}^{\ft} \times_{\mO_{F'}} \mO_{F'}/\p_{F'}^m$, but
also that the restricted isomorphism, evaluated on
$\mO_{\tilde F}/\p_{\tilde F}^m = \mO_{\tilde F'}/\p_{\tilde F'}^m$,
is the isomorphism
$\T(\tilde F)_b/\T(\tilde F)_m \rightarrow \T'(\tilde F')_b/\T'(\tilde F')_m$
obtained by restricting a standard isomorphism. Thus, by definition
(and Lemma \ref{lm: weakly induced}),
$\mcT^{\ft} \times_{\mO_F} \mO_F/\p_F^m
\rightarrow {\mcT'}^{\ft} \times_{\mO_{F'}} \mO_{F'}/\p_{F'}^m$
is a Chai-Yu isomorphism. Its uniqueness follows
from Lemma \ref{lm: Chai-Yu isomorphism well-defined}.
\end{proof}

\section{Putting things together} \label{sec: putting things together}

\begin{proof}[Proof of Theorem \ref{thm: Ganapathy Kazhdan isomorphism}]
In the setting of \eqref{homomorphisms} of the theorem,
Theorem \ref{thm: CY01}, interpreted using 
Proposition \ref{pro: Chai-Yu isomorphism CY01}, gives us a Chai-Yu isomorphism
$\mcT^{\ft} \times_{\mO_F} \mO_F/\p_F^m \rightarrow
{\mcT'}^{\ft} \times_{\mO_{F'}} \mO_{F'}/\p_{F'}^m$.
Hence Proposition \ref{pro: Chai-Yu congruent} provides us with
a congruent isomorphism $\T(F)/\T(F)_m \rightarrow
\T'(F')/\T'(F')_m$. Note that this automatically also
gives the compatibility with the Chai-Yu isomorphism
(i.e., the commutativity of
the left square of \eqref{eqn: Chai-Yu Kottwitz compatibility}).
The latter assertion of
\eqref{homomorphisms} therefore follows from
Proposition \ref{pro: minimal congruent Chai-Yu}.
Item \eqref{homomorphisms functorial} of the theorem follows
from Lemma \ref{lm: congruent isomorphism}\eqref{congruent isomorphism functorial},
whose extra condition is satisfied by taking
$r$ to be any real number between
$m + \max(h(F, \T_1), h(F, \T_2))$
and $\min(\psi_{L_1/F}(l)/e(L_1/F), \psi_{L_2/F}(l)/e(L_2/F))$, where
$L_1$ and $L_2$ are minimal splitting extensions for $\T_1$ and $\T_2$
(use Lemma \ref{lm: MP intersection}).

As for \eqref{Chai-Yu Kottwitz compatibility} of the theorem, it
remains to prove the compatibility with the Kottwitz homomorphism
(the commutativity of the right-square of \eqref{eqn: Chai-Yu Kottwitz compatibility}).
By definition (Definition \ref{df: standard congruent Chai-Yu}\eqref{congruent isomorphism})
we reduce to a similar assertion for a suitable
$\T(\tilde F)/\T(\tilde F)_m \rightarrow \T'(\tilde F')/\T'(\tilde F')_m$, 
with $\tilde F/F$ a maximal unramified extension, which is
induced by a standard isomorphism
$\T(\tilde F)/\T(\tilde F)^{\naive}_r \rightarrow
\T'(\tilde F')/\T'(\tilde F')^{\naive}_r$.
Thus, the desired compatibility with the Kottwitz homomorphism
follows from the compatibility of the standard isomorphism
with the Kottwitz homomorphism
(Proposition \ref{pro: standard isomorphism Kottwitz homomorphism}).

Now we come to \eqref{LLC compatibility}. Set $h = h(F, \T)$.
The assumption $m + 4 h \lleq_{\T} l$ implies that we have
a congruent isomorphism
$\T(F)/\T(F)_{m + h} \rightarrow \T'(F')/\T'(F')_{m + h}$,
which by Lemma \ref{lm: congruent isomorphism}\eqref{congruent isomorphism s}
induces a standard isomorphism
$\T(F)/\T(F)^{\naive}_{m + h} \rightarrow \T'(F')/\T'(F')^{\naive}_{m + h}$,
and also induces a congruent isomorphism
$\T(F)/\T(F)_m \rightarrow \T'(F')/\T'(F')_m$.
Since $\T(F)^{\naive}_{m + h} \subset \T(F)_m$ and
$\T'(F')^{\naive}_{m + h} \subset \T'(F')_m$
(Remark \ref{rmk: MP intersection}),
$\T(F)/\T(F)^{\naive}_{m + h} \rightarrow \T'(F')/\T'(F')^{\naive}_{m + h}$
induces $\T(F)/\T(F)_m \rightarrow \T'(F')/\T'(F')_m$ as well.
Now \eqref{LLC compatibility} is easy to see
from Proposition \ref{pro: standard isomorphism LLC}, applied
with $m + h$ in place of $r$.

Now we address the weakly induced case.
Proposition \ref{pro: weakly induced tori isomorphism} gives
\eqref{homomorphisms} with $h(F, \T)$ replaced by $0$.
Lemma \ref{lm: standard isomorphism functorial}
then gives \eqref{homomorphisms functorial} with $h(F, \T_1)$ and $h(F, \T_2)$
replaced by $0$. For the compatibility with the Chai-Yu isomorphism,
first note that a Chai-Yu isomorphism
$\mathcal{T}^{\ft} \times_{\mO_F} \mO_F/\p_F^m
\rightarrow {\mathcal{T}'}^{\ft} \times_{\mO_{F'}} \mO_{F'}/\p_{F'}^m$
exists (Proposition \ref{pro: weakly induced tori Chai-Yu isomorphism}).
In this weakly induced case, the middle vertical arrow of
\eqref{eqn: Chai-Yu Kottwitz compatibility} is a standard
isomorphism (Proposition \ref{pro: weakly induced tori isomorphism}), so the
commutativity of the left square of \eqref{eqn: Chai-Yu Kottwitz compatibility}
is automatic from the definition of a Chai-Yu isomorphism.
Since the middle vertical arrow of
\eqref{eqn: Chai-Yu Kottwitz compatibility} is a standard isomorphism,
its compatibility with the Kottwitz homomorphism is
immediate from Proposition \ref{pro: standard isomorphism Kottwitz homomorphism}.
For the same reason, the compatibility with the LLC is
obvious from Proposition \ref{pro: standard isomorphism LLC}.
\end{proof}


\begin{thebibliography}{ABPSA}
\bibitem[Ber93]{Ber93}
Vladimir G. Berkovich. \'{E}tale cohomology for non-Archimedean analytic spaces, \textit{Inst. Hautes \'Etudes Sci. Publ. Math.}, 78:5--161, 1993.
\bibitem[BLR90]{BLR}
Siegfried Bosch, Werner L\"utkebohmert, and Michel Raynaud. \textit{N\'eron models}, volume 21. Springer-Verlag, Berlin, New York, 1990.
\bibitem[CY01]{CY01}
Ching-Li Chai and Jiu-Kang Yu. Congruences of N\'eron models for tori and the Artin conductor. \textit{Ann. of Math. (2)}, 154(2):347--382, 2001. With an appendix by Ehud de Shalit.
\bibitem[Del84]{Del84}
Pierre Deligne. Les corps locaux de caract\'eristique $p$, limites de corps locaux de caract\'eristique $0$. In \textit{Representations of reductive groups over a local field}, Travaux en Cours, pages 119–157. Hermann, Paris, 1984.
\bibitem[Gan15]{Gan15}
Radhika Ganapathy. The local Langlands correspondence for ${\mathrm{GSp}}_4$ over local function fields. \textit{Amer. J. Math.}, 137(6):1441--1534, 2015.
\bibitem[Gan22]{Gan22}
Radhika Ganapathy. A Hecke algebra isomorphism over close local fields. \textit{Pacific J. Math.}, 319(2):307--332, 2022.
\bibitem[Kaz86]{Kaz86}
David Kazhdan. Representations of groups over close local fields. \textit{J. Analyse Math.}, 47:175--179, 1986. 
\bibitem[Kot97]{Kot97}
Robert E. Kottwitz. Isocrystals with additional structure. II. \textit{Compositio Math.}, 109(3):255--339, 1997. 
\bibitem[KP23]{KP23}
Tasho Kaletha and Gopal Prasad. \textit{Bruhat–Tits theory: a new approach}, volume 44. Cambridge University Press, 2023.
\bibitem[MRR23]{MRR23}
Arnaud Mayeux, Timo Richarz, and Matthieu Romagny. N\'eron blowups and low-degree cohomological applications. \textit{Algebraic Geometry}, 10(6):729--753, 2023.
\bibitem[Ser02]{Ser02} 
Jean-Pierre Serre. \textit{Galois cohomology}. Springer Monographs in Mathematics. Springer-Verlag, Berlin, english edition, 2002. Translated from the French by Patrick Ion and revised by the author.
\bibitem[Yu09]{Yu09}
Jiu-Kang Yu. On the local Langlands correspondence for tori. In \textit{Ottawa lectures on admissible representations of reductive $p$-adic groups}, volume 26 of \textit{Fields Inst. Monogr.}, pages 177--183. Amer. Math. Soc., Providence, RI, 2009.
\bibitem[Yu15]{Yu15}
Jiu-Kang Yu. Smooth models associated to concave functions in Bruhat-Tits theory. In \textit{Autour des sch\'emas en groupes. Vol. III}, volume 47 of Panor. Synth\`eses, pages 227--258. Soc. Math. France, Paris, 2015.
\end{thebibliography}
\end{document}